\documentclass[11pt,a4paper]{scrartcl}
\usepackage[english]{babel}
\usepackage{amsmath}
\usepackage{amsfonts}
\usepackage[toc,page]{appendix}
\usepackage[latin1]{inputenc}
\usepackage[T1]{fontenc}
\usepackage{amssymb}
\usepackage[ruled,algosection]{algorithm2e}
\usepackage{color}
\usepackage{cite}
\usepackage{amsthm}
\usepackage{latexsym}
\usepackage{graphicx}
\usepackage[nottoc]{tocbibind}
\usepackage[lofdepth,lotdepth]{subfig}
\usepackage{url}
\usepackage{longtable}
\usepackage[]{geometry}
\newtheorem{theorem}{Theorem}[section]

\newtheorem{proposition}[theorem]{Proposition}

\newcommand{\diam}{\operatorname{diam}} 
 
\newcommand{\divergence}{\operatorname{div}} 
\newcommand{\supp}{\operatorname{supp}} 
\newcommand{\argmin}{\operatorname{argmin}} 
 
\newcommand{\spanlin}{\operatorname{span}}

\title{Approximation of skewed interfaces with tensor-based model reduction procedures: application to the reduced basis hierarchical model reduction approach}
\author{Mario Ohlberger\footnote{Institute for Computational and Applied Mathematics, University of Muenster,
Einsteinstr. 62, 48149 M\"{u}nster, Germany, \, \texttt{mario.ohlberger@uni-muenster.de, kathrin.smetana@wwu.de}} \, and Kathrin Smetana$^{*}$}
\begin{document}
\maketitle

\begin{abstract}
In this article we introduce a procedure, which allows to recover the potentially very good approximation properties of tensor-based model reduction procedures for the solution of partial differential equations in the presence of interfaces or strong gradients in the solution which are skewed with respect to the coordinate axes. The two key ideas are the location of the interface either by solving a lower-dimensional partial differential equation or by using data functions and the subsequent removal of the interface of the solution by choosing the determined interface as the lifting function of the Dirichlet boundary conditions. We demonstrate in numerical experiments 
for linear elliptic equations and the reduced basis-hierarchical model reduction approach that the proposed procedure locates the interface well and yields a significantly improved convergence behavior even in the case when we only consider an approximation of the interface. 
\end{abstract}

\vspace*{\baselineskip}



\section{Introduction}
Fluid flow problems such as subsurface flow or blood flow problems often feature one distinguished (dominant) direction along which the essential dynamics develop. Therefore, tensor-based model reduction procedures such as the proper generalized decomposition (PGD) method, the hierarchical model reduction (HMR), and the reduced basis-hierarchical model reduction (RB-HMR) approach are well suited to compute an efficient and accurate approximation of the full(-dimensional) solution of the underlying partial differential equation (PDE). The common idea of such tensor-based model reduction procedures is to approximate the full solution by a truncated tensor product decomposition of the form 
$
p_{m}(x,y) = \sum_{l=1}^{m} \bar{p}_{l}(x)\phi_{l}(y)
$, where $x,y$ lie in the computational domain $\Omega$ and are associated with different coordinate axes. The resulting model reduction approaches then differ from one another in the way the tensor products $\bar{p}_{l}(x)\phi_{l}(y)$, $l=1,...,m$ are computed. 
 
In the PGD method, introduced in \cite{AmMoChKe06,Nouy07}, the tensor products are determined by iteratively solving the Euler-Lagrange equations associated with the considered problem. Alternatively, and in some cases equivalently, they may be computed as the minimizer of the variational functional corresponding to the considered PDE \cite{LeBLelMad09,CanEhrLel11}. For an overview on the PGD method we refer to \cite{ChiCue14,ChKeLe13}.

In contrast, the HMR approach, introduced in \cite{VogBab81a,VogBab81b, VogBab81c} and studied in a more general 
geometric setting in \cite{ErnPerVen08, PerErnVen10}, considers a reduced space which is a combination of the full (Finite Element) solution space along the dominant (flow) direction with a reduction space spanned by orthonormal basis functions $\{ \phi_{l} \}_{l=1}^{m}$ in the so-called transverse direction. The function $p_{m}$ then solves a reduced problem obtained by a Galerkin projection onto the reduced space. While in \cite{VogBab81a,VogBab81b, VogBab81c,ErnPerVen08, PerErnVen10} the reduction space is chosen a priori as the span of trigonometric or Legendre polynomials, a highly nonlinear approximation is employed for the construction in the RB-HMR approach \cite{OhlSme11,OhlSme14,SmeOhl14,Sme13}. To this end, first a parametrized problem in the transverse direction is derived from the full dimensional problem, where the parameters reflect the influence from the unknown solution in the dominant direction. Then, reduced basis (RB) techniques \cite{PatRoz2006,RoHuPa08} are applied for the efficient construction of the reduction space from snapshots of the parametrized transverse problem, exploiting their good approximation properties \cite{DePeWo12,KahVol07}. Thus, both in the construction of the solution manifold of the parametrized lower-dimensional problem and in the subsequent choice of the basis functions, information on the full solution is included to obtain a fast convergence of the reduced solution to the full one. In general, this yields an improved convergence rate compared to a priori chosen reduction spaces \cite{OhlSme14,Sme13}. 

In spite of their mentioned good performance for say fluid flow problems, the approximation capacity of tensor-based model reduction procedures suffers considerably if the target solution exhibits an interface, i.e. a steep gradient or even a discontinuity, which is skewed with respect to the coordinate axes. Such behavior can often be encountered in fluid flow problems and particularly in subsurface flow, where, depending on the permeability of the soil, the saturation profile may form a skewed interface along the water table. This deteriorated convergence behavior is due to the fact that for a full approximation of the skewed interface the saturation or concentration profile in each point $x$ in the dominant direction has to be included. In this article we introduce a new ansatz to tackle this problem. We propose to first approximately locate the interface by solving a lower-dimensional model or for simple model problems to infer the location of the interface from data functions. We assume that we have Dirichlet data available at the positions where the interface intersects the boundary of the considered computational domain. Thus we can then infer an approximation of the shape of the interface from the known Dirichlet boundary conditions. Otherwise an approximate shape of the interface can be computed in a preprocessing step. Finally, we prescribe the obtained saturation or concentration profile as the lifting function of the Dirichlet boundary conditions. In this way, we hope to remove the part of the full solution, which causes the bad convergence rate from the approximation process and therefore significantly improve the convergence behavior of the employed tensor-based model reduction approach. This will be demonstrated in numerical experiments. 

Alternative to our approach, in \cite{GerLom14} an interface or shock propagating in time is included in a time-dependent basis, which is spanned by the eigenfunctions of a linear Schr\"{o}dinger operator and yields a numerical approximation of a Lax pair. In \cite{LeMMath10} a reduced basis in space is constructed via a proper orthogonal decomposition of snapshots and the evolution of the coefficients in time is computed by a suitable mapping and thus in an equation-free manner. In the case of parametrized PDEs it is well-known that convection dominated evolution equations where shocks may develop are difficult to tackle with RB methods \cite{RoHuPa08,PatRoz2006} if linear spaces are employed. The reason for this is that similar to the setting of the skewed interface considered in this article the solution for nearly every time step has to be included in the basis, which deteriorates the approximation properties of the RB space. Therefore, in \cite{OhlRav13} a nonlinear approximation is applied by employing the method of freezing to decompose the target solution into a shape and group component. Then RB methods are applied to approximate the former while the group component say captures a drift of the interface. Also in \cite{TaPeQu13} the authors propose to employ a nonlinear approximation strategy for the approximation of the solution of parametrized conservation laws in one space dimension. The approach in \cite{TaPeQu13} consists of a partition of the domain induced from a suitable approximation of the shock curve such that the solution in each obtained subdomain is regular. The empirical interpolation method \cite{BMNP04} --- an interpolation strategy from the RB framework --- is used to reconstruct the smooth parts of the solution in the subdomains. 

The remainder of this article is organized as follows. In Section \ref{ansatz_interface} we first describe our approach for the location of the interface using the example of subsurface flow and subsequently outline how the location of the interface can be inferred from data functions for linear advection-diffusion problems (Section \ref{locate_interface}). Afterwards, we demonstrate for linear advection-diffusion problems how the information on the location of the interface can be used to remove the interface from the model reduction procedure in Section \ref{remove_interface}. In Section \ref{1d_gd} we exemplify this ansatz for the RB-HMR method  and present an approach for the derivation of a lower-dimensional parametrized problem particularly suited for the presence of interfaces, which will be validated in Section \ref{numerics_gd}. The capacity of the ansatz proposed in Section \ref{ansatz_interface} to improve the convergence behavior is demonstrated in Section \ref{numerics_gd} for linear problems for the RB-HMR approach in several numerical experiments, including a test case, where we do not include the exact interface but only an approximation.

%\begin{figure}[h!]
%\center
%{\includegraphics[scale=0.08]{plain/groundwater4b-eps-converted-to.pdf} }
%\caption{{\footnotesize Schematic picture of subsurface flow with a skewed water table}\label{fig_groundwater}}
%\end{figure}

\section[An ansatz for approximating skewed interfaces]{An ansatz for approximating skewed interfaces with tensor-based model reduction approaches}\label{ansatz_interface}

Let $\Omega \subset \mathbb{R}^2$ denote the computational domain  with Lipschitz boundary $\partial \Omega$, $\Sigma_{D} \subset \partial \Omega$ the Dirichlet boundary,  and $\Sigma_{N} \subset \partial \Omega$ the Neumann boundary. We require that $\Sigma_{D}$ has positive Hausdorff measure. We assume that $\Omega$ can be considered as a two-dimensional fiber bundle:
\begin{equation*}
\Omega = \underset{x \in \Omega_{1D}}{\bigcup} \{x\} \times \omega_x,
\end{equation*}
where $\Omega_{1D} = (x_0,x_1)$ and $\omega_x$ denotes the transverse fiber associated with $x \in \Omega_{1D}$. Note that the generalization to domains with a more complex geometry is straightforward \cite{Per12}.
We define for any $x \in \Omega_{1D}$ the mapping
$
 \psi(\cdot;x) : \omega_x \rightarrow \hat{\omega}
$
between the fiber $\omega_x$ associated with $x \in \Omega_{1D}$ and a reference fiber $\hat{\omega}$ with $\hat{\omega} = ]y_{0},y_{1}[$. Furthermore, we introduce the mapping $\Psi: \Omega\rightarrow \widehat{\Omega}$, defined as $\hat{y} = \psi(y;x)$ for $y \in \omega_x$ and $\hat{x} = x$ for $x \in \Omega_{1D}$. We require that $\psi(\cdot;x)$ is a $C^{1}$-diffeomorphism and that the geometric transformation $\Psi$ is differentiable with respect to $z=(x,y) \in \Omega$.

\subsection{Locating the interface}\label{locate_interface}

In this subsection we propose several approaches in order to (approximately) locate the (skewed) interface. First, we address our motivating example of saturated-unsaturated subsurface flow, where the interface can be located by solving a reduced model. Although the approach in this paper is mainly intended for situations such as subsurface flow, where skewed interfaces naturally occur, we can also think of some simplified settings in which more heuristic approaches can be used to locate the interface. Therefore, we address in a second step a linear advection-diffusion problem. Here, thanks to the nature of the problem, the interface is induced by the data functions.

\paragraph*{Saturated-unsaturated subsurface flow} Let $\Omega$ be occupied by a homogeneous soil. In the time interval $[0,T]$ we consider the Richards equation (see e.g. \cite{Bear88,BOSS13})
for the water saturation $s: \Omega \times [0,T] \rightarrow [0,1]$ and the water pressure $p: \Omega \times [0,T] \rightarrow \mathbb{R}$
\begin{equation}\label{richards}
 \frac{ds}{dt}- \divergence\left(K \kappa(s)( \nabla p + {\boldsymbol g} )\right) = 0 \quad \text{in $\Omega \times [0,T]$}. 
\end{equation}
Here, $\kappa$ denotes the relative permeability, $K$ the hydraulic conductivity\footnote{Following \cite{Bear88} the hydraulic conductivity describes the ability of the soil to conduct water through it under hydraulic gradients.}, and ${\boldsymbol g}$ the gravity vector. We have normalized the porosity and the viscosity. 
\begin{figure}[t]
\center
{\includegraphics[scale=0.13]{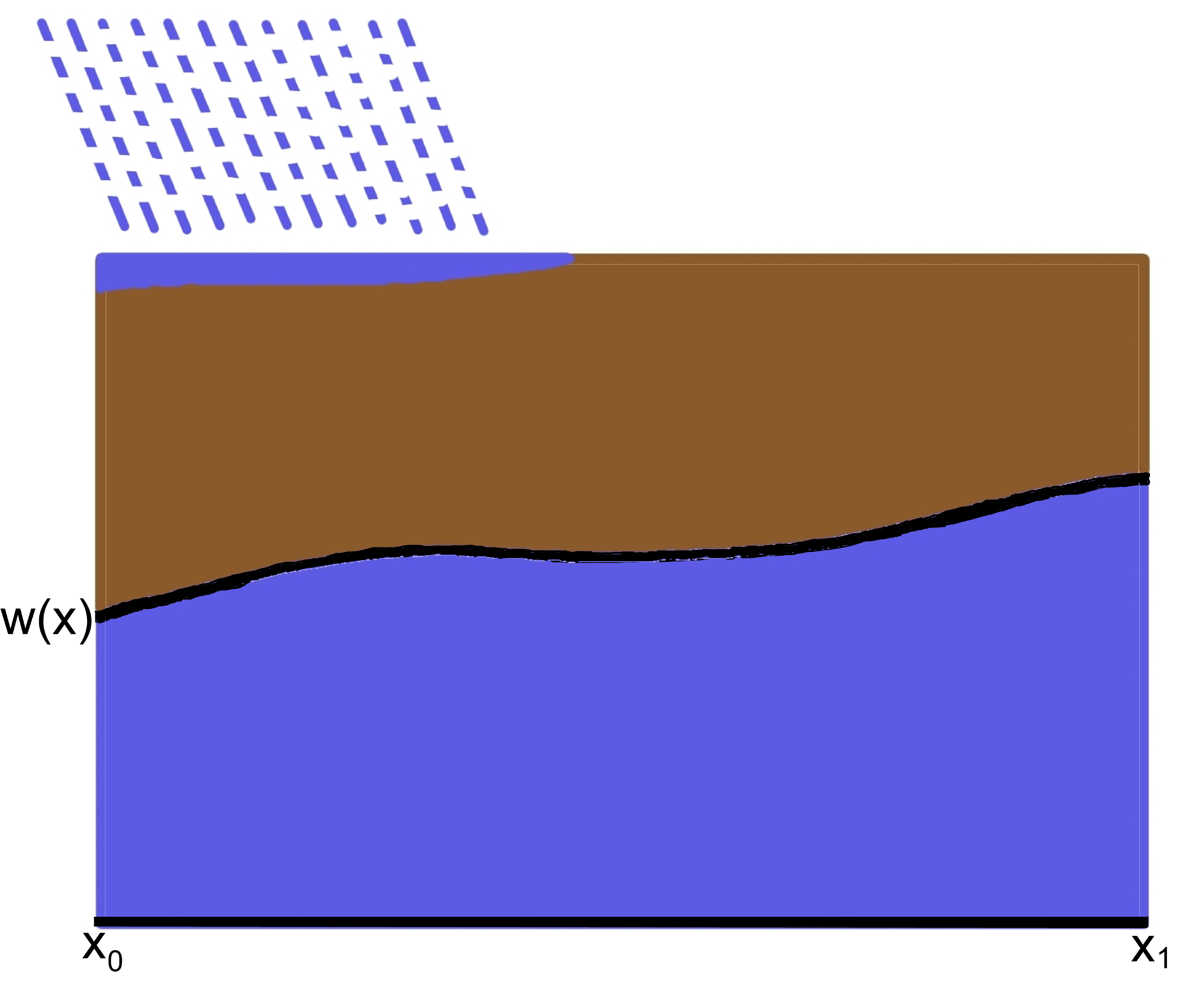} }
\caption{ Schematic picture of subsurface flow with a skewed water table $w: \Omega_{1D} = (x_{0},x_{1}) \rightarrow \mathbb{R}$.\label{fig_groundwater}}
\end{figure}

To locate the water table, we first consider the groundwater flow equation for the piezometric head $\varphi = y + p/(\rho g)$ with a free surface. The latter is characterized by an atmospheric pressure and thus describes the location of the water table. Furthermore, $\rho$ denotes the density of the fluid and $g$ designates  the gravity acceleration. Assuming the incompressibility of water and a flat bottom of the considered domain $\Omega$ at level $y = 0$, the piezometric head or potential can be described by the following PDE \cite{Bear88} 
\begin{align}
\nonumber- K \Delta \varphi &= 0 \quad t \in [0,T] \enspace \text{and} \enspace 0 \leq y \leq w(x),\\
\label{groundwater}\\[-1.5ex]
\nonumber \frac{d\varphi}{dt}  - K\left[ \left(\frac{d\varphi}{dx}\right)^{2} + \left(\frac{d\varphi}{dy}\right)^{2}\right] + \frac{d\varphi}{dy} (K + N) - N &= 0 \quad t \in [0,T] \enspace \text{and} \enspace y = w(x).
\end{align}
Here, $N$ accounts for accretion and suitable initial conditions and additional boundary conditions  are prescribed. Starting from \eqref{groundwater} one can derive a dimensionally reduced model for the height of the water table $w$ by assuming a hydrostatic pressure distribution \cite{Dupuit1863} or by employing an asymptotic expansion \cite{Dagan67}. The PDE for the reduced model then reads 
\begin{equation}\label{Boussinesq}
\frac{dw}{dt} - \frac{K}{2} \Delta w^{2} + N = 0, \quad \text{in} \enspace \Omega_{1D} \times [0,T].
\end{equation}
One possible scenario and the corresponding water table $w$ are depicted in Fig.~\ref{fig_groundwater}.

Note that thanks to the structure of \eqref{Boussinesq}, solving \eqref{Boussinesq} has the same computational complexity as approximating the solution of \eqref{richards} in the tensor space 
\begin{equation*}
V_{1} = \left\{ v(x,y,t) = \bar{v}(x,t)\, \phi(\psi(y;x)),\enspace \mbox{with} \enspace \bar{v}_{1} \in L^{2}([0,T],X)\,, \phi_{1} \in Y \,\right \},
\end{equation*}
where $H^{1}_{0}(\Omega_{1D}) \subseteq X \subseteq H^{1}(\Omega_{1D})$
and $H^{1}_{0}(\widehat{\omega}) \subseteq Y \subseteq H^{1}(\widehat{\omega})$. It is therefore reasonable from a computational perspective to solve \eqref{Boussinesq} in a preprocessing step, if the so gained information accelerates the convergence of the tensor-based model reduction procedure.

Employing the location of the water table or a similar interface described by the solution $w$ of \eqref{Boussinesq}, we can define a function $\mathfrak{h}: \Omega \rightarrow \mathbb{R}$ which describes the corresponding saturation or concentration profile. Here, we use the shape of the interface at the Dirichlet boundary as a shape for the whole interface, where the location of the interface is given by $w$ and a corresponding indicator function $\varphi_{w}: \Omega \rightarrow \mathbb{R}$, which is defined as $\varphi_{w}(x,y) = 1$ for $y < w(x)$ and $\varphi_{w}(x,y) = 0$ for $y > w(x)$. If we do not have suitable boundary data available, we suggest to compute an approximation of the solution of the PDE on a mesh that is coarse along one direction and fine along the other direction to obtain an approximation of the shape of the interface. 

\paragraph*{Linear advection-diffusion equation} Similar to the situation above also for linear elliptic and parabolic problems interfaces occurring in the solution can be directly related to interfaces in the data functions. Here, we consider a linear advection-diffusion problem as parabolic problems can be reduced to the former case via discretizing in time. In detail, we consider the following (simplified) model problem for a global pressure $\tilde{p}: \Omega \rightarrow \mathbb{R}$
\begin{alignat}{2}
\nonumber \nabla\cdot  ( k \nabla \tilde{p}) + \mathbf{b} \cdot \nabla \tilde{p} &= F \quad & \text{in} \enspace \Omega,\\
\label{strong form}\tilde{p}&= g_{D} &\text{on} \enspace \Sigma_{D}, \\
\nonumber k \nabla \tilde{p} &= g_{N} &\text{on} \enspace \Sigma_{D}, 
\end{alignat}
where $k \in L^{\infty}(\Omega)$ with $0< k_{1} \leq k \leq k_{2}$ for constants $k_{1},k_{2} \in \mathbb{R}^{+}$ and $\mathbf{b} = (b_{1},b_{2})^{t} \in [W^{1,\infty}(\Omega)]^{2}$ with $\divergence b \leq 0$ and $F \in L^{2}(\Omega)$. Moreover, $g_{D}$ and $g_{N}$ are given Dirichlet and Neumann boundary conditions, respectively.

First, we consider the cases where either the right hand side $F$ or the diffusion coefficient $k$ exhibit a skewed interface. If we have $\mathbf{b}=0$ then the location of the interface in the solution $\tilde{p}$ equals the one in the respective data functions. In detail the areas where the shape of $\tilde{p}$ changes from nearly flat to a steep slope and vice versa are the areas where $\tilde{p}$ exhibits the highest/smallest curvature. As a consequence if the interface is induced by $F$, and $k$ is constant or varies only moderately, we expect to be able to identify the ``boundary'' of the interface by determining where $F$ has a maximal (or minimal) derivative in $y$-direction (in $x$-direction). Here, it can be inferred either from the shape of the interface at the Dirichlet boundary or from an approximation of the shape computed in a preprocessing step as outlined in the previous paragraph whether one has to search for the maximal or minimal values of the derivative. In case that the interface is induced by $k$, and $F$ is constant or varies only moderately, we can use the smallest \emph{and} largest derivative in $y$-direction of $k$ to deduce (an approximation) of the location of the interface. A possible numerical procedure which determines the location of the interface from $F$ and $k$ in those cases is discussed below. Then, we can use this ``boundary'' of the interface either together with the prescribed Dirichlet data or an approximation of the shape to define a function $\mathfrak{h}: \Omega \rightarrow \mathbb{R}$ that represents an approximation of the interface in the solution $\tilde{p}$. 

In case that the interface in $\tilde{p}$ is induced by the right hand side $F$, we expect that for $\mathbf{b}\neq  0$ we can still exploit the procedure discussed above for $F$ even for rather strong advective fields for the following reason: As we consider steep interfaces we expect that the curvature of $\tilde{p}$ dominates the gradient of $\tilde{p}$ times the advective field, which is why we expect that we can still obtain a good approximation of the location of the interface in $\tilde{p}$ by considering only $F$ and neglecting $\mathbf{b}$. This will be demonstrated in numerical experiments in \S \ref{numerics_gd}. In contrast if $k$ exhibits an interface we expect that unless the advective field is parallel to this interface the advective term $\mathbf{b} \cdot \nabla \tilde{p}$ will dominate the term $\nabla \cdot (k \nabla \tilde{p})$ and as a consequence we will observe that $\tilde{p}$ has (strong) boundary layers and only a rather moderate slope which does not require additional measures to improve convergence of tensor-based model reduction procedures. 

Finally, if we have $k=F=0$ and an ``inflow'' boundary conditions on one part of the Dirichlet boundary then for constant advective fields it is possible to infer the location of the interface and thus $\mathfrak{h}$ in $\tilde{p}$ from $\mathbf{b}$. For one possible example in this context see \cite{DHSW12}.

We close this subsection with the proposal of a numerical procedure that determines the location of the interface in $\tilde{p}$ if this interface is either induced by $F$ or $k$. To this end we introduce a fine partition $\tau_{h}$ of $\hat{\omega}$ with $n_{h}$ elements and a coarse partition $\mathcal{T}_{H'}$ of $\Omega_{1D}$ with $N_{H'}$ elements. Here, it is recommended that the partition of $\hat{\omega}$ has approximately the same number of elements as the mesh that is employed for the FE computations in the model reduction procedure, while, due to computational feasibility, the partition $\mathcal{T}_{H'}$ should be significantly coarser as the mesh of the FE approximation used in the model reduction procedure. Then we evaluate $k$ or $F$ on the partition $\mathcal{T}_{H'}\times\tau_{h}$, determine for each grid point in $x$-direction the elements in $y$-direction with the highest and/or smallest derivative in $y$-direction (or vice versa), and interpolate between the midpoints of those elements to obtain an approximative location of the interface in the solution $\tilde{p}$. Note that by using the coarse partition $\mathcal{T}_{H'}\times\tau_{h}$ and an associated Finite Element (FE) space one can compute a (rough) approximation of the solution of the PDE which can then be used to infer an approximation of the shape of the interface.

\subsection{Removing the interface from the model reduction procedure}\label{remove_interface}

For the sake of clarity we restrict ourselves for the rest of this article to the model problem \eqref{strong form}. The ideas in the nonlinear and time-dependent setting are essentially the same. To remove the skewed interface from the model reduction procedure we propose to prescribe the saturation or concentration profile $\mathfrak{h}$ of the preceding subsection as the lifting function of the Dirichlet boundary conditions.  In detail we define the solution space $V$ such that $H^{1}_{0}(\Omega) \subseteq V \subseteq H^{1}(\Omega)$ and consider the following full problem:  
\begin{equation}\label{fullprob_gd}
\text{Find} \enspace p \in V: \enspace a(p,v) = f(v) - a(\mathfrak{h},v) \quad \forall v \in V,
\end{equation}
where 
\begin{equation*}
a(p,v):= \int_{\Omega} k \nabla p \nabla v \, dx dy + \int_{\Omega} \mathbf{b} \cdot \nabla p v \, dx dy \quad \text{and} \quad f(v) = \int_{\Omega} Fv \, dx dy.
\end{equation*}
The full solution is given as $\tilde{p} = p + \mathfrak{h}$. 

Next, we recall that the spaces $X$ and $Y$ introduced in \S \ref{locate_interface} satisfy 
$H^{1}_{0}(\Omega_{1D}) \subseteq X \subseteq H^{1}(\Omega_{1D})$
and $H^{1}_{0}(\widehat{\omega}) \subseteq Y \subseteq H^{1}(\widehat{\omega})$,
supposing compatibility with the boundary conditions prescribed on $\partial \Omega$. We approximate $p$ by a linear combination of tensor products 
\begin{equation}\label{reduced sol}
p_{m} := \sum_{l = 1}^{m} \bar{p}_{l}(x) \phi_{l}(\psi(y;x)),
\end{equation}
where $\bar{p}_{l} \in X$ and $\phi_{l} \in Y$, $l = 1,...,m$, and define the reduced solution $\tilde{p}_{m}:= p_{m} + \mathfrak{h}$. Depending on the employed tensor-based model reduction method the function $p_{m}$ solves a reduced problem obtained by a Galerkin projection as for instance in the HMR \cite{PerErnVen10} or the RB-HMR approach (cf.~\cite{OhlSme14} and Section \ref{1d_gd}). Alternatively in the PGD method (see for instance \cite{LeBLelMad09}) the lower dimensional functions $\bar{p}_{l}$ and $\phi_{l}(\psi(y;x))$, $l=1,...,m$ either minimize a variational functional\footnote{
Using the PGD method (cf.~\cite{LeBLelMad09}), we may obtain $p_{m}$  by computing sequentially the pair of functions $(\bar{p}_{l},\phi_{l})$, $l=1,...,m$, as the solution of the minimization problem
\begin{equation*}
(\bar{p}_{m},\phi_{m}) = \underset{\bar{p}\in X, \phi \in Y}{\argmin} \mathcal{E}\left(\sum_{l = 1}^{m-1} \bar{p}_{l}(x) \phi_{l}(\psi(y;x)) + \bar{p}(x) \phi(\psi(y;x)) \right),
\end{equation*} 
where $\mathcal{E}: H^{1}(\Omega) \rightarrow \mathbb{R}$ shall be the variational functional whose minimizer $p \in V$ is the unique solution of \eqref{fullprob_gd} 
for a symmetric bilinear form $a(\cdot , \cdot )$.}  or solve the associated Euler-Lagrange equations (see for instance \cite{LeBLelMad09,AmMoChKe06,ChiCue14}).

Let us assume for a moment that $\mathfrak{h}$ describes the exact location of the interface. Let us furthermore assume that the solution $p$ of \eqref{fullprob_gd} for $\mathfrak{h}\equiv 0$, meaning in the case that no interface is present can be approximated exponentially fast by a tensor-based approximation as $p_{m}$ in \eqref{reduced sol}. Thanks to \eqref{fullprob_gd} we hope that we may then also find for any function $\mathfrak{h}$ an approximation $p_{m}$ \eqref{reduced sol} which converges with the same or a slightly deteriorated exponential rate in $m$ to the solution $p$ of \eqref{fullprob_gd}. In that sense we hope to be able to recover a possibly exponential convergence rate of a tensor-based model reduction procedure in the case of a skewed interface, which will be verified in the numerical experiments in \S \ref{numerics_gd}. Note that it depends on the applied tensor-based model reduction approach and the underlying problem whether an exponential rate can be realized or not. 

For future reference we close this section by introducing some notations. We denote by $a_{s}( \cdot , \cdot )$ the symmetric part of $a( \cdot , \cdot )$ and define a $V$-inner product and the induced $V$-norm as
$
( \cdot , \cdot )_{V} := a_{s}( \cdot , \cdot )$ and $\| \cdot \|_{V} := \sqrt{(\cdot , \cdot )_{V}}.
$
Finally, we define the coercivity and the continuity constants of the bilinear form $a(\cdot , \cdot )$ with respect to the $V$-norm as
$
c_{0} := \inf_{v \in V} (a(v,v)/\| v\|_{V}^{2})$ and $c_{1} := \sup_{v \in V} \sup_{w \in V} (a(v,w)/\| v\|_{V} \| w \|_{V}).
$

\section{Exemplification for the RB-HMR approach}\label{1d_gd}

In this section we exemplify the ansatz for the treatment of skewed interfaces proposed in the previous section \ref{ansatz_interface} for the RB-HMR approach introduced in \cite{OhlSme14,OhlSme11}. %Considering problems that exhibit a dominant spatial direction, the idea of RB-HMR is
%to perform a Galerkin projection onto a reduced space, which combines the full solution space in the dominant ($x$-)direction along $\Omega_{1D}$ with a reduction space $Y_{m} \subset Y$ in the transverse direction. The latter is spanned by the modal orthonormal basis functions $\{\phi_{l}\}_{l=1}^{m}$ and constructed by employing a highly nonlinear approximation. To this end first a parametrized problem in the transverse direction is derived
%from the full problem where the parameters reflect the influence from the unknown solution in the dominant direction. Exploiting the good approximation properties of RB methods \cite{DePeWo12,KahVol07}, then a reduction space is constructed to approximate a suitable solution manifold of the parametrized transverse problem from snapshots and a subsequent proper orthogonal decomposition. For an efficient construction of the snapshots adaptive refinement in parameter space similar to \cite{HaaOhl2008b,HaaDihOhl2011} is applied based on a rigorous a posteriori error estimator \cite{OhlSme14}. \\
To obtain a good approximation of solutions exhibiting a skewed interface, we have to eliminate the interface in the solutions of the lower-dimensional problem in the transverse direction. It is therefore crucial to reproduce the balance of the relative terms in the equation of the full problem \eqref{fullprob_gd}. This is difficult to realize using the approach introduced in \cite{OhlSme14} as choosing the evaluation of the unknown part of the solution in $x$-direction as a parameter allows too much variation in the scaling of the respective terms to counterbalance them. Thus, we present in \S \ref{1dproblem_bilFEM} a new approach for the derivation of a lower dimensional problem based on a FE discretization of the full problem and exemplify it for an advection-diffusion equation in \S \ref{example_3}. We also briefly describe the algorithms for the construction of the reduction space introduced in \cite{OhlSme14} and comment on necessary adaptations due to the exchange of the parametrized 1D problem. We begin this section by formulating the reduced problem of the RB-HMR approach for the full problem \eqref{fullprob_gd}.

\subsection{Formulation of the reduced problem}

We assume orthonormality of the set of functions $\{ \phi_{l} \}_{l = 1}^{m}$ with respect to the $L^{2}$-inner product on $\widehat{\omega}$ and define the reduced space
\begin{equation*}
V_{m} = \left\{ v_{m}(x,y) = \underset{k = 1}{\overset{m}{\sum}}\, \overline{v}_{k}(x)\, \phi_{k}(\psi(y;x)),\enspace \mbox{with} \enspace \overline{v}_{k}(x) \in X,\, x \in \Omega_{1D}, \, y \in \omega_{x}\, \right \},
\end{equation*}
where
\begin{equation*}
\overline{v}_{k}(x) = \int_{\widehat{\omega}}  v_{m}(x,\psi^{-1}(\hat{y};x))\,\phi_{k}(\hat{y})\, d\hat{y}, \qquad k = 1,...,m.
\end{equation*}
By using the Galerkin projection we obtain the reduced problem:
\begin{eqnarray}\label{red_prob_gd}
\text{Find} \enspace p_{m} \in V_{m}: \quad a(p_{m}, v_{m})= f(v_{m}) - a(\mathfrak{h}, v_{m}) \enspace \forall\, v_{m} \in V_{m},
\end{eqnarray}
which can be rewritten as: Find $\overline{p}_{k} \in X$, $k=1,\hdots, m$ such that 
\begin{eqnarray*}
\sum_{k = 1}^{m} a(\overline{p}_{k} \phi_{k}, \xi \phi_{l})= f(\xi \phi_{l}) - a(\mathfrak{h},\xi \phi_{l}) \enspace \forall\, \xi \in X \enspace \mbox{and} \enspace l = 1,...,m.
\end{eqnarray*}
To compute an approximation of the coefficient functions $\overline{p}_{k}(x)$, $k = 1,...,m$, we introduce a subdivision $\mathcal{T}_{H}$ of $\Omega_{1D}$ with elements $\mathcal{ T}_{i} = (x_{i-1},x_{i})$ of width $H_{i}= x_{i} - x_{i-1}$ and maximal step size $H := \max_{\mathcal{ T}_{i}}\, H_{i}$. We also introduce a corresponding conforming FE space $X^{H} \subset X$ with $\dim(X^{H}) = N_{H} < \infty$ and basis $\xi_{i}^{H}$, $i=1,...,N_{H}$. Combining $X^{H}$ with the reduction space $Y_{m}:=\spanlin\{\phi_{1},\hdots,\phi_{m}\}$, we define the discrete reduced space
\begin{equation*}
V_{m}^{H} = \left\{ v_{m}^{H}(x,y) = \underset{k = 1}{\overset{m}{\sum}}\, \overline{v}_{k}^{H}(x)\, \phi_{k}(\psi(y;x)),\enspace \mbox{with} \enspace \overline{v}_{k}^{H}(x) \in X^{H},\, x \in \Omega_{1D}, \, y \in \omega_{x}\, \right \},
\end{equation*}
and obtain the discrete reduced problem: Find $\overline{p}_{k}^H \in X^H$, $k = 1,...,m$, such that
\begin{eqnarray}\label{prob_gd}
\sum_{k = 1}^{m} a(\overline{p}_{k}^H \phi_{k}, \xi_{i}^H \phi_{l})= f(\xi_{i}^H \phi_{l}) - a(\mathfrak{h}, \xi_{i}^H \phi_{l}) \enspace \text{for} \enspace i= 1,...,N_{H} \enspace \mbox{and} \enspace l = 1,...,m,
\end{eqnarray}
where the discrete reduced solution is defined as $\tilde{p}_{m}^{H}:=p_{m}^{H} + \mathfrak{h}$ for $p_{m}^{H}(x,y) =  \sum_{k = 1}^{m}\, \overline{p}_{k}^{H}(x)\, \phi_{k}(\psi(y;x))$.

\subsection{Derivation of a parametrized 1D problem in transverse direction}\label{1dproblem_bilFEM}
First, we introduce a subdivision $\tau_{h}$ of $\hat{\omega}$ with elements $\tau_{j} = (\hat{y}_{j-1},\hat{y}_{j})$ of width $h_{j} = \hat{y}_{j} - \hat{y}_{j-1}$ and maximal step size $h:= \max_{\tau_{j}}\, h_{j}$. Furthermore, we introduce an associated conforming FE space $Y^{h} \subset Y$ with $\dim(Y^{h}) =n_{h} < \infty$, and basis $\upsilon^{h}_{j}, \, j=1,...,n_{h}$.
Using the FE spaces $X^{H} := \{ w^{H} \in C^{0}(\Omega_{1D}) \, : \, w^{H}|_{\mathcal{T}_{i}} \in \mathbb{P}_{d}^{1}, \mathcal{T}_{i} \in  \mathcal{T}_{H}\} \subset X$ and $Y^{h} := \{ w^{h} \in C^{0}(\widehat{\omega}) \, : \, w^{h}|_{\tau_{j}} \in \mathbb{P}^{1}_{s}, \tau_{j} \in \tau_{h}\} \subset Y$, where, $\mathbb{P}^{1}_{k}$ denotes the set of polynomials of order $\leq k$ in one variable, we may consider the following reference FE approximation of the full problem \eqref{fullprob_gd}: Find $\mathcal{P}_{i}^{h}\in Y^{h}$, $i = 1,...,N_{H}$, such that 
\begin{equation}\label{truth_prob_gd}
\sum_{i = 1}^{N_{H}} a(\xi_{i}^{H} \mathcal{P}_{i}^{h},\xi_{k}^{H}\upsilon_{j}^{h}) = f(\xi_{k}^{H}\upsilon_{j}^{h}) - a(\mathfrak{h},\xi_{k}^{H}\upsilon_{j}^{h}) \quad k = 1,...,N_{H}, \, j = 1,...,n_{h},
\end{equation}
where 
\begin{equation*}
\mathcal{P}_{i}^{h}(\hat{y}) = \sum_{j = 1}^{n_{h}} p_{i,j} \upsilon^{h}_{j}(\hat{y}), \enspace i = 1,...,N_{H}
\end{equation*}
and we define $p^{H\times h}(x,\hat{y}) := \sum_{i=1}^{N_{H}} \sum_{j=1}^{n_{h}} p_{i,j} \xi_{i}^{H}(x)\upsilon_{j}^{h}(\hat{y})$.

Next, we introduce for an arbitrary integrand $t \in L^{1}(\widehat{\Omega})$ of an integral $I(t) := \int_{\widehat{\omega}} \int_{\Omega_{1D}} t(x,\hat{y})\,  dx d\hat{y}$ the quadrature formula
\begin{equation}\label{quad_formula_gd}
\bar{Q}(t) := \sum_{l = 1}^{\bar{Q}} \alpha_{l} \int_{\widehat{\omega}} \tilde{t}(x_{l}^{q},\hat{y}) \, \, d\hat{y}, \quad
\tilde{t}(x_{l}^{q},\hat{y}) := \underset{\varepsilon \rightarrow 0}{\lim} \frac{1}{|B_{\varepsilon}(x^{q}_{l})|} \int_{B_{\varepsilon}(x^{q}_{l})} t(x,\hat{y})\, dx,
\end{equation}
where $\alpha_{l}$, $l = 1,...,\bar{Q}$ are the weights, and $x^{q}_{l}$, $l = 1,...,\bar{Q}$ are the quadrature points.
Replacing $I(t)$ by $\bar{Q}(t)$ in the bilinear form $a( \cdot , \cdot)$ and the linear form $f(\cdot )$, we obtain the approximations $a^{\bar{q}}(\cdot, \cdot)$ and $f^{\bar{q}}(\cdot)$. The discrete problem with quadrature then reads: Find $\mathcal{P}_{i}^{h}\in Y^{h}$, $i = 1,...,N_{H}$, such that 
\begin{equation}\label{1D_prob_quad_gd}
\sum_{i = 1}^{N_{H}} a^{\bar{q}}(\xi_{i}^{H} \mathcal{P}_{i}^{h},\xi_{k}^{H}\upsilon_{j}^{h}) = f^{\bar{q}}(\xi_{k}^{H}\upsilon_{j}^{h}) -  a^{\bar{q}}(\mathfrak{h},\xi_{k}^{H}\upsilon_{j}^{h})\quad k = 1,...,N_{H}, \, j = 1,...,n_{h}. 
\end{equation}
\begin{figure}[t]
\centering
\includegraphics[scale = 0.4]{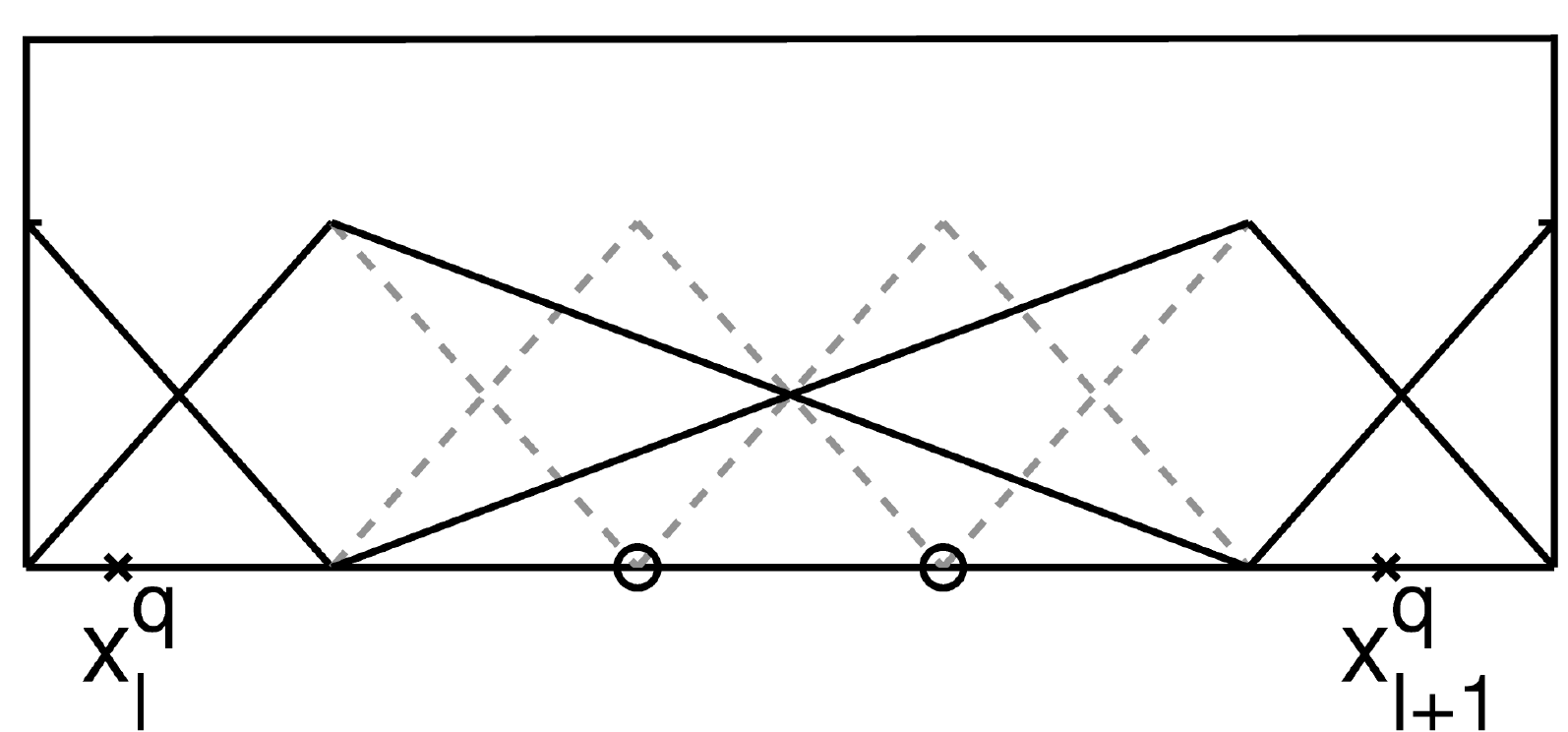}
\caption{Plot of the modified basis functions $\xi^{q}_{i}$ with $\supp (\xi_{i}^{q}) \cap \mu_{l} \neq \emptyset$, $l=1,...,\bar{Q}$, for a polynomial order $d=1$. The 'o' marks the removed nodes of the original triangulation $\mathcal{T}_{H}$. The original basis functions $\xi^{H}_{i}$ are plotted with a dashed line.\label{fig0}}
\end{figure}
Next, we parametrize \eqref{1D_prob_quad_gd} by introducing a parameter vector $\mu$ with entries $\mu_{l} =x^{q}_{l}$, $l = 1,...,\bar{Q}$, in order to find the optimal locations of the quadrature points by applying RB methods and thus to find the optimal points in $\Omega_{1D}$ for solving the lower-dimensional problem in transverse direction. The parameter domain $\mathcal{D}$ is defined as $\mathcal{D} := [\Omega_{1D}]^{\bar{Q}} \subset \mathbb{R}^{\bar{Q}}$. As solving \eqref{1D_prob_quad_gd} is for reasons of efficiency only feasible for small values of $\bar{Q}$, the dimension of $\mathcal{D}$ is limited to $\bar{Q} \ll N_{H}$. Therefore, we have in general $\supp(\xi_{i}^{H}) \cap \supp(\xi_{i'}^{H}) = \emptyset$  for functions $\xi_{i}^{H},\xi_{i'}^{H}  \in \chi^{H} := \{ \xi_{i}^{H} : \supp (\xi_{i}^{H}) \cap \mu_{l} \neq \emptyset, \, l = 1,...,\bar{Q}\}$. 

To introduce a coupling between the respective functions we first replace the functions $\xi_{i}^{H}$, $i = 1,...,N_{H}$, by basis functions $\xi_{i}^{q}$ associated with a new subdivision of $\Omega_{1D}$. The latter is obtained by deleting all nodes of $\mathcal{T}_{H}$ in the open intervals $(\lceil (x_{l}^{q}/H) \rceil H, \lfloor (x_{l+1}^{q}/H) \rfloor H)$, $l = 1,...,\bar{Q}$, as depicted in Fig.~\ref{fig0}. Here and henceforth we assume that the quadrature points are sorted in ascending order and that $x_{0} \geq 0$, where $x_{0}$ has been defined as the left interval boundary of $\Omega_{1D}$. $\lceil \cdot \rceil$ denotes the ceil and $\lfloor \cdot \rfloor$ the floor function. Moreover,  we enhance the set of quadrature points by the points $x_{\bar{Q}+l}^{q} := 0.5 (x_{l}^{q} + x_{l +1}^{q})$, $l = 1,...,\hat{Q}$, if  $\lfloor (x_{l+1}^{q}/H) \rfloor - \lfloor (x_{l}^{q}/H) \rfloor \geq 2$. 
Possible weights of the quadrature $\alpha_{l}$, $l=1,...,Q$ with $ Q = \bar{Q} + \hat{Q}$ are defined as 
\begin{align*}
\alpha_{l}:= \begin{cases} H \qquad & \text{if} \quad \lfloor \frac{x_{l-1}^{q}}{H} \rfloor \neq \lfloor \frac{x_{l}^{q}}{H} \rfloor \enspace \text{and} \enspace \lfloor \frac{x_{l}^{q}}{H} \rfloor \neq \lfloor \frac{x_{l +1}^{q}}{H} \rfloor ,\\[0.5ex]
\frac{x_{l}^{q} + x_{l+1}^{q}}{2} - \lfloor \frac{x_{l}^{q}}{H} \rfloor \quad & \text{if} \quad \lfloor \frac{x_{l-1}^{q}}{H} \rfloor \neq \lfloor \frac{x_{l}^{q}}{H} \rfloor , \\[0.5ex]
\lceil \frac{x_{l}^{q}}{H} \rceil - \frac{x_{l}^{q} + x_{l-1}^{q}}{2} \quad & \text{if} \quad \lfloor \frac{x_{l}^{q}}{H} \rfloor \neq \lfloor \frac{x_{l+1}^{q}}{H} \rfloor , \\[0.5ex]
\frac{x_{l+1}^{q} - x_{l-1}^{q}}{2} \quad & \text{else} .
\end{cases}
\end{align*}
This closes the description of the quadrature rule 
\begin{equation}\label{quad_formula_gd2}
Q(t) := \sum_{l = 1}^{Q} \alpha_{l} \int_{\widehat{\omega}} \tilde{t}(x_{l}^{q},\hat{y}) \, \, d\hat{y}.
\end{equation}
Using the quadrature formula $Q(t)$ \eqref{quad_formula_gd2} instead of $\bar{Q}(t)$ \eqref{quad_formula_gd} we obtain the following coupled system of parametrized 1D partial differential equations in the transverse direction: Given any $\mu \in \mathcal{D}$, find $\mathcal{P}_{i}^{h}(\mu) \in Y^{h}$, $i = 1,\hdots,|\chi^{q}|$, such that
\begin{align}\label{1D_prob_quad_para_gd}
a^{q}(\underset{x_{l}^{q} \in \supp (\xi_{i}^{q} )}{\sum_{\xi_{i}^{q}}} \xi_{i}^{q}(x_{l}^{q}) \mathcal{P}_{i}^{h}(\mu),\upsilon^{h}_{j}\, \xi_{k}^{q} \,;\mu) = f^{q}(\upsilon^{h}_{j}\, \xi_{k}^{q} \,;\mu) - a^{q}(\mathfrak{h},\upsilon^{h}_{j}\, \xi_{k}^{q} \,;\mu),
\end{align}
%\begin{align}
%\nonumber \text{Given any} \enspace \mu \in \mathcal{D}, \enspace 
%\text{find} \enspace \mathcal{P}^{h}_{l}(\mu_{l}) &:= \underset{\mu_{l} \in \supp (\xi_{i}^{q} )}{\sum_{\xi_{i}^{q}}} \xi_{i}^{q}(\mu_{l}) \mathcal{P}_{i}^{h} \in Y^{h},  l = 1,...,\bar{Q}, \\[-3ex]
%\label{1D_prob_quad_para_gd}\\
%\nonumber \text{such that} \enspace a^{q}(\mathcal{P}^{h}_{l}(\mu_{l}),\upsilon^{h}_{j}\, \xi_{k}^{q} \,;\mu) = f^{q}(\upsilon^{h}_{j}\, &\xi_{k}^{q} \,;\mu) - a^{q}(\mathfrak{h},\upsilon^{h}_{j}\, \xi_{k}^{q} \,;\mu) \enspace \text{for} \enspace j = 1,...,n_{h},\, \xi_{k}^{q} \in \chi^{q},
%\end{align}
for $j = 1,...,n_{h},\, \xi_{k}^{q} \in \chi^{q}$, where $\chi^{q} := \{ \xi_{i}^{q} : \supp (\xi_{i}^{q}) \cap \mu_{l} \neq \emptyset, \, l = 1,...,\bar{Q}\}$ and $x_{l}^{q}$, $l=1,\hdots,Q$ are the quadrature points of the quadrature formula $Q(t)$. Note that \eqref{1D_prob_quad_para_gd} is a coupled system of size $\leq 2d \bar{Q} n_{h} \times 2d \bar{Q} n_{h}$, where $d$ has been defined in the beginning of this subsection as the polynomial order of the FE space $X^{H}$. We emphasize that in contrast to \cite{OhlSme14} we are solving in \eqref{1D_prob_quad_para_gd} for the unknown parts of the solution in the dominant direction via the coefficient functions $\mathcal{P}_{i}^{h}(\hat{y};\mu)$ and do not consider them as part of the parameter.

Note that solving \eqref{1D_prob_quad_para_gd} without an artificial coupling is equivalent to solving $\bar{Q}$ coupled systems of size $2d n_{h} \times 2d n_{h}$. This may lead to rather limited variations in the solutions and hence the solution manifold, in which case the solution manifold does not contain all essential information of the full solution in transverse direction. We would hence expect a poor convergence behavior also for $\bar{Q}=1$ which is confirmed by the numerical experiments. In contrast, for $\bar{Q} \geq 2$  and the artificial coupling suggested above we include global information from the dominant direction and hence expect a very much improved convergence behavior, which is again confirmed by the numerical experiments. This stresses the importance of introducing an artificial coupling. 

Other choices of an artificial coupling are of course also possible. For instance one could only delete the nodes of $\mathcal{T}_{H}$ in the open intervals $((\lceil (x_{l}^{q}/H) \rceil + 1)H, (\lfloor (x_{l+1}^{q}/H) \rfloor -1 )H)$, $l = 1,...,\hat{Q}$ and thus keep the two nodes marked with a circle in Fig.~\ref{fig0}.  Then one could preserve the FE basis functions associated with the nodes  $\lceil (x_{l}^{q}/H) \rceil H$ and $\lfloor (x_{l+1}^{q}/H) \rfloor H$, and add an additional FE basis function(s), which couples the former. However, due to this additional FE basis function the size of system \eqref{1D_prob_quad_para_gd} would increase significantly --- at most by $\hat{Q}$. As it is in general only computationally feasible to solve system \eqref{1D_prob_quad_para_gd} if it is of small size, this (further) limits the number of quadrature points $\bar{Q}$ that can be chosen via RB methods. Since choosing as many quadrature points as possible via RB methods seems preferable as this yields nearly optimal chosen quadrature points, we suggest using the artificial coupling described above as in this case the size of system \eqref{1D_prob_quad_para_gd} only depends on $\bar{Q}$. How adding more quadrature points manually without adding more basis functions improves the approximation behavior of the proposed method is subject of future research. Note however that as soon as the quadrature rule is exact no further improvement can be realized without adding new basis functions.

\subsection{Example: An advection-diffusion problem}\label{example_3}
We exemplify the derivation of the coupled system of parametrized 1D partial differential equations for the model problem \eqref{fullprob_gd} with non-homogeneous Dirichlet boundary conditions on $\partial \Omega$. For the sake of clarity we restrict our exposition to a rectangular domain $\Omega$, implicating $\Omega = \widehat{\Omega}$ and $y = \hat{y}$. 
The full space $V$ thus coincides with $H^{1}_{0}(\Omega)$ and the spaces $X$ and $Y$ coincide with $H^{1}_{0}(\Omega_{1D})$ and $H^{1}_{0}(\widehat{\omega})$, respectively. 

By applying the quadrature formula defined in \eqref{quad_formula_gd}, we obtain the discrete problem with quadrature: Find $\mathcal{P}_{i}^{h} \in Y^{h}$, $i = 1,...,N_{H}$, such that 
\begin{subequations}
\begin{align*}
\nonumber & \qquad \qquad \sum_{i=1}^{N_{H}} \int_{\widehat{\omega}}\hspace{-2.5pt} \bar{\mathcal{A}}_{i,k}(y)\frac{d \mathcal{P}_{i}^{h}}{dy} \frac{d \upsilon_{j}^{h}}{dy} + \bar{\mathcal{B}}_{i,k}(y)\frac{d \mathcal{P}_{i}^{h}}{dy} \upsilon_{j}^{h} + \bar{\mathcal{C}}_{i,k}(y)\,\mathcal{P}_{i}^{h}\,\upsilon_{j}^{h}\, dy   \\
\nonumber &= \int_{\widehat{\omega}} \bar{\mathcal{F}}_{k}(y)\,\upsilon_{j}^{h}\, dy -  \int_{\widehat{\omega}} \bar{\mathcal{H}}_{1,k}(y)\,\frac{d \upsilon_{j}^{h}}{dy} + \bar{\mathcal{H}}_{2,k}(y)\,\upsilon_{j}^{h}\, dy \quad \text{for} \enspace j = 1,...,n_{h}, k = 1,...,N_{H}, \qquad \quad
\end{align*}
where the coefficients $\bar{\mathcal{A}}_{i,k}(y), \bar{\mathcal{B}}_{i,k}(y), \bar{\mathcal{C}}_{i,k}(y), \bar{\mathcal{F}}_{k}(y), \bar{\mathcal{H}}_{1,k}(y)$ and $\bar{\mathcal{H}}_{2,k}(y)$ are given by 
\begin{eqnarray*}
\nonumber \bar{\mathcal{A}}_{i,k}(y) &=&  \underset{l = 1}{\overset{\bar{Q}}{\sum}} \alpha_{l} k(x_{l}^{q},y) \xi_{i}^{H}(x_{l}^{q})\xi_{k}^{H}(x_{l}^{q}) ,\qquad
 \bar{\mathcal{B}}_{i,k}(y) = \underset{l = 1}{\overset{\bar{Q}}{\sum}} \alpha_{l} b_{2}(x_{l}^{q},y) \xi_{i}^{H}(x_{l}^{q})\xi_{k}^{H}(x_{l}^{q}), \\ 
\nonumber \bar{\mathcal{C}}_{i,k}(y) &=& \underset{l = 1}{\overset{\bar{Q}}{\sum}} \alpha_{l} k(x_{l}^{q},y) \partial_{x} \xi_{i}^{H}(x_{l}^{q}) \partial_{x} \xi_{k}^{H}(x_{l}^{q}) + b_{1}(x_{l}^{q},y) \partial_{x} \xi_{i}^{H}(x_{l}^{q})\xi_{k}^{H}(x_{l}^{q}), 
\end{eqnarray*}
\begin{eqnarray*}
\nonumber \bar{\mathcal{F}}_{k}(y) &=& \sum_{l = 1}^{\bar{Q}}\alpha_{l} F(x_{l}^{q},y) \xi_{k}^{H}(x_{l}^{q}), \qquad \bar{\mathcal{H}}_{1,k}(y) = \underset{l = 1}{\overset{\bar{Q}}{\sum}} \alpha_{l} k(x_{l}^{q},y)\partial_{y}\mathfrak{h}(x_{l}^{q},y) \xi_{k}^{H}(x_{l}^{q}), \\
\nonumber\bar{\mathcal{H}}_{2,k}(y) &=& \underset{l = 1}{\overset{\bar{Q}}{\sum}} \alpha_{l} k(x_{l}^{q},y)\partial_{x}\mathfrak{h}(x_{l}^{q},y) \partial_{x}\xi_{k}^{H}(x_{l}^{q}) 
+ (b_{1}(x_{l}^{q},y)\partial_{x}\mathfrak{h}(x_{l}^{q},y) + b_{2}(x_{l}^{q},y)\partial_{y}\mathfrak{h}(x_{l}^{q},y)) \xi_{k}^{H}(x_{l}^{q}).
\end{eqnarray*}
\end{subequations}
Here we have omitted the $\backsim$ on the integrands (cf. \eqref{quad_formula_gd}) to simplify notations. Using the artificial coupling introduced in the previous subsection and the associated quadrature formula \eqref{quad_formula_gd2} we obtain the parametrized coupled 1D PDE in transverse direction: Given any $\mu \in \mathcal{D}$, find $\mathcal{P}^{h}_{i}(\mu) \in Y^{h}$, $i=1,\hdots,|\chi^{q}|$ such that
\begin{subequations}
\begin{align*}
\nonumber &\sum_{l=1}^{Q} \Bigg [ \int_{\widehat{\omega}}\hspace{-2.5pt} \mathcal{A}_{k}^{l}(y;\mu) \Bigg (\underset{x_{l}^{q} \in \supp (\xi_{i}^{q} )}{\sum_{\xi_{i}^{q}}} \xi_{i}^{q}(x_{l}^{q}) \frac{d \mathcal{P}_{i}^{h}(\mu)}{dy}\Bigg) \frac{d \upsilon_{j}^{h}}{dy} + \mathcal{B}_{k}^{l}(y;\mu)\Bigg(\underset{x_{l}^{q} \in \supp (\xi_{i}^{q} )}{\sum_{\xi_{i}^{q}}} \xi_{i}^{q}(x_{l}^{q}) \frac{d \mathcal{P}_{i}^{h}(\mu)}{dy}\Bigg)  \upsilon_{j}^{h}\\
& \qquad\qquad + \mathcal{C}_{k}^{l}(y;\mu)\,\Bigg(\underset{x_{l}^{q} \in \supp (\xi_{i}^{q} )}{\sum_{\xi_{i}^{q}}} \partial_{x}\xi_{i}^{q}(x_{l}^{q}) \mathcal{P}_{i}^{h}(\mu)\Bigg) \,\upsilon_{j}^{h}\, dy  \Bigg ]  \\
&= \int_{\widehat{\omega}} \mathcal{F}_{k}(y;\mu)\,\upsilon_{j}^{h}\, dy  -  \int_{\widehat{\omega}} \mathcal{H}_{1,k}(y;\mu)\,\frac{d \upsilon_{j}^{h}}{dy} + \mathcal{H}_{2,k}(y;\mu)\,\upsilon_{j}^{h}\, dy \qquad \text{for} \enspace j = 1,...,n_{h},
\end{align*}
where for all $\xi_{k}^{q} \in \chi^{q}$ the coefficients $\mathcal{A}_{k}^{l}(y;\mu), \mathcal{B}_{k}^{l}(y;\mu), \mathcal{C}_{k}^{l}(y;\mu), \mathcal{F}_{k}(y;\mu), \mathcal{H}_{1,k}(y;\mu)$ and $\mathcal{H}_{2,k}(y;\mu)$ are given by 
\begin{eqnarray*}
\nonumber \mathcal{A}_{k}^{l}(y;\mu)&=&  \alpha_{l} k(x_{l}^{q},y) \xi_{k}^{q}(x_{l}^{q}) ,\qquad
 \mathcal{B}_{k}^{l}(y;\mu) = \alpha_{l} b_{2}(x_{l}^{q},y) \xi_{k}^{q}(x_{l}^{q}), \\ [1.5ex]
\nonumber \mathcal{C}_{k}^{l}(y;\mu)&=& \alpha_{l} k(x_{l}^{q},y) \partial_{x} \xi_{k}^{q}(x_{l}^{q}) + b_{1}(x_{l}^{q},y) \xi_{k}^{q}(x_{l}^{q}), \\
\nonumber \mathcal{F}_{k}(y;\mu) &=& \underset{l = 1}{\overset{Q
}{\sum}} \alpha_{l} F(x_{l}^{q},y) \xi_{k}^{q}(x_{l}^{q}), \qquad \mathcal{H}_{1,k}(y;\mu) = \underset{l = 1}{\overset{Q}{\sum}} \alpha_{l} k(x_{l}^{q},y)\partial_{y}\mathfrak{h}(x_{l}^{q},y) \xi_{k}^{q}(x_{l}^{q}), \\
\nonumber \mathcal{H}_{2,k}(y;\mu)&=& \underset{l = 1}{\overset{Q}{\sum}} \alpha_{l} k(x_{l}^{q},y)\partial_{x}\mathfrak{h}(x_{l}^{q},y) \partial_{x}\xi_{k}^{q}(x_{l}^{q}) 
+ (b_{1}(x_{l}^{q},y)\partial_{x}\mathfrak{h}(x_{l}^{q},y) + b_{2}(x_{l}^{q},y)\partial_{y}\mathfrak{h}(x_{l}^{q},y)) \xi_{k}^{q}(x_{l}^{q}).
\end{eqnarray*}
\end{subequations}

\subsection{Reduced basis generation --- the \textsc{Adaptive-RB-HMR} algorithm}\label{adapt-RB-HMR}
In this subsection we briefly summarize the \textsc{Adaptive-RB-HMR} algorithm introduced in \cite{OhlSme14} which
constructs the reduction space $Y_{m}=\spanlin \{\phi_{1}, \dots, \phi_{m}\} \subset Y^{h}$ using RB sampling techniques and comment on necessary modifications due to the different parametrized 1D problem. 

\begin{algorithm}[t]
\caption{Adaptive training set extension and snapshot generation }\label{adapt-para}
\textsc{AdaptiveTrainExtension}$(G_{0},\Xi_{G_{0}},m_{\mbox{{\scriptsize{max}}}},i_{max},n_{\Xi},\theta,\sigma_{thres},N_{H'})$\\
\textbf{Initialize}  $G = G_{0}, \Xi_{G} = \Xi_{G_{0}}, \phi_{0} = \emptyset, \rho_{0}(G) = 0$\\
\For{$m=1:m_{\mbox{{\scriptsize max}}}$}{
Compute $\mathcal{ P}_{G}^{h}$\\
$[\eta(G), \sigma(G)]= $ \textsc{ElementIndicators}$(\{\phi_{k}\}^{m-1}_{k=1},\mathcal{ P}_{G}^{h},G,\rho(G),N_{H'})$\\
\For{i = 1:$i_{max}$}{
$\mathcal{ G} :=$ \textsc{Mark}$(\eta(G),\sigma(G),\theta,\sigma_{thres})$\\
$(G,\Xi_{G}):=$ \textsc{Refine}$(\mathcal{ G},\Xi_{\mathcal{ G}},n_{\Xi})$\\
$\rho(G\setminus\mathcal{ G}) = \rho(G\setminus\mathcal{ G}) + 1$\\
Compute $\mathcal{ P}_{\mathcal{ G}}^{h}$\\
$[\eta(\mathcal{ G}),\rho(\mathcal{ G}),\sigma(\mathcal{ G})] = $ \textsc{ElementIndicators}$(\{\phi_{k}\}^{m-1}_{k=1},\mathcal{ P}_{\mathcal{ G}}^{h},N_{H'})$ }
 $\{\phi_{k}\}^{m}_{k=1}:= \mbox{POD}(\mathcal{ P}_{G}^{h},m)$\\}
 \Return $\mathcal{ P}_{G}^{h}, \Xi_{G}$
\end{algorithm}
First, the discrete snapshot set
\begin{equation}\label{disc_mani}
\mathcal{M}_{\Xi} := \lbrace \mathcal{P}^{h}(\mu)\, |\, \mu \in  \Xi \rbrace \subset \mathcal{M}, \quad \text{for} \quad \mathcal{M} := \lbrace \mathcal{P}^{h}(\mu)\, |\, \mu \in \mathcal{D} \rbrace,
\end{equation}
and a discrete training set $\Xi \subset \mathcal{D}$ is efficiently constructed in Algorithm \ref{adapt-para} by an adaptive training set extension similar to the one considered in \cite{HaaOhl2008b,HaaDihOhl2011}: Let $G \subset \mathcal{D}$ denote a hyper-rectangular possibly non-conforming grid, $g$ a cell of $G$ and $N_{G}$ the number of cells in $G$. We assume that the parameter values in the training set $\Xi_{g}$ are sampled from the uniform distribution over the cell $g$, where the sample size $n_{\Xi}$ of $\Xi_{g}$ shall be identical for all cells $g$ and $\Xi_{G} = \cup_{g \in G} \Xi_{g}$. We apply a
$\mbox{SOLVE} \rightarrow \mbox{ESTIMATE} \rightarrow \mbox{MARK} \rightarrow \mbox{REFINE}$
strategy to adaptively refine $G$ and construct $\Xi_{G}$ beginning with a given coarse mesh $G_{0}$ and an associated initial training set $\Xi_{G_{0}}$. To estimate the error between the discrete RB-HMR solution $p_{m}^{H}$ of \eqref{red_prob_gd} and the full dimensional reference solution $p^{H\times h}$ of \eqref{truth_prob_gd} we employ the error estimator $\Delta_{m}$ proposed in \cite{OhlSme14} and recalled in the next subsection \S \ref{a posteriori}. We may then define cell indicators  $\eta(g) := \min_{\mu \in \Xi_{g}} \Delta_{m}(\mu)$ and $\sigma(g) := \diam(g)\cdot \rho(g)$, where $\rho(g)$ counts the number of loops in which the cell $g$ has not been refined, since its last refinement. We mark for fixed $\theta \in (0,1]$ in each iteration the $\theta N_{G}$ cells $g$ with the smallest indicators $\eta(g)$ and additionally the cells for which $\sigma(g)$ lies above a certain threshold $\sigma_{thres}$. Afterwards all cells marked for refinement are bisected in each direction. Note that in actual practice we use a coarser space $X^{H'}$ of dimension $N_{H'}\ll N_{H}$ in the dominant direction for the computation of the error indicator $\eta(g)$ \cite{OhlSme14}. 

Subsequently, we define the reduction space $Y_{m}$ as the principal components of $\mathcal{M}_{\Xi_{G}}$ determined by a POD in Algorithm \ref{adapt-HMR-POD} \textsc{Adaptive-RB-HMR}.\\

We emphasize that in contrast to  \cite{OhlSme14} we obtain $\bar{Q}$ snapshots $\mathcal{P}^{h}(\mu)$ per parameter vector $\mu = (\mu_{1},...,\mu_{\bar{Q}})$ --- one for each component. As a consequence one has to slightly modify Algorithm \ref{adapt-para} \textsc{AdaptiveParameterRefinement} and Algorithm \ref{adapt-HMR-POD} \textsc{Adaptive-RB-HMR} introduced in \cite{OhlSme14}.  First, the error indicators $\eta(G)$ and $\sigma(G)$ in Algorithm \ref{adapt-para} have to be computed $\bar{Q} n_{\Xi,G}$ times, where $n_{\Xi,G}$ denotes the sample size of $\Xi_{G}$. However, in return the training set can be reduced significantly. The \textsc{Mark} and \textsc{Refine} strategies are maintained. 
For the application of the POD in Algorithm \ref{adapt-HMR-POD}  $\mathcal{O}(\bar{Q}^{2}n_{\Xi,G}^{2}n_{h})$ operations for the assembling of the correlation matrix and $\mathcal{O}(\bar{Q}^{3}n_{\Xi,G}^{3})$ operations for the solution of the eigenvalue problem are required. The higher costs for the computation of the snapshots $\mathcal{P}^{h}(\mu)$ are still dominated by the costs for the computation of the error estimator. Overall, we do not expect a significant effect on the computational costs of the RB-HMR approach by changing the parametrized 1D problem due to the trade-off between the factor $\bar{Q}$ and the smaller sample size $n_{\Xi,G}$.

\begin{algorithm}[t]
\caption{Construction of the reduction space $Y^{h}_{m}$}\label{adapt-HMR-POD}
\textsc{Adaptive-HMR-RB}$(G_{0},m_{\mbox{{\scriptsize{max}}}},i_{max},n_{\Xi},\theta,\sigma_{thres},N_{H'},\varepsilon_{\mbox{{\scriptsize tol}}})$\\
\textbf{Initialize} $\Xi_{G_{0}}$\\
$[\mathcal{ P}_{G}^{h} , \Xi_{G}]=$ \textsc{AdaptiveTrainExtension}$(G_{0},\Xi_{G_{0}}, m_{\mbox{{\scriptsize{max}}}},i_{max},n_{\Xi},\theta,\sigma_{thres},N_{H'})$\\
$Y_{m}:= \mbox{POD}(\mathcal{ P}_{G}^{h},\varepsilon_{\mbox{{\scriptsize tol}}})$, such that $e^{\mbox{{\tiny POD}}}_{m}\leq\varepsilon_{\mbox{{\scriptsize tol}}}$.\\
\Return $Y_{m}$
\end{algorithm}

\subsection{A posteriori error estimation}\label{a posteriori}
For the sake of completeness we recall in this subsection the a posteriori error estimator $\Delta_{m}$ introduced in \cite{OhlSme14}. To this end we introduce a partition $\widehat{T} := \mathcal{ T}_{H} \times \tau_{h}$ of $\widehat{\Omega}$ with elements $T_{i,j} := \mathcal{ T}_{i} \times \tau_{j}$, where $ \mathcal{ T}_{i} \in \mathcal{ T}_{H}$ and $\tau_{j} \in \tau_{h}$. Moreover, we define the conforming Tensor Product FE space 
\begin{equation*}
V^{H\times h} := \left\{ v^{H \times h} \in C^{0}(\widehat{\Omega}) \, \mid \, v^{H \times h}|_{T_{i,j}} \in \mathbb{Q}_{k,l}, T_{i,j} \in \widehat{T} \right\} \subset V,
\end{equation*}
where $
\mathbb{Q}_{k,l}:= \{ \sum_{j} c_{j} v_{j}(x)w_{j}(\hat{y}) \enspace : \enspace v_{j} \in \mathbb{P}^{1}_{k}, w_{j} \in \mathbb{P}^{1}_{l}  \}.
$ 
We may then introduce a Riesz representative $\mathcal{ R}_{m}^{H \times h} \in V^{H\times h}$ as the solution of 
$
(\mathcal{ R}_{m}^{H \times h}, v^{H \times h})_{V} = f(v^{H \times h}) - a(p_{m}^{H},v^{H \times h})$ $\forall v^{H \times h} \in V^{H \times h},
$
where the $V$-inner product has been defined in \S \ref{remove_interface}.
\begin{proposition}[A posteriori error bound]
The error estimator $\Delta_{m}$ defined as
\begin{eqnarray*}
\Delta_{m} &:=& \|\mathcal{ R}_{m}^{H\times h}\|_{V} / {c_{0}} \\
\text{satisfies} \qquad \| p^{H\times h} - p_{m}^{H} \|_{V} &\leq & \Delta_{m} \leq \frac{c_{1}}{c_{0}}\| p^{H\times h} - p_{m}^{H} \|_{V},
\end{eqnarray*}
where $c_{0}$ and $c_{1}$ have been defined in \S \ref{remove_interface}.
\end{proposition}
\begin{proof}
We refer to the RB literature for the proof of this standard result (see e.g. \cite{RoHuPa08}).
\end{proof}

\section{Numerical Experiments}\label{numerics_gd}

In this section we demonstrate in several numerical experiments the capacity of the ansatz proposed in \S \ref{ansatz_interface} to improve the convergence behavior for tensor-based model reduction approaches using the example of the RB-HMR method. 
First, in \S \ref{subsect:locate interface}, we demonstrate that by applying the procedure suggested in \S \ref{locate_interface} we are indeed able to approximate the location  of the interface very well. Subsequently in \S \ref{subsect:remove interface} we compare the convergence behavior of the model error of the RB-HMR approach in case we include information on the interface with the approximation behavior if we use an arbitrary lifting function of the Dirichlet boundary conditions for three different test cases. In the first test case, the full solution $p$ of \eqref{fullprob_gd} is chosen as a multiple of the solution of test case 1 in \cite{OhlSme14} to have a benchmark for the convergence rate. We see a considerable improvement of the convergence behavior of the model error. An even more substantial improvement can be observed for the second test case, where the solution exhibits more complex structures and little spatial regularity due to a discontinuous source term. While we include in test case 1 and 2 the exact interface, we use in test case 3 only an approximation of the interface as the lifting function of the Dirichlet boundary conditions and still observe a significant improvement of the convergence behavior of the RB-HMR approach. Unless otherwise stated we have used $\bar{Q} = 2$ in \eqref{1D_prob_quad_para_gd} in the numerical tests. Moreover, in all three test cases we have used linear FE in $x$- and $y$-direction, that means
$ 	
X^{H}= \left \{ v^{H} \in C^{0}(\Omega_{1D}) \,:\, v^{H}|_{\mathcal{T}_{i}} \in \mathbb{P}^{1}_{1}(\mathcal{T}_{i}), \mathcal{T}_{i} \in \mathcal{T}_{H}\right \}, 
Y^{h} = \left \{ v^{h} \in C^{0}(\widehat{\omega}) \,:\, v^{h}|_{\tau_{j}} \in \mathbb{P}^{1}_{1}(\tau_{j}), \tau_{j} \in \tau_{h}\right \},
$
and
$
V^{H\times h} =  \{ v^{H \times h} \in C^{0}(\widehat{\Omega}) \, : \, v^{H \times h}|_{T_{i,j}} \in \mathbb{Q}_{1,1}, T_{i,j} \in \hat{T} \}.
$
We define the relative model error in the $V$- or $L^{2}$-norm as $\|e_{m}\|_{V}^{rel} := \|e_{m}\|_{V}/\|p^{H \times h}\|_{V}$, or $\|e_{m}\|_{L^{2}(\Omega)}^{rel} := \|e_{m}\|_{L^{2}(\Omega)}/\|p^{H \times h}\|_{L^{2}(\Omega)}$, respectively, for $e_{m} = p^{H \times h} - p^{H}_{m}$. Finally, we have used equidistant grids in $x$- and $y$-direction for all computations in this section.
\begin{figure}[h!]
\center
\subfloat[{\footnotesize $\mathbf{b}=(0,0)^{T}$, $\partial_{y} F$} ]{\includegraphics[scale=0.3]{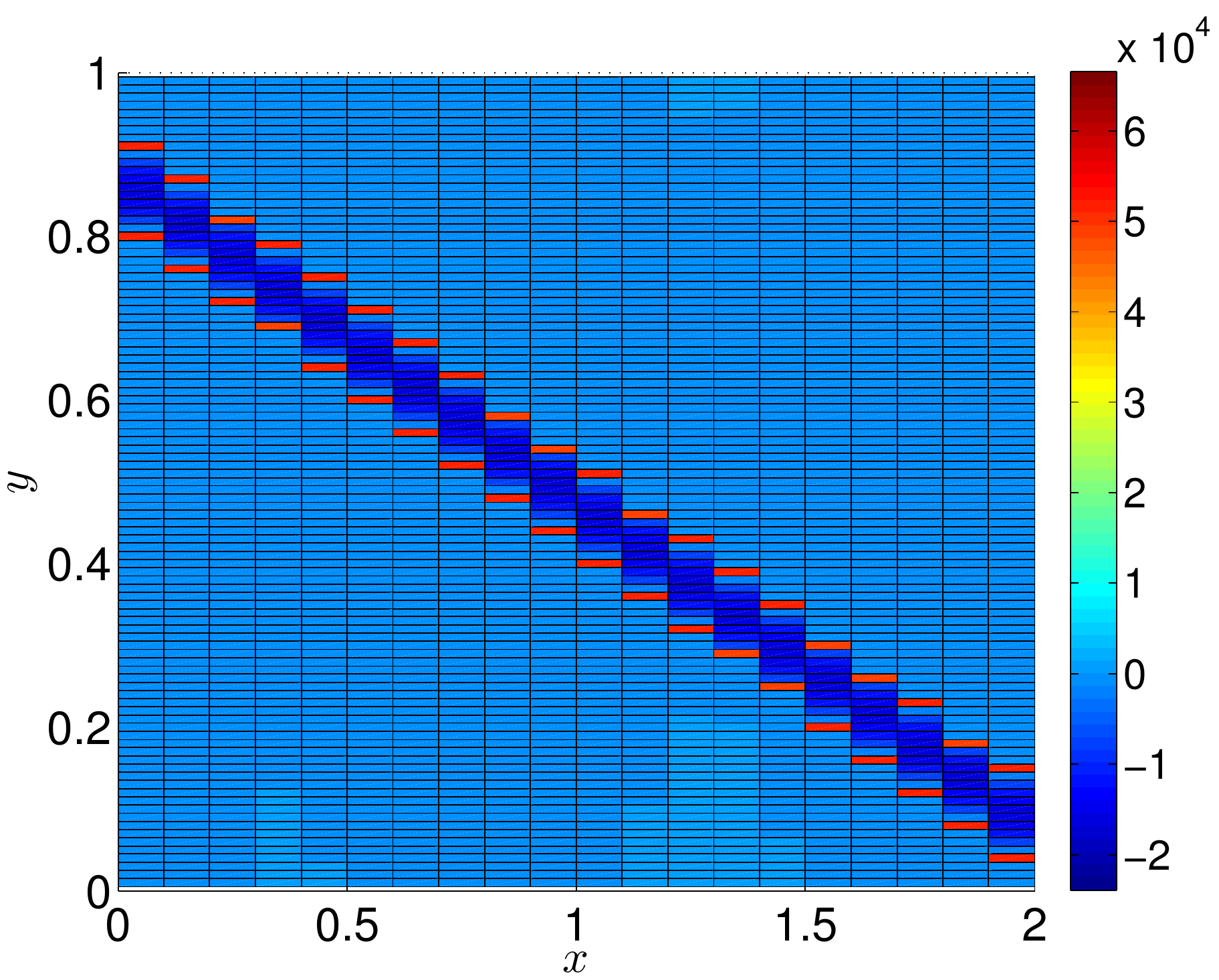} \label{without_advection_f}}\qquad
\subfloat[{\footnotesize $\mathbf{b}=(0,0)^{T}$, $\tilde{p}$} ]{\includegraphics[scale=0.3]{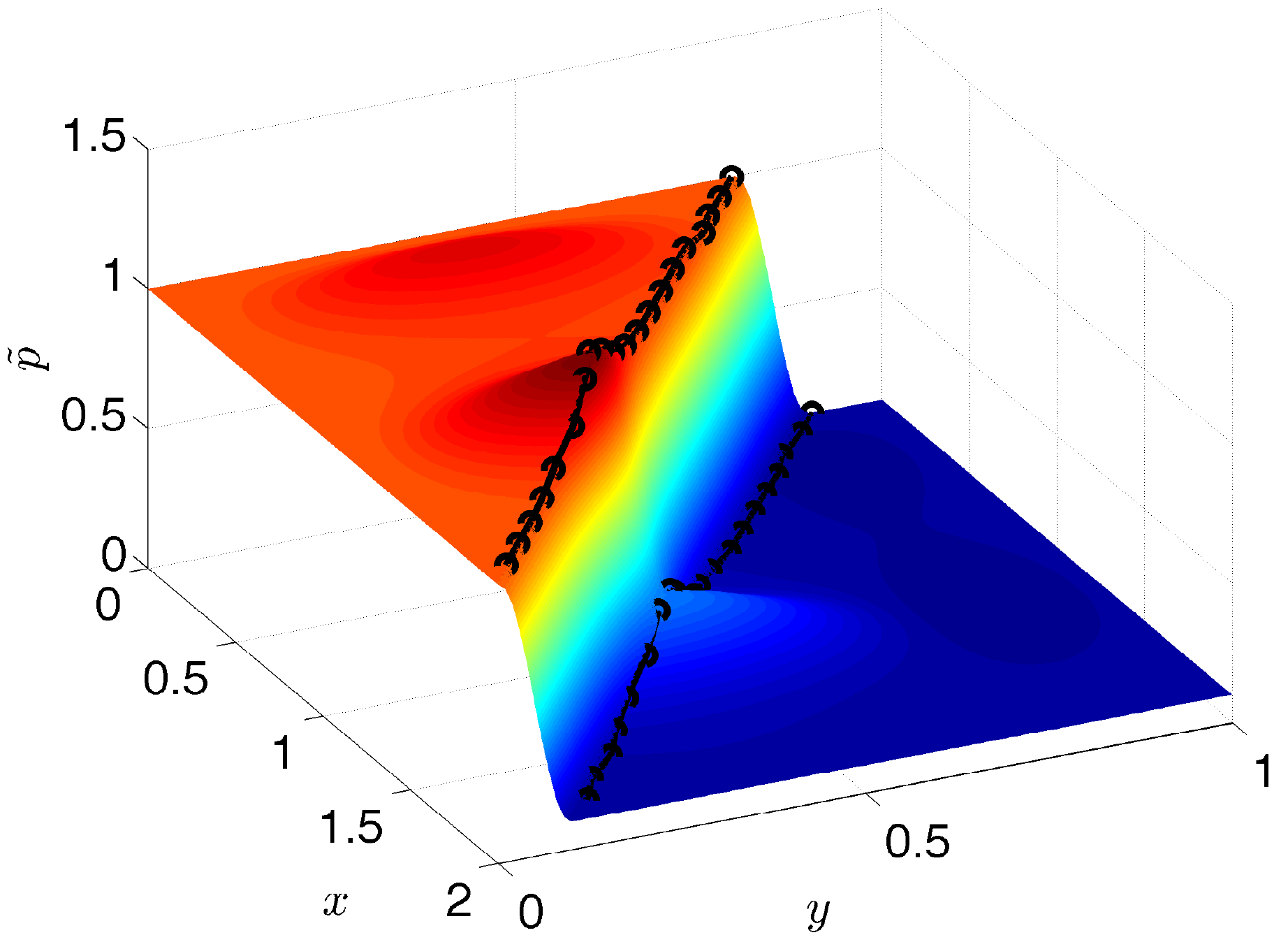} \label{without_advection}}\\
\subfloat[{\footnotesize $\mathbf{b}=(100,0)^{T}$, $\partial_{y} F$} ]{\includegraphics[scale=0.3]{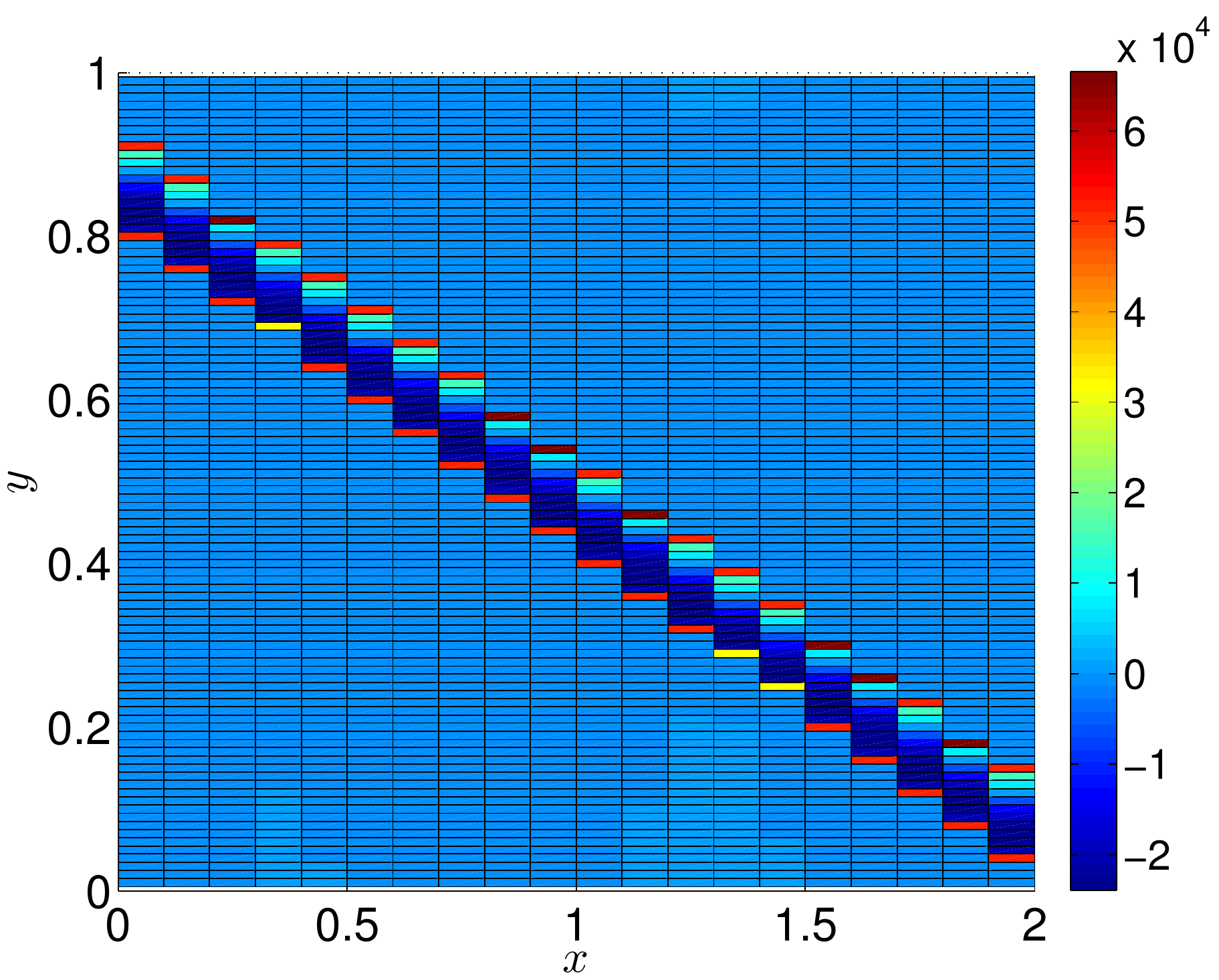} \label{with_advection_f}}\qquad
\subfloat[{\footnotesize $\mathbf{b}=(100,0)^{T}$, $\tilde{p}$} ]{\includegraphics[scale=0.3]{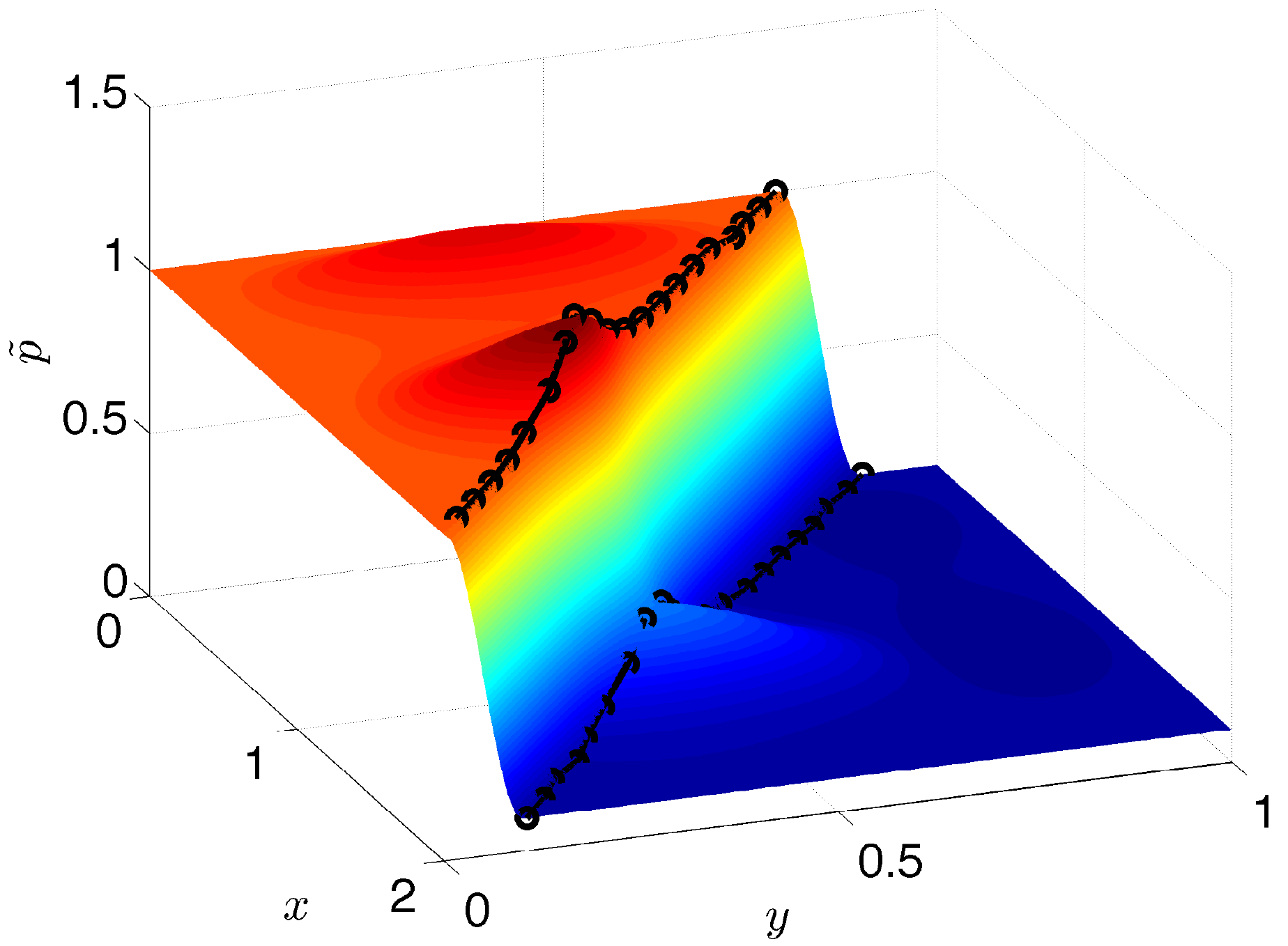} \label{with_advection}}\\
\caption{Discrete derivative in $y$-direction of $F$ for model problem \eqref{fullprob_gd} with $\mathbf{b}=(0,0)^{T}$ (a) and $\mathbf{b}=(100,0)^{T}$ (c) and $\tilde{p}$ with the interpolant resulting from the procedure introduced in \S \ref{locate_interface} for $\mathbf{b}=(0,0)^{T}$ (b) and $\mathbf{b}=(100,0)^{T}$ (d). \label{p_with_h}}
\end{figure}

\subsection{Locating the interface}\label{subsect:locate interface}

We consider a numerical example with an exact solution $\tilde{p} = p + \mathfrak{h}$ on $\Omega$, where the interface $\mathfrak{h}$ and $p$ are given as
\begin{align}
\label{test1_interface}\mathfrak{h}(x,y) &= \begin{cases}
                        1 \quad &\text{if} \enspace y + 0.4x < 0.8,\\
                        0.1 \quad & \text{if} \enspace y + 0.4x > 0.9,\\
                        0.55 + 0.45\cos(10\pi (y+0.4x-0.8)) &\text{if} \enspace 0.8 \leq y+0.4x\leq 0.9
                        \end{cases},\\
\label{test1_interfaceb} p(x,y) &= 5y^{2}(1-y)^{2}(0.75-y)x(2-x)\exp(\sin(2\pi x)). \qquad \qquad
\end{align}
First, we consider $k=1.0$ and $\mathbf{b}=(0,0)^{T}$ in \eqref{fullprob_gd}. We introduce partitions $\mathcal{T}_{H'}$ and $\tau_{h}$ with $N_{H'}=20$ and $n_{h}=100$ elements, respectively. The discrete derivative of $F$ in $y$-direction is depicted in Fig.~\ref{without_advection_f}. As suggested in \S \ref{locate_interface} we first identified the two elements in $y$-direction which exhibit the maximal value of the derivative for each grid point in $x$-direction. Subsequently, we interpolated between the midpoints of those elements. The resulting interpolant is plotted in Fig.~\ref{without_advection} in black; however, we have adjusted the values of the interpolant from the derivative of $F$ to $\tilde{p}$ such that it can be better compared with $\tilde{p}$. Recall to that end that we can infer the height of the interface approximately from the prescribed Dirichlet boundary conditions. In Fig.~\ref{without_advection} it can be observed that the procedure suggested in \S \ref{locate_interface} produces a very accurate approximation of the ``boundaries'' of the interface. 

Next, we consider $\tilde{p}$ as in the preceding example, $k=1.0$, and $\mathbf{b}=(100,0)^{T}$ and therefore a rather strong advective field. The function $F$ is thus chosen here as $F =-\Delta \tilde{p} + b \nabla \tilde{p}$. We use the same partitions $\mathcal{T}_{H'}$ and $\tau_{h}$ as above. If we compare the discrete derivative of $F$ in $y$-direction depicted in Fig.~\ref{with_advection_f} with the one of the previous example in Fig.~\ref{without_advection_f} we observe some (minor) changes due to the strong advective field but it is clearly observable that for each grid point in $x$-direction the two elements in $y$-direction with the maximal values of the discrete derivative are the same. Comparing in Fig.~\ref{with_advection}  the interpolant resulting from the discrete derivative in $y$-direction of $F$ for $\mathbf{b}=(100,0)^{T}$ with the solution $\tilde{p}$ we observe that also in this case the procedure proposed in \S \ref{locate_interface} yields a very good approximation of the ``boundaries'' of the interface. 

\begin{figure}[t]
\center
\includegraphics[scale=0.55]{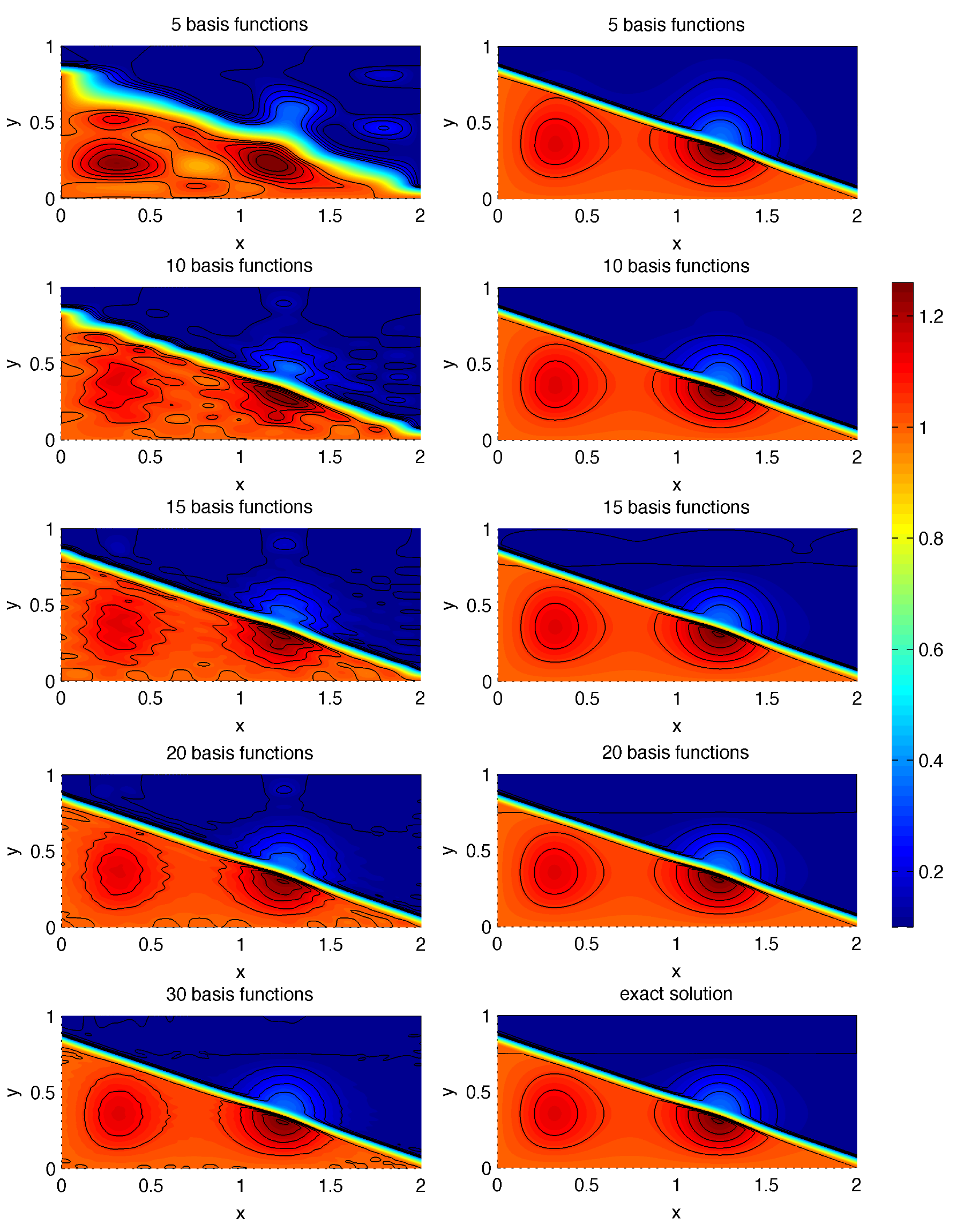} 
\caption{Test case 1: Comparison of the discrete reduced solution $p_{m}^{H}$ using the lifting function $g_{D}$ \eqref{lifting_gd} (left side) for $m =5,10,15,20,30$ (top-bottom), $p_{m}^{H}$ using $\mathfrak{h}$ \eqref{test1_interface} as the lifting function (right) and $m =5,10,15,20$ with the exact solution $\tilde{p}$ (right,bottom) for $N_{H}=800$, $n_{h}=400$, $N_{H'}=80$.\label{figwelle_loesung}}
\end{figure}
\subsection{Removing the interface from the model reduction procedure}\label{subsect:remove interface}

\paragraph*{Test case 1}

First, we consider a numerical example with an exact solution, where the full solution $p$ of \eqref{fullprob_gd} is a multiple of the analytic solution of test case 1 in \cite{OhlSme14} to enable a comparison with the situation where no interface is present. In detail, we solve a Poisson problem on $\Omega=(0,2)\times(0,1)$ with an exact solution $\tilde{p} = p + \mathfrak{h}$, where the interface $\mathfrak{h}$ and $p$ are defined as in 
\eqref{test1_interface} and \eqref{test1_interfaceb}, respectively. The solution $\tilde{p}$ is displayed in the last picture of Fig.~\ref{figwelle_loesung} and the skewed interface $\mathfrak{h}$ is clearly recognizable. First, we compare $\tilde{p}$ with the discrete reduced solution $\tilde{p}_{m}^{H}$, which has been computed using the lifting function 
\begin{equation}\label{lifting_gd}
g_{D} = (-0.5x+1) \mathfrak{h}(0,y) +0.5x \mathfrak{h}(2,y)
\end{equation}
for $m = 5,10,15,20,30$, $N_{H}= 800$, $n_{h} =400$ and $N_{H'} = 80$ (Fig.~\ref{figwelle_loesung}, left), where $N_{H'}$ has been defined in the previous section. 
Whereas $15$ basis functions are sufficient to obtain a good approximation of the interface $\mathfrak{h}$, we detect strong oscillations of $\tilde{p}_{15}^{H}$ in the other parts of the domain yielding still a bad approximation of $p$ \eqref{test1_interfaceb}. These oscillations decrease for increasing $m$ and for $m=30$ we obtain a reasonable approximation of $\tilde{p}$. If we use $\mathfrak{h}$ as the lifting function of the Dirichlet boundary conditions to compute $\tilde{p}_{m}^{H}$ (Fig.~\ref{figwelle_loesung}, right) no oscillations can be detected. The contour lines of $\tilde{p}_{20}^{H}$ match perfectly with the ones of $\tilde{p}$ and already for $m=10$ and $15$ only small deviations can be observed. All in all we see a much better qualitative convergence behavior for the solution of \eqref{prob_gd}. 

\begin{figure}[t]
\centering
\subfloat[{\footnotesize lifting function $g_{D}$ of \eqref{lifting_gd}} ]
{\includegraphics[scale = 0.31]{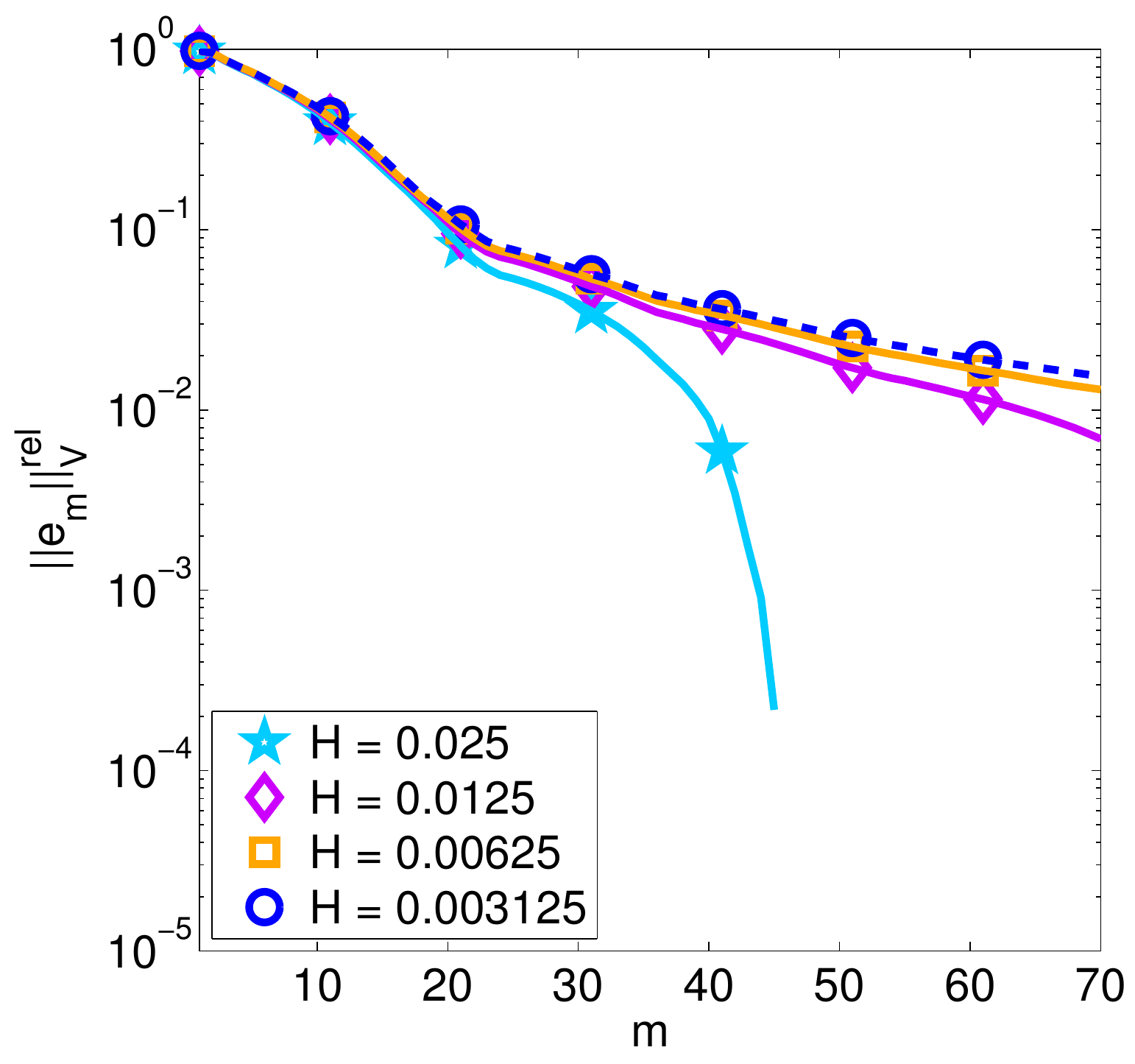}\label{fig44a}}
\subfloat[{\footnotesize solution of \eqref{prob_gd}, $\bar{Q}=2$}]{
\includegraphics[scale = 0.31]{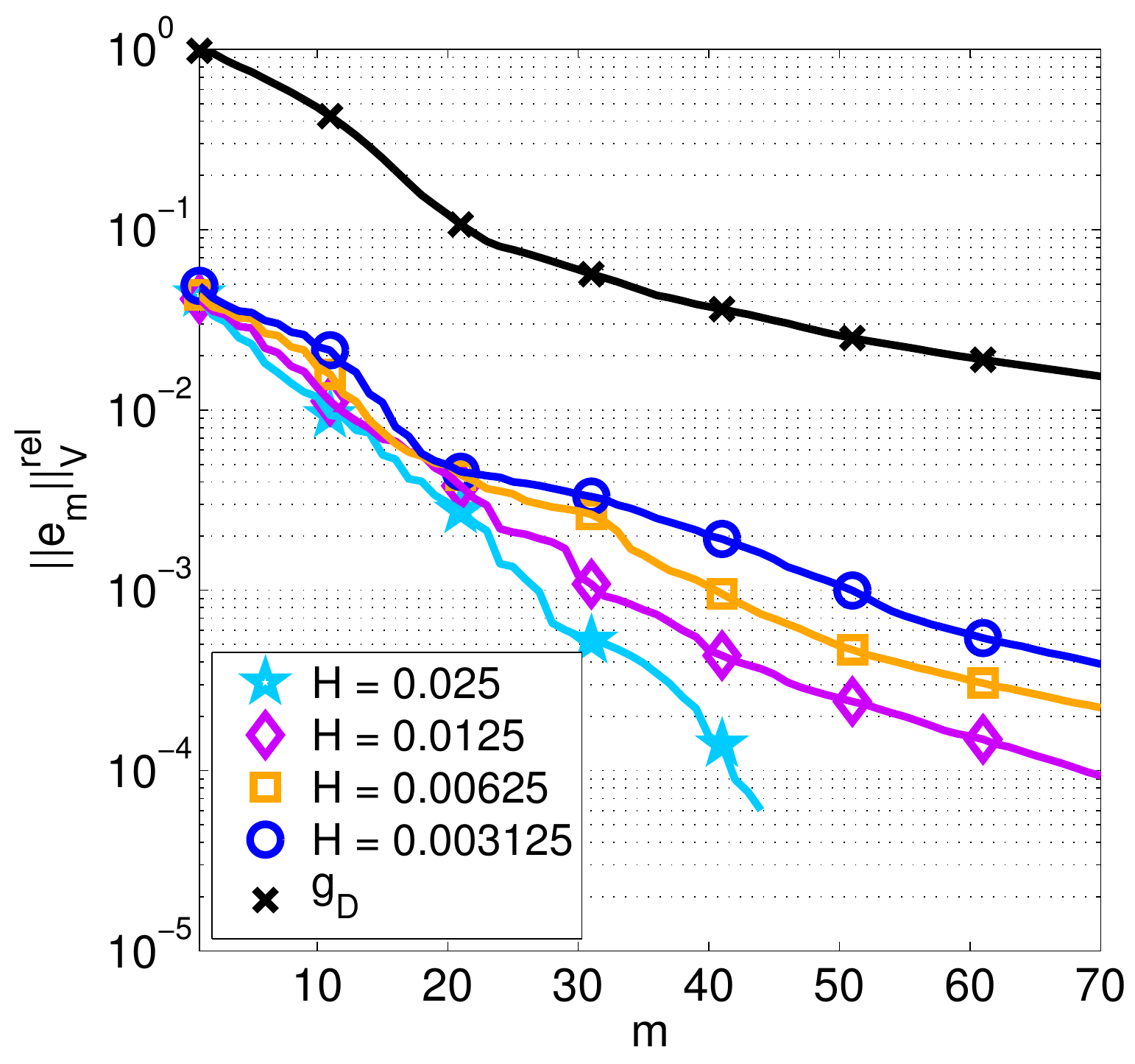}\label{fig44b}}
\subfloat[{\footnotesize including $\Delta \mathfrak{h}$}]{
\includegraphics[scale = 0.31]{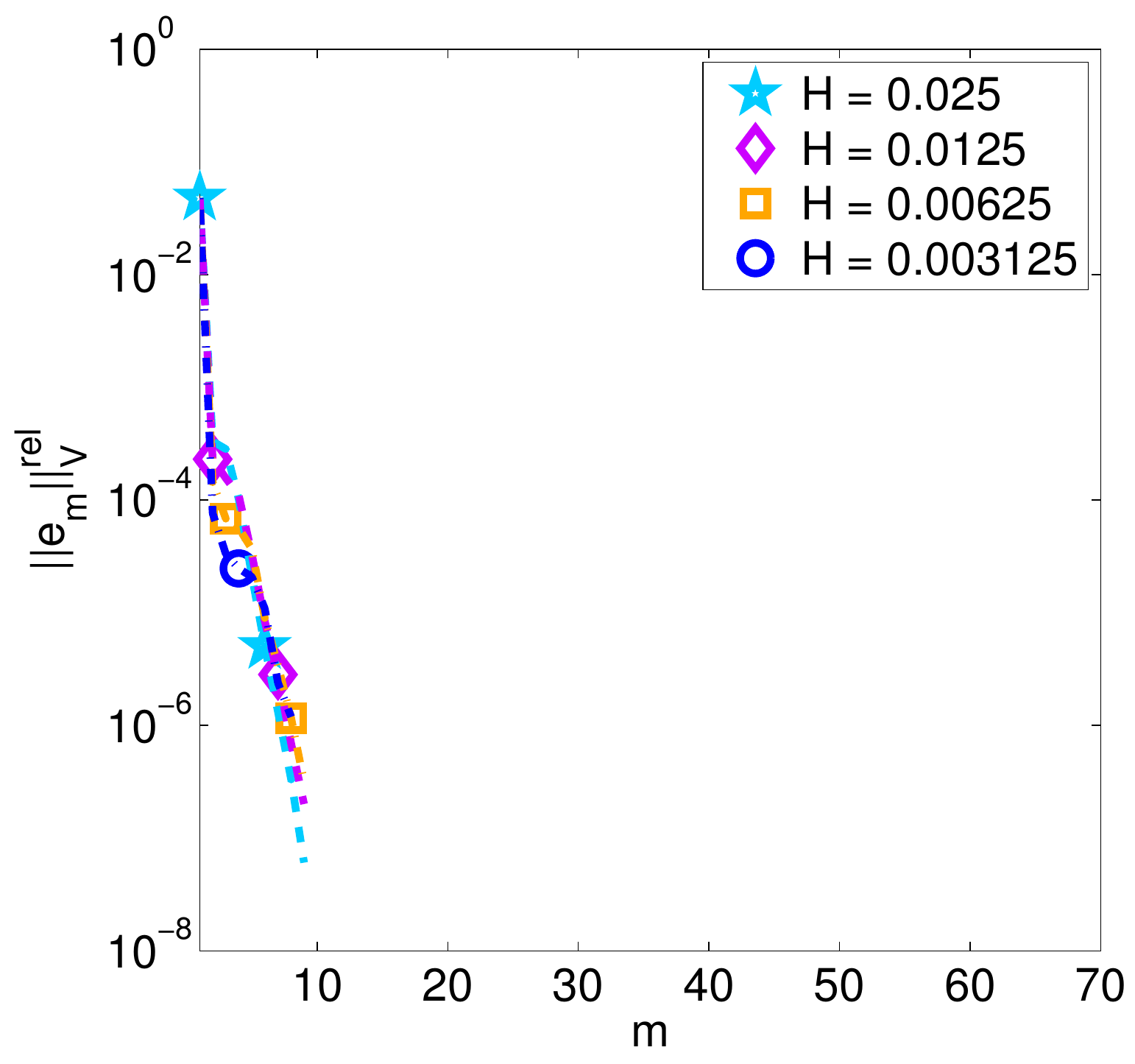}\label{fig44c}}
\caption{Test case 1: Comparison of the model error convergence
of $\|e_{m}\|_{V}^{rel}$. (a): We use the lifting function $g_{D}$ defined in \eqref{lifting_gd}. (b): We consider \eqref{prob_gd} and choose $\bar{Q} = 2$ in \eqref{1D_prob_quad_para_gd}. (c): We consider the right hand side $F+\Delta \mathfrak{h}$ in \eqref{prob_gd} and neglect the term $-a(\mathfrak{h},\xi_{i}^{H}\phi_{l})$ instead. All plots: $N_{H'}=10$.\label{fig441}}
\end{figure}

\begin{figure}[t]
\centering
\subfloat[{\footnotesize $\|e_{m}\|_{L^{2}(\Omega)}^{rel}$ and $ e_{m}^{\mbox{{\tiny POD}}}$}]{
\includegraphics[scale = 0.31]{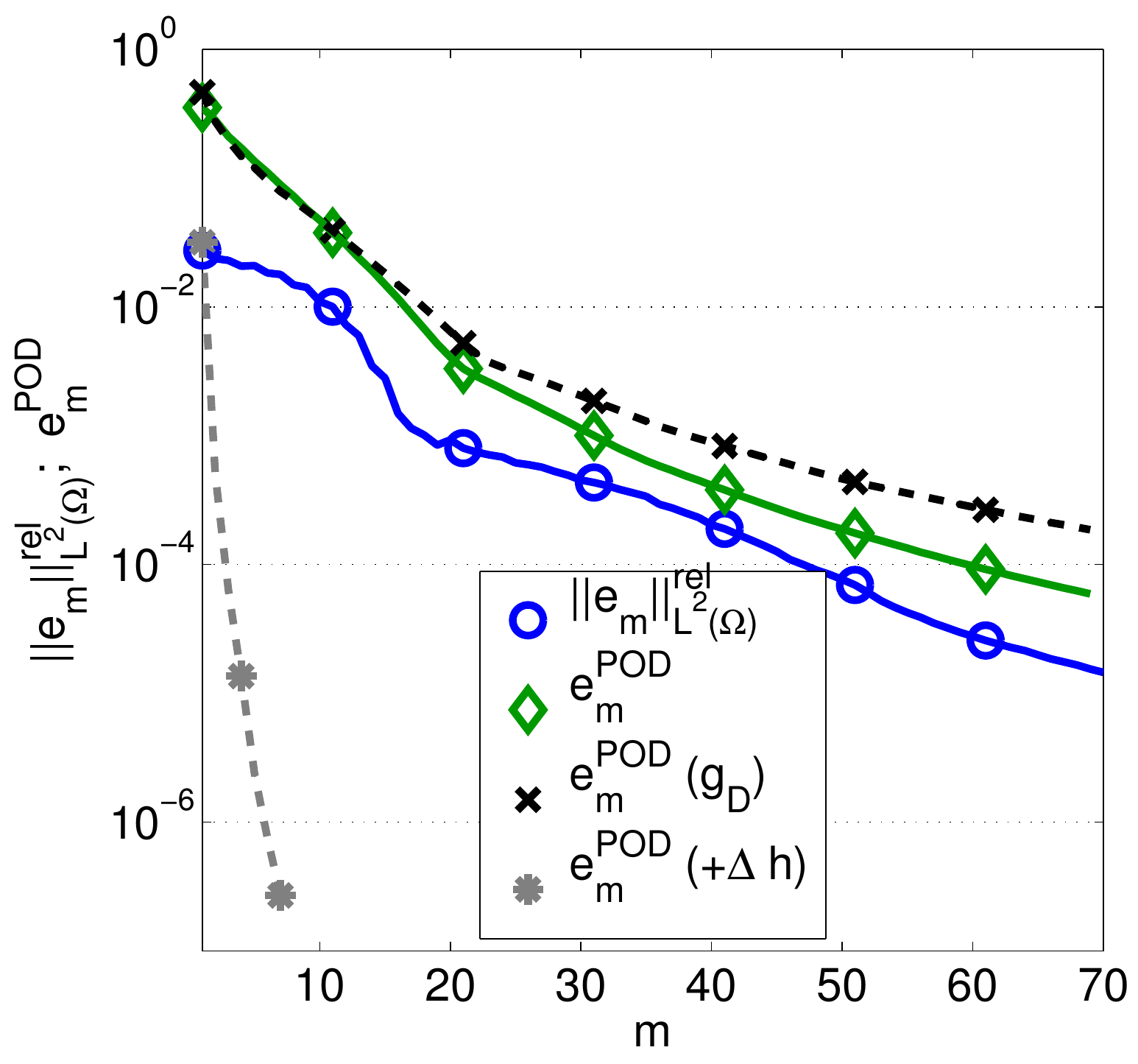} \label{fig44d}}
\subfloat[{\footnotesize $\lambda_{m}$ and $\|\bar{p}_{m}^{H}\|_{L^{2}(\Omega_{1D})}$}]{
\includegraphics[scale = 0.31]{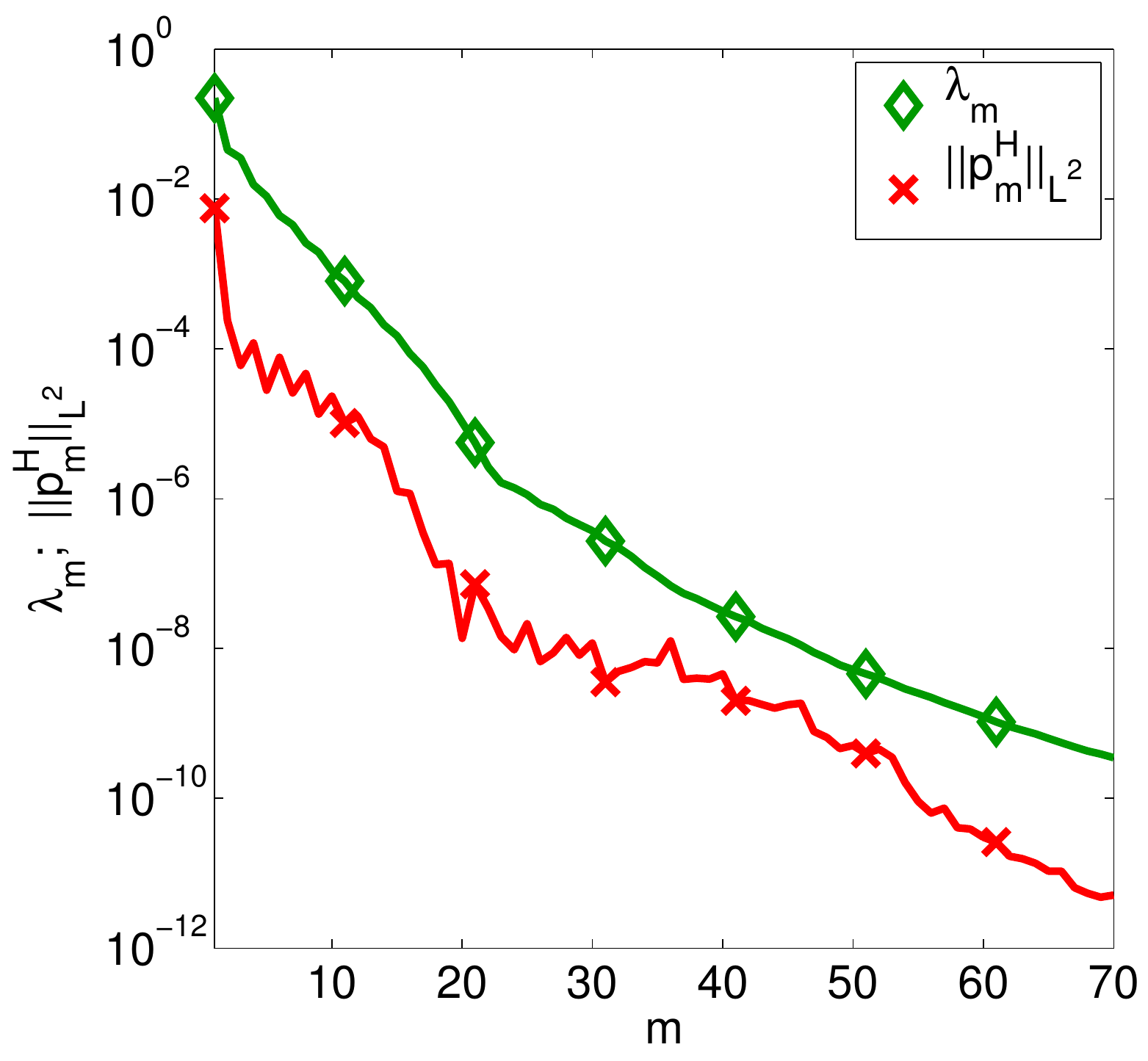}\label{fig44e}}
\subfloat[{\footnotesize $\|e_{m}\|_{V}^{rel}$, solution of \eqref{prob_gd}, $\bar{Q}=1$}]{
\includegraphics[scale = 0.31]{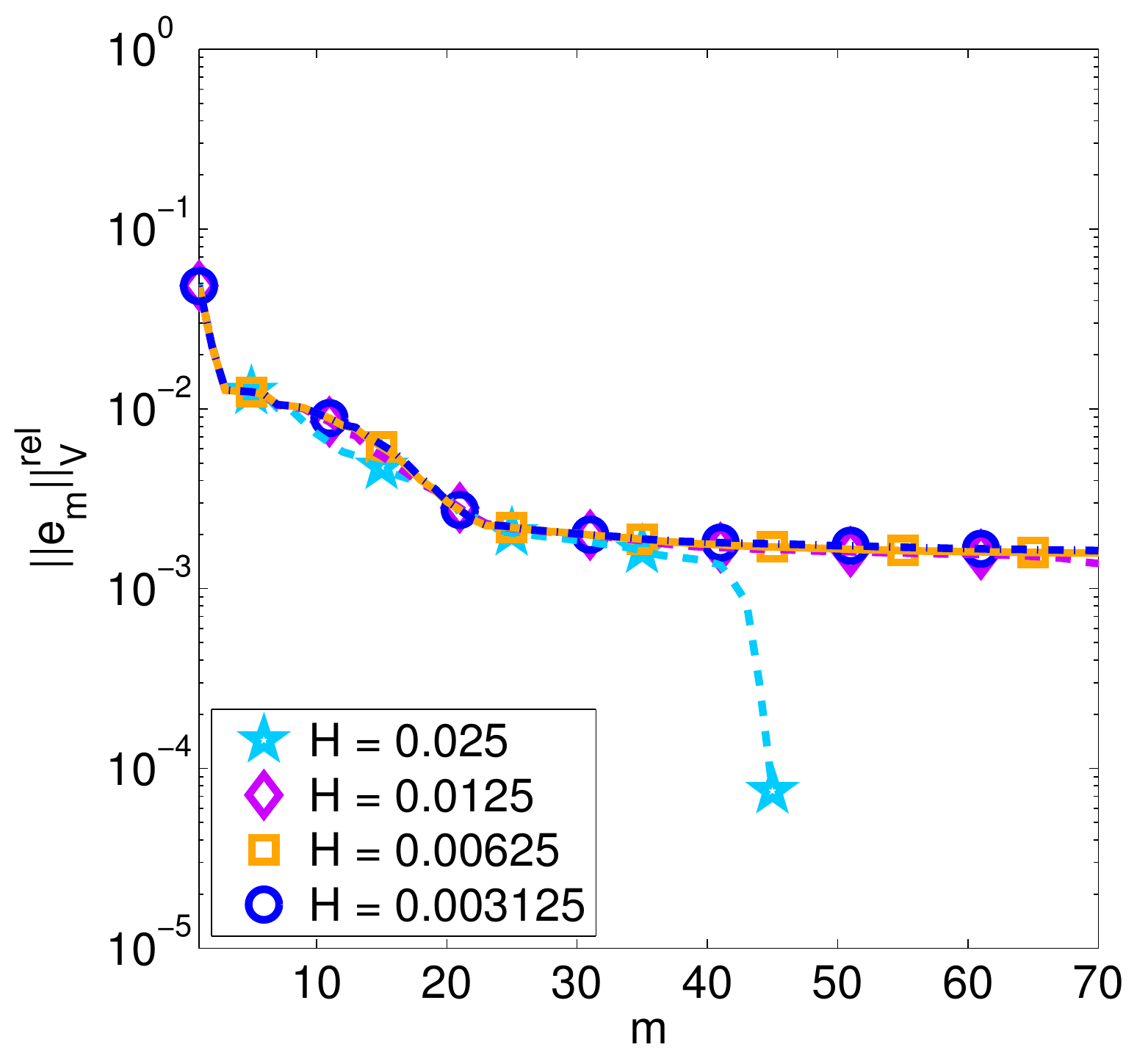} \label{fig44f}}
\caption{Test case 1: (a): Comparison of $\|e_{m}\|_{L^{2}(\Omega)}^{rel}$ and $e_{m}^{\mbox{{\tiny POD}}}$ for the following situations and legend entries: $e_{m}^{\mbox{{\tiny POD}}}$: We consider the discrete reduced problem \eqref{prob_gd}. $e_{m}^{\mbox{{\tiny POD}}} (+ \Delta \mathfrak{h})$: We consider the right hand side $F+\Delta \mathfrak{h}$ in \eqref{prob_gd} and neglect the term $-a(\mathfrak{h},\xi_{i}^{H}\phi_{l})$ instead. $e_{m}^{\mbox{{\tiny POD}}} (g_{D})$: We use the lifting function $g_{D}$ defined in \eqref{lifting_gd}. (b) Comparison of $\lambda_{m}$ and $\|\bar{p}_{m}^{H}\|_{L^{2}(\Omega_{1D})}^{2}$ for the solution of \eqref{prob_gd}. (c): Comparison of $\|e_{m}\|_{V}^{rel}$ for \eqref{prob_gd} and $\bar{Q} = 1$ in \eqref{1D_prob_quad_para_gd}. All plots: $N_{H'}=10$. \label{fig44}}
\end{figure}

\begin{figure}[t]
\centering
\subfloat[{\footnotesize lift. funct. $g_{D}$ \eqref{lifting_gd}}]{
\includegraphics[scale = 0.23]{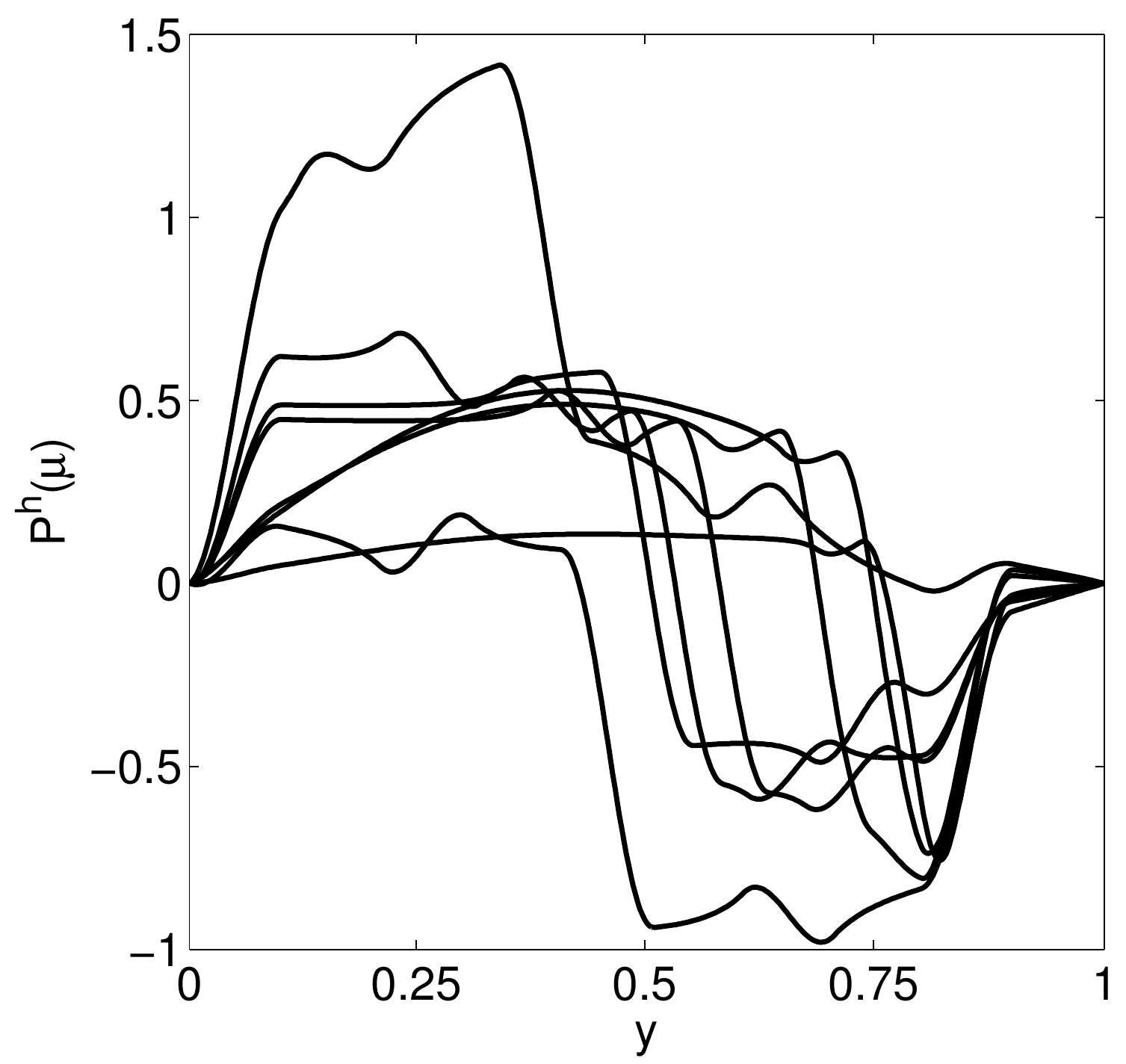}\label{figmani_wellec}}
\subfloat[{\footnotesize incl. $\Delta \mathfrak{h}$}]{
\includegraphics[scale = 0.23]{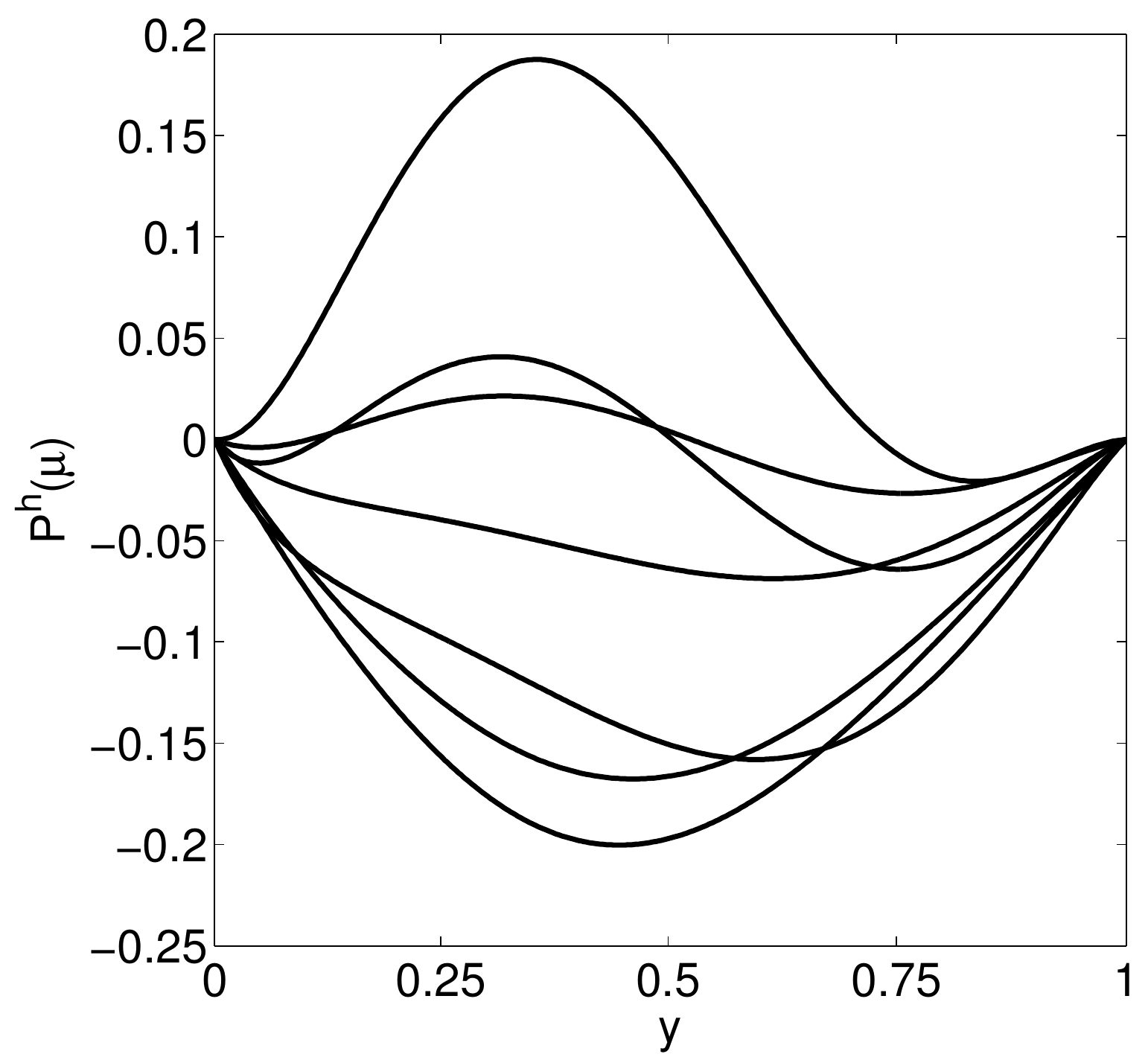}\label{figmani_welleb}}
\subfloat[{\footnotesize \eqref{1D_prob_quad_para_gd}, $\bar{Q}=2$} ]
{\includegraphics[scale = 0.23]{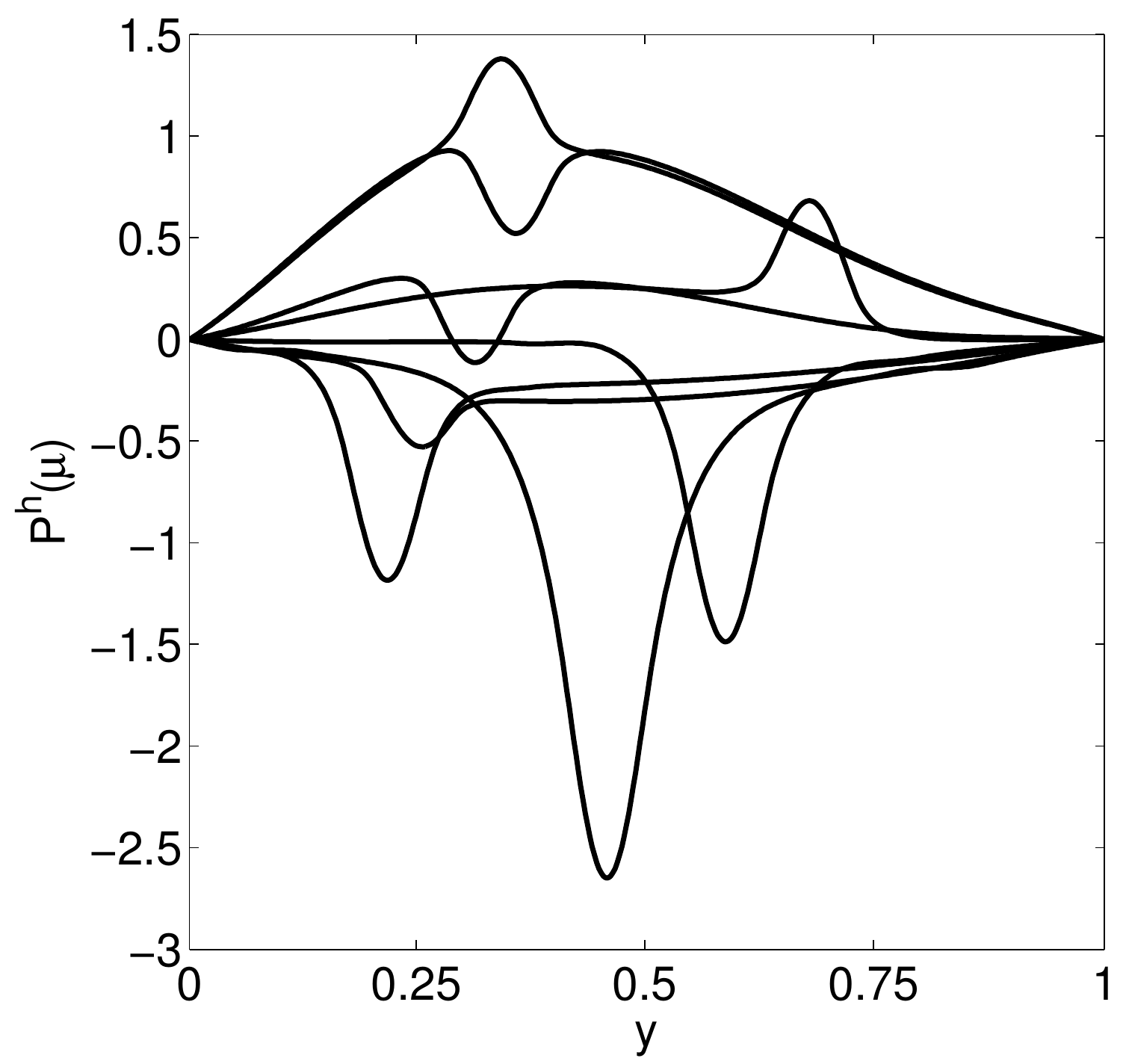}\label{figmani_wellea}}
\subfloat[{\footnotesize \eqref{1D_prob_quad_para_gd}, $\bar{Q}=1$}]{
\includegraphics[scale = 0.23]{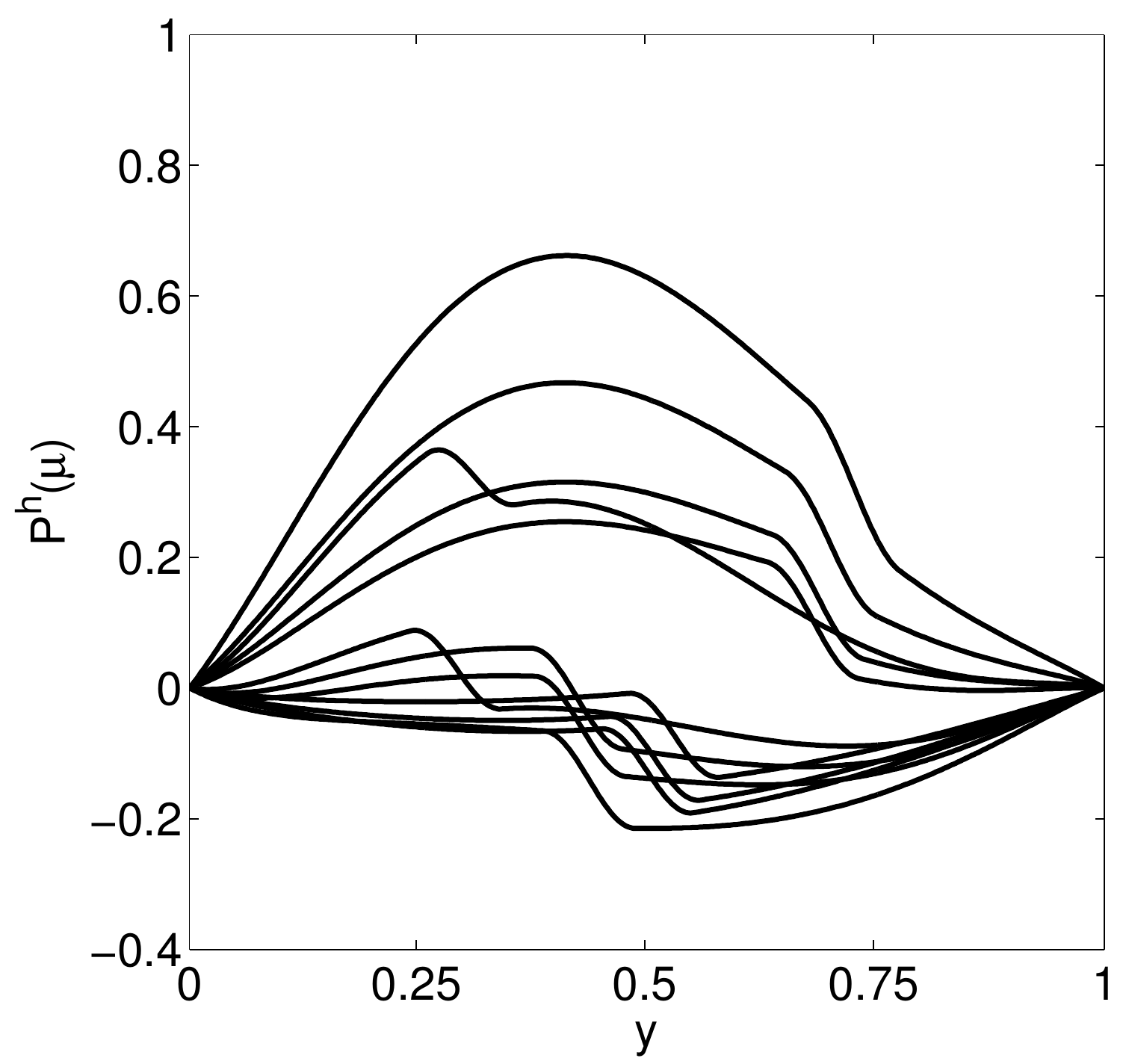} \label{figmani_welled}}
\caption{Test case 1: Typical snapshots of the discrete solution manifold $\mathcal{M}_{\Xi}$ \eqref{disc_mani} if the lifting function $g_{D}$ \eqref{lifting_gd} is used and therefore no information on the interface is included in the model reduction procedure (a), if $\Delta \mathfrak{h}$ is added to $F$ in \eqref{red_prob_gd}  (b), \eqref{1D_prob_quad_para_gd} is solved with $\bar{Q}=2$ (c) and $\bar{Q} = 1$ (d).\label{figmani_welle}}
\end{figure}
To analyze the capacity of the proposed method to improve the convergence behavior also from a quantitative viewpoint, we compare the convergence behavior of the relative model error $\|e_{m}\|_{V}^{rel}$ for increasing model order $m$. 
Using $g_{D}$ requires more than $70$ basis functions to obtain $\|e_{m}\|_{V}^{rel} \leq 0.01$ for $H \leq 0.00625$ (Fig.~\ref{fig44a}), which can be attained for $m = 16$ if choosing $\mathfrak{h}$ instead (Fig.~\ref{fig44b}). However, we see in Fig.~\ref{fig44b} that the convergence rate has only improved slightly and that the improvement of the convergence behavior can be mainly ascribed to a better relative error already for $m =1$ of the solution of \eqref{prob_gd}. We suppose that this is due to the fact that the information on the derivative of the interface in the dominant direction is not properly included in the parametrized lower dimensional system \eqref{1D_prob_quad_para_gd}. We have thus recomputed the RB-HMR approximation, employing a reconstruction of the derivative of the interface in the dominant direction for the derivation of the parameterized lower dimensional system \eqref{1D_prob_quad_para_gd}.  This reconstruction mimics the behavior of the derivative in two space dimensions and is added to the right hand side $F$, where the term $a^{q}(\mathfrak{h},\upsilon^{h}_{j}\xi_{k}^{q};\mu)$ is neglected. One example for such a reconstruction is the Riesz representative $\mathcal{R}^{H\times h}$, defined as the solution of 
\begin{equation}\label{Riesz_gd}
\int_{\widehat{\Omega}} \mathcal{R}^{H\times h} v^{H\times h} = a(\mathfrak{h}, v^{H\times h}), \qquad \forall v^{H\times h} \in V^{H \times h}. 
\end{equation}
Note that adding $\mathcal{R}^{H\times h}$ is equivalent to adding 
$\Delta \mathfrak{h}$, if $\mathfrak{h}$ is sufficiently regular. Finally, we add also in \eqref{prob_gd} the reconstruction to the right hand side and neglect the term $-a(\mathfrak{h},\xi_{i}^{H}\phi_{l})$ instead. As $\mathfrak{h}$ is sufficiently regular for the present test case, we use $\Delta \mathfrak{h}$ as a reconstruction. By doing so we obtain the expected fast exponential convergence rate of $\|e_{m}\|_{V}^{rel}$ (Fig.~\ref{fig44c}), already observed in test case 1  in \cite{OhlSme14}, Fig.~2(a). 

As the convergence rates of $\|e_{m}\|_{L^{2}(\Omega)}^{rel}$ and $e_{m}^{{\tiny \text{POD}}}$ and also $\lambda_{m}$ and $\|\bar{p}_{m}^{H}\|_{L^{2}(\Omega_{1D})}$ for the solution of \eqref{prob_gd} coincide to a great extend (Fig.~\ref{fig44d} and Fig.~\ref{fig44e}), we conclude that already the discrete solution manifold $\mathcal{M}_{\Xi}$ --- defined in \eqref{disc_mani} --- used for the computation of the reduction space $Y_{m}$ and thus the discrete reduced solution $\tilde{p}_{m}^{H}$ of \eqref{prob_gd} can be badly approximated by a $m$-dimensional subspace. Hence although the reference solution $p^{H\times h}$ of \eqref{truth_prob_gd} can be approximated exponentially fast (Fig.~\ref{fig44c}), this cannot be exploited due to a deficient construction of $\mathcal{M}_{\Xi}$. We further infer that the equivalence of including $\mathfrak{h}$ in the strong or weak formulation of \eqref{red_prob_gd} cannot be reproduced properly by the lower-dimensional problems, because otherwise the convergence rates of $e_{m}^{{\tiny \text{POD}}}$ would be the same for both cases. 

This statement can be confirmed by comparing the discrete solution manifold $\mathcal{M}_{\Xi}$ for the different cases. If using $g_{D}$ as the lifting function the snapshots $\mathcal{P}^{h}(\mu)$ contain all the sharp interface, have seldom the right slope in the other part of $\widehat{\omega}$ and exhibit often strong oscillations (Fig.~\ref{figmani_wellec}), explaining both the qualitative and quantitative convergence behavior obtained in this case. If we add $\Delta \mathfrak{h}$ in the strong formulation of the PDE we see in Fig.~\ref{figmani_welleb} that the snapshots $\mathcal{P}^{h}(\mu)$ resemble evaluations of $p$ \eqref{test1_interfaceb}. In contrast the snapshots of \eqref{1D_prob_quad_para_gd} for $\bar{Q}=2$ are either located peaks or additionally contain small peaks at the location of the interface (Fig.~\ref{figmani_wellea}), showing that the interface $\mathfrak{h}$ has not been removed completely from the snapshots $\mathcal{P}^{h}(\mu)$ as it has been possible in the case of adding $\Delta \mathfrak{h}$ (Fig.~\ref{figmani_welleb}). Nevertheless, we emphasize that a significant improvement of the convergence behavior has been achieved (Fig.~\ref{fig44b}) 
for $2$ quadrature points. Finally, we remark that for $\bar{Q}=1$ we observe for $m \geq 20$ a stagnation of $\|e_{m}\|_{V}^{rel}$ (Fig.~\ref{fig44f}). This can be ascribed to the fact that although the snapshots $\mathcal{P}^{h}(\mu) \in \mathcal{M}_{\Xi}$ in Fig.~\ref{figmani_welled} show a certain similarity to the snapshots in Fig.~\ref{figmani_welleb}, they all contain different slopes of the interface, which are not present in $p$ \eqref{test1_interfaceb}, prohibiting a convergence to the reference solution. This indicates that the artificial coupling introduced in \S \ref{1dproblem_bilFEM} is necessary to obtain a good approximation behavior of the RB-HMR approach. 
\begin{figure}[h!]
\center
\includegraphics[scale=0.6]{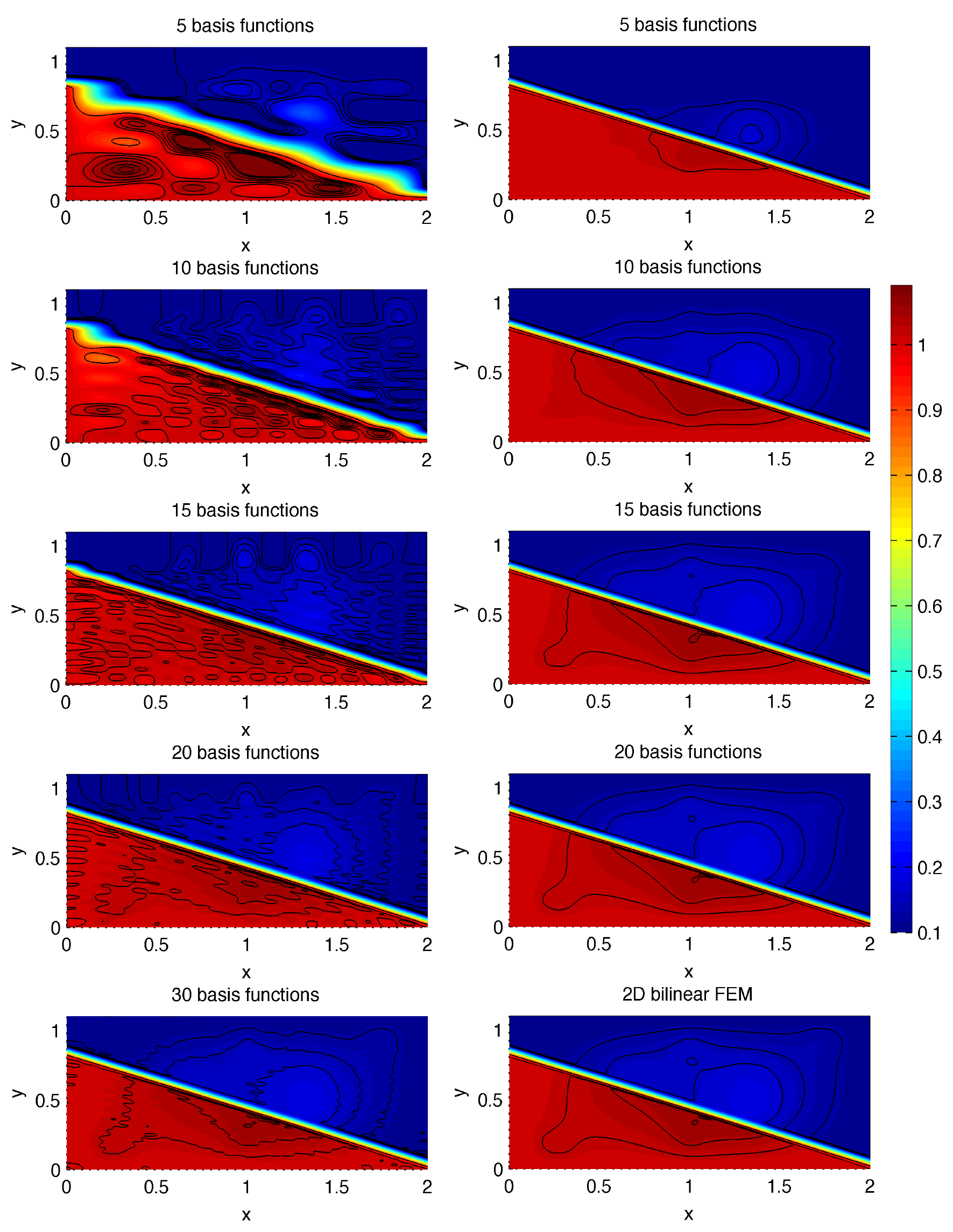} 
\caption{Test case 2: Comparison of the discrete reduced solution $p_{m}^{H}$ using the lifting function $g_{D}$ defined in \eqref{lifting_gd2} (left side) for $m =5,10,15,20,30$ (top-bottom), $p_{m}^{H}$ using $\mathfrak{h}$ \eqref{test2_interface} as the lifting function (right) and $m =5,10,15,20$ with the reference solution $\tilde{p}^{H\times h}$ (right,bottom) for $N_{H}=800$, $n_{h}=440$, $N_{H'}=80$.\label{figsource_loesung}}
\end{figure}

\paragraph*{Test case 2}
\begin{figure}[t]
\centering
\subfloat[{\footnotesize lifting function $g_{D}$ \eqref{lifting_gd2}} ]
{\includegraphics[scale = 0.31]{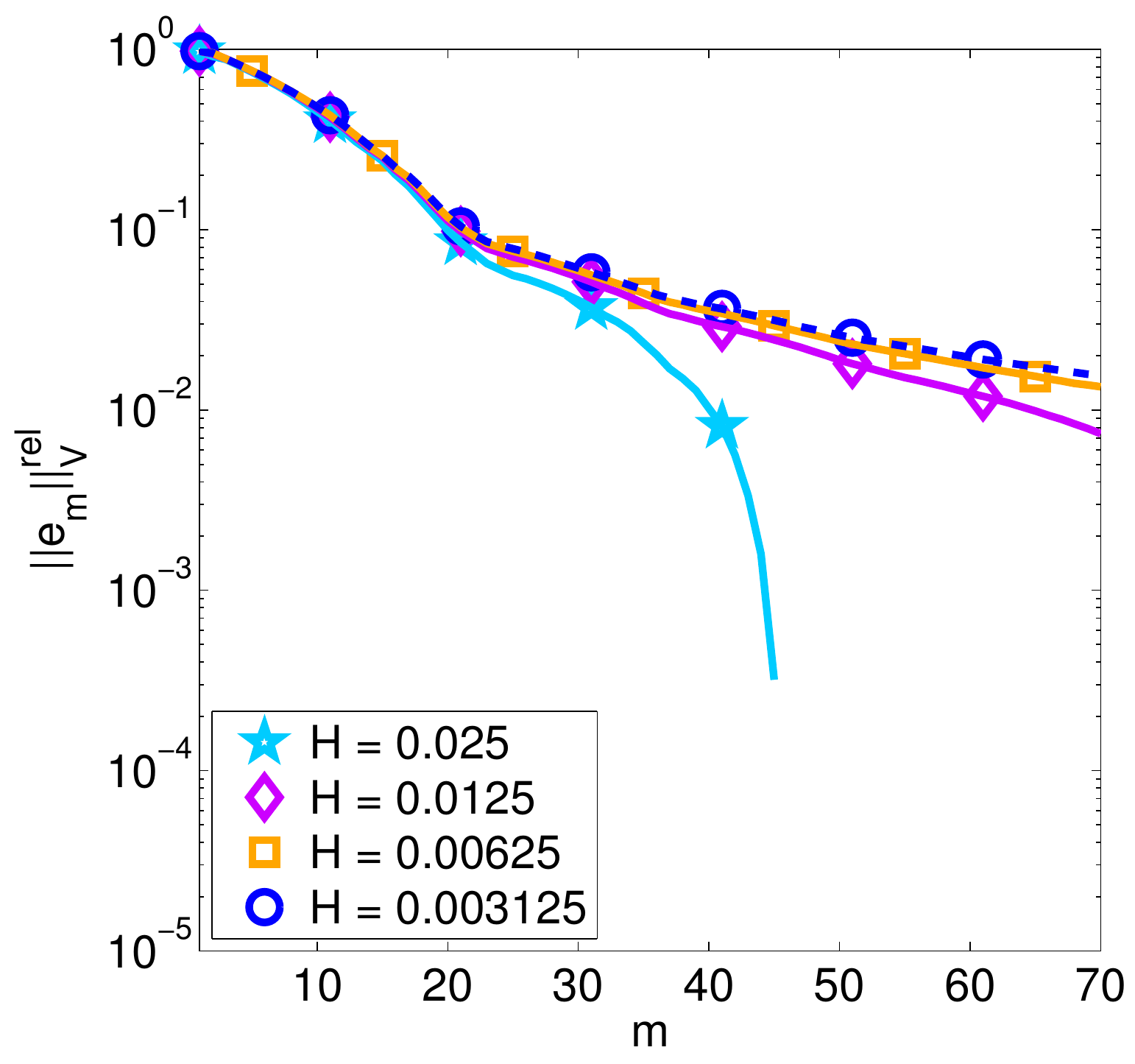}\label{fig46a}}
\subfloat[{\footnotesize solution of \eqref{prob_gd}, $\bar{Q}=2$}]{
\includegraphics[scale = 0.31]{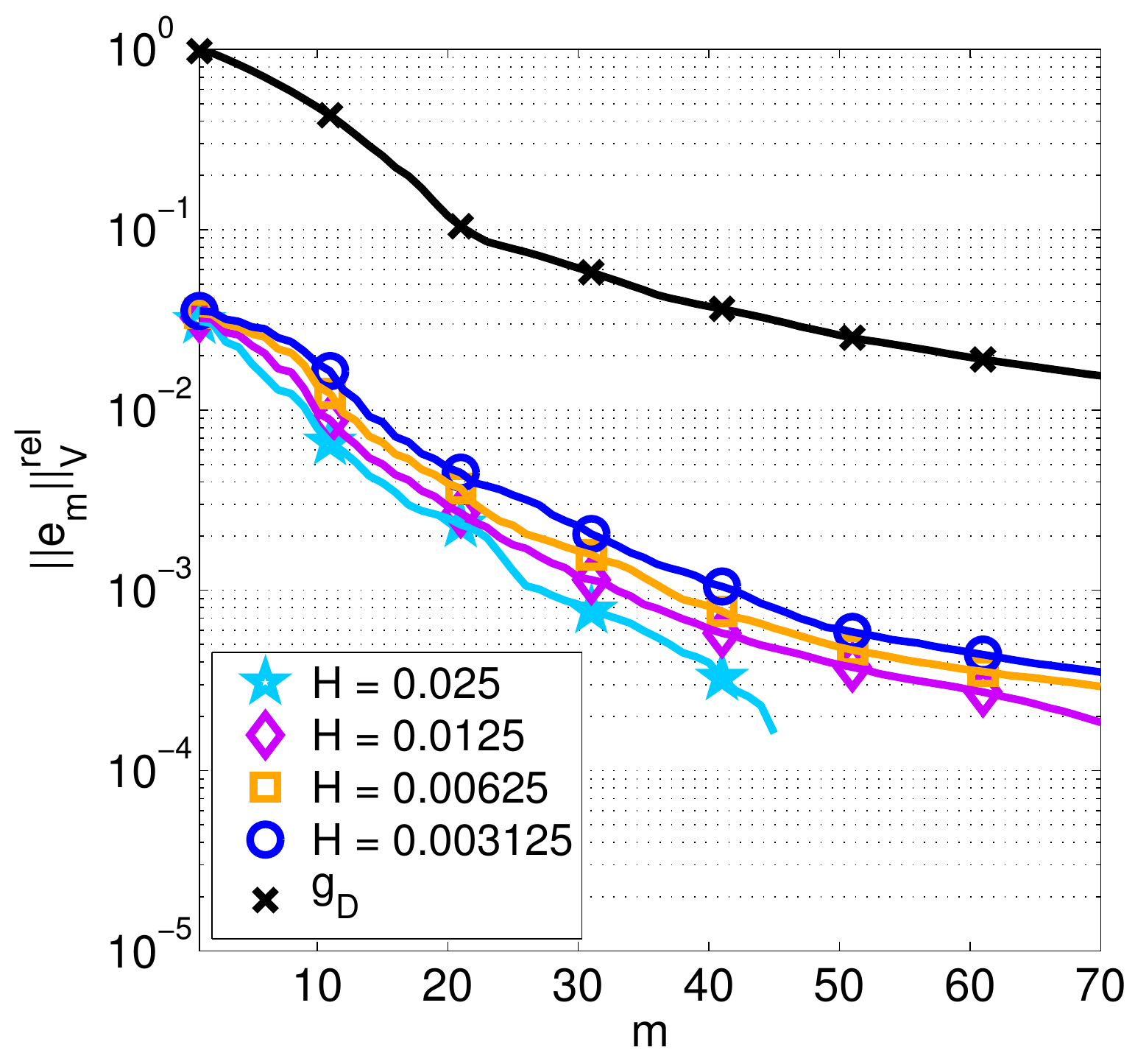}\label{fig46b}}
\subfloat[{\footnotesize including $\Delta \mathfrak{h}$}]{
\includegraphics[scale = 0.31]{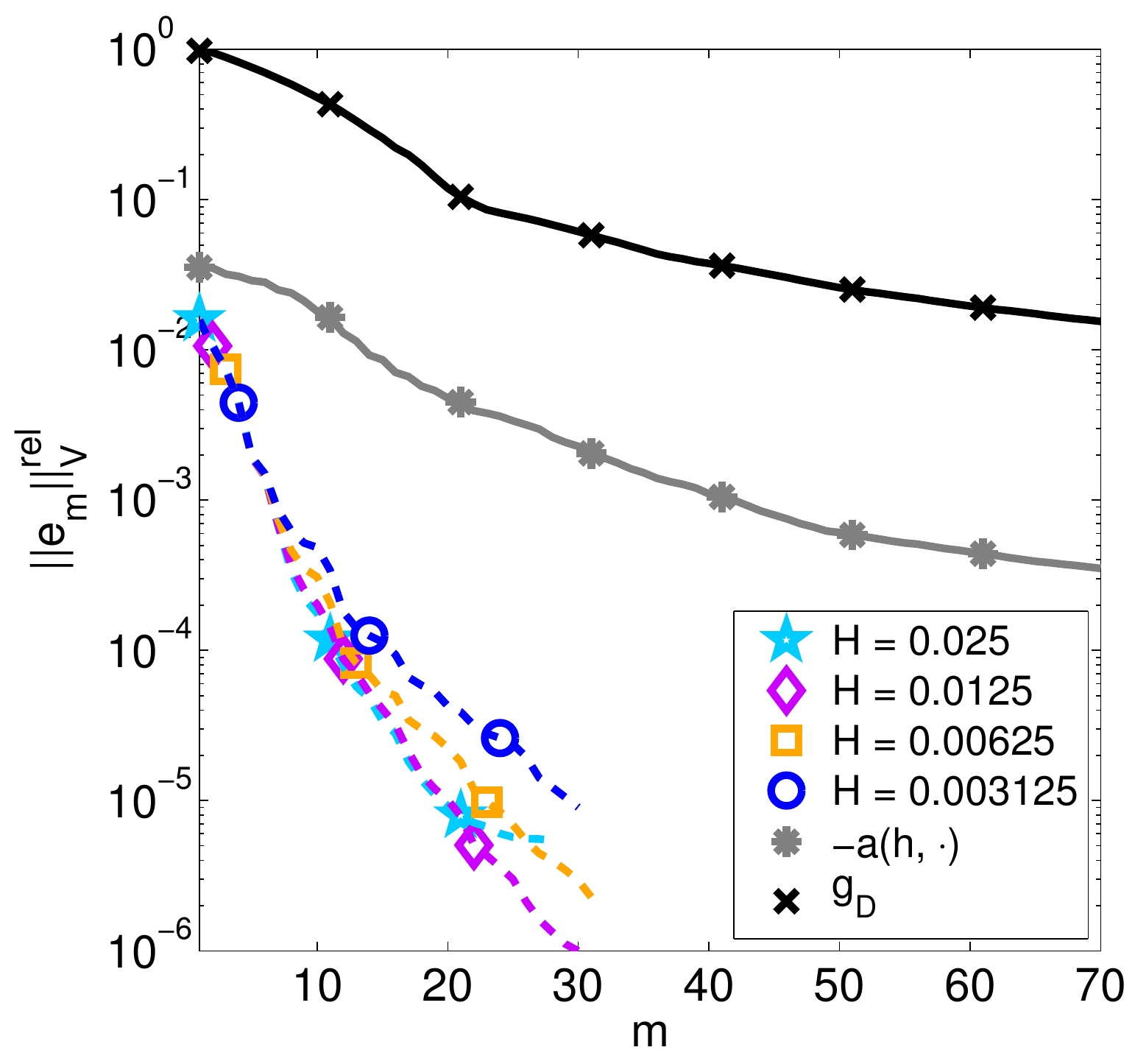}\label{fig46c}}
\caption{Test case 2: Comparison of the model error convergence
of $\|e_{m}\|_{V}^{rel}$. (a): We use the lifting function $g_{D}$ defined in \eqref{lifting_gd2}. (b): We consider \eqref{prob_gd} and choose $\bar{Q} = 2$ in \eqref{1D_prob_quad_para_gd}. (c): We consider the right hand side $F+\Delta \mathfrak{h}$ in \eqref{prob_gd} and neglect the term $-a(\mathfrak{h},\xi_{i}^{H}\phi_{l})$ instead. All plots: $N_{H'}=10$.\label{fig46}}
\vspace*{-10pt}
\end{figure}
In this test case we consider a problem in which the solution exhibits more complex features than in the previous example due to a discontinuous source term. In detail, we solve a Poisson problem on $\Omega=(0,2)\times(0,1.1)$ with a concentration profile $\mathfrak{h}$ and a source term $F$  defined as
 \begin{align}                                           
\label{test2_interface}  \mathfrak{h}(x,y) = \begin{cases}
                        1 \quad &\text{if} \enspace y + 0.4x < 0.8,\\
                        0.1 \quad & \text{if} \enspace y + 0.4x > 0.9,\\
                        0.55 + 0.45\cos(10\pi (y+0.4x-0.8)) &\text{if} \enspace 0.8 \leq y+0.4x\leq 0.9.
                        \end{cases}
\end{align}
and
\begin{align*}
 \quad F(x,y) = - \Delta \mathfrak{h} + \begin{cases}
2.25 \quad &\text{if} \enspace 0.15 \leq x \leq 0.35, \enspace 0.05 \leq y \leq 0.25,\\ 
2.25 \quad &\text{if} \enspace 0.55 \leq x \leq 0.75, \enspace 0.6 \leq y \leq 0.8,\\ 
4 \quad &\text{if} \enspace 0.95 \leq x \leq 1.05, \enspace 0.15 \leq y \leq 0.35\enspace \text{and} \enspace 0.75 \leq y \leq 0.95,\\ 
4 \quad &\text{if} \enspace 1.25 \leq x \leq 1.45, \enspace 0.35 \leq y < 0.55,\\ 
2 \quad &\text{if} \enspace 1.25 \leq x \leq 1.45, \enspace 0.55 \leq y \leq 0.75,\\ 
2.25 \quad &\text{if} \enspace 1.65 \leq x \leq 1.85, \enspace 0.85 \leq y \leq 1.05,
\end{cases}
 \end{align*}
We prescribe non-homogeneous Dirichlet boundary conditions on the whole $\partial \Omega$, where the respective boundary values are obtained by evaluating $\mathfrak{h}$ on $\partial \Omega$. The reference solution $\tilde{p}^{H\times h}$ for $N_{H}=800$, $n_{h} =440$ is depicted in the last picture of Fig.~\ref{figsource_loesung} and contains apart from the skewed interface $\mathfrak{h}$ also a part $p^{H\times h}$ induced by the source term. We have done a convergence study in the mesh size to ensure that $\tilde{p}^{H\times h}$ contains all essential features of the exact solution.
Comparing $\tilde{p}^{H\times h}$ with the discrete reduced solution $\tilde{p}_{m}^{H}$, computed using the lifting function 
\begin{equation}\label{lifting_gd2}
g_{D} = (-0.5x+1) \mathfrak{h}(0,y) +0.5x \mathfrak{h}(2,y)
\end{equation}
for $m = 5,10,15,20,30$, $N_{H}=800$, $n_{h}=440$ and $N_{H'} = 80$ (Fig.~\ref{figsource_loesung}, left), we observe that $30$ basis functions are required to obtain an acceptable approximation of $\tilde{p}^{H\times h}$. The oscillations are even stronger than in the previous example and result in a very slow approximation of $p^{H \times h}$. Note that even for $m=15$ no resemblance between the contour lines of $\tilde{p}_{15}^{H}$ and $\tilde{p}^{H\times h}$ can be detected. In contrast already $\tilde{p}_{15}^{H}$ --- the solution of \eqref{prob_gd} for $m=15$ --- contains all essential features of the solution including the two small peaks around $x =1$ and we see a very good visual agreement of  $\tilde{p}_{20}^{H}$ and  $\tilde{p}^{H\times h}$. Altogether we observe a very much better qualitative convergence behavior, when choosing $\mathfrak{h}$ instead of $g_{D}$ as the lifting function, where the difference is even larger than in the previous test case. 

\begin{figure}[t]
\centering
\subfloat[{\footnotesize lift. func. $g_{D}$ \eqref{lifting_gd2}}]{
\includegraphics[scale = 0.23]{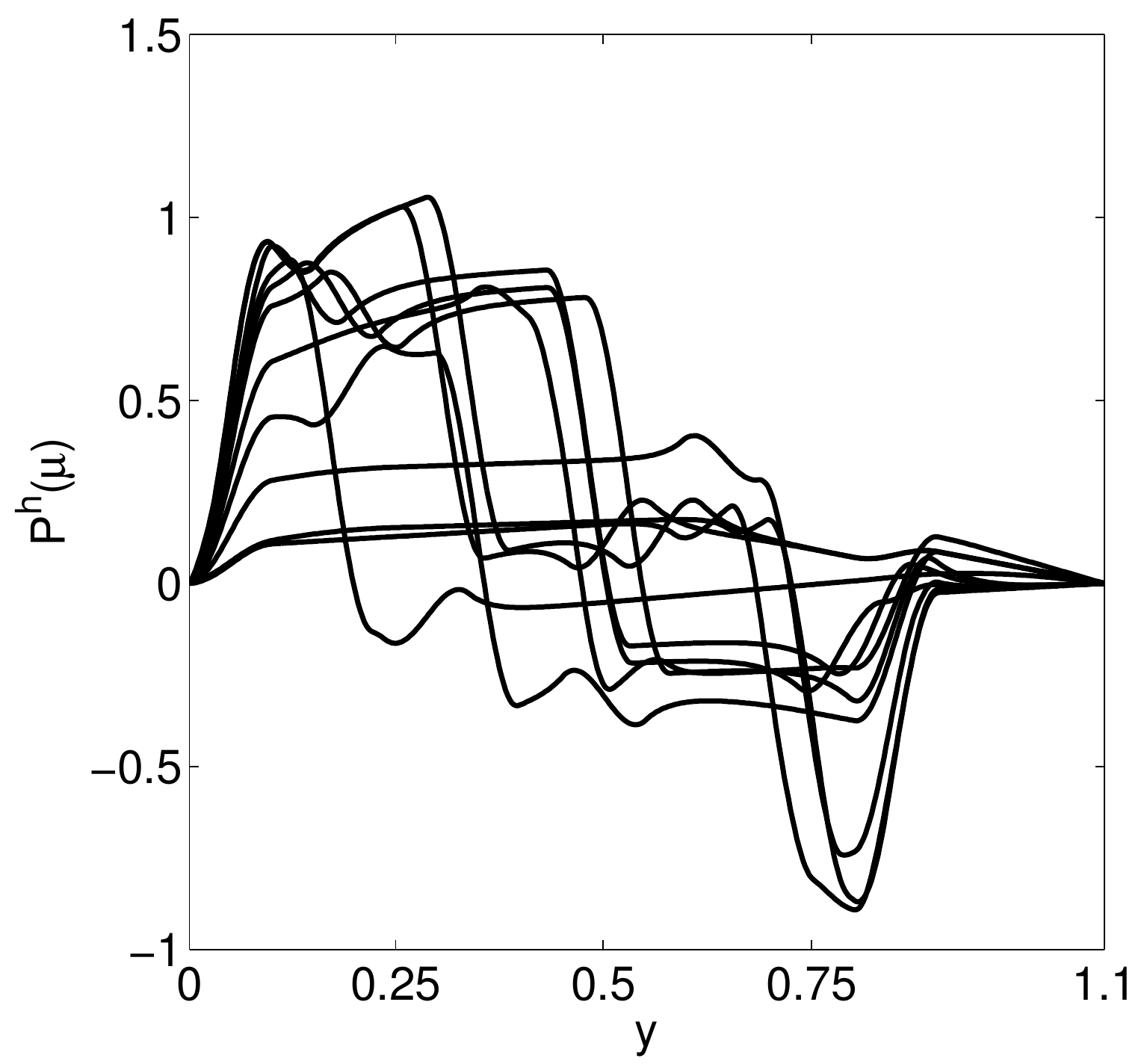}\label{figmani_sourcec}}
\subfloat[{\footnotesize incl. $\Delta \mathfrak{h}$}]{
\includegraphics[scale = 0.23]{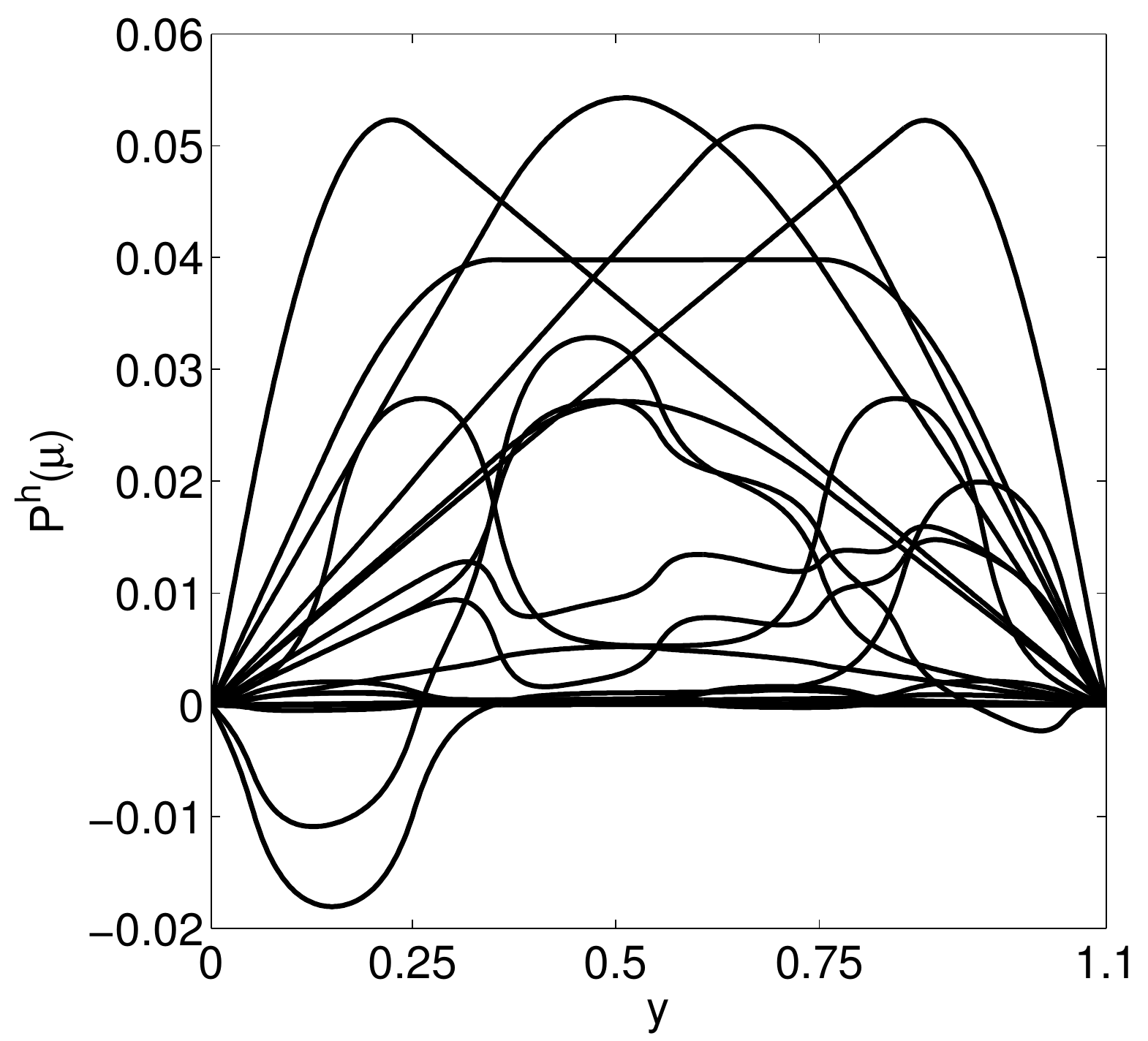}\label{figmani_sourcea}}
\subfloat[{\footnotesize \eqref{1D_prob_quad_para_gd}, $\bar{Q}=2$} ]
{\includegraphics[scale = 0.23]{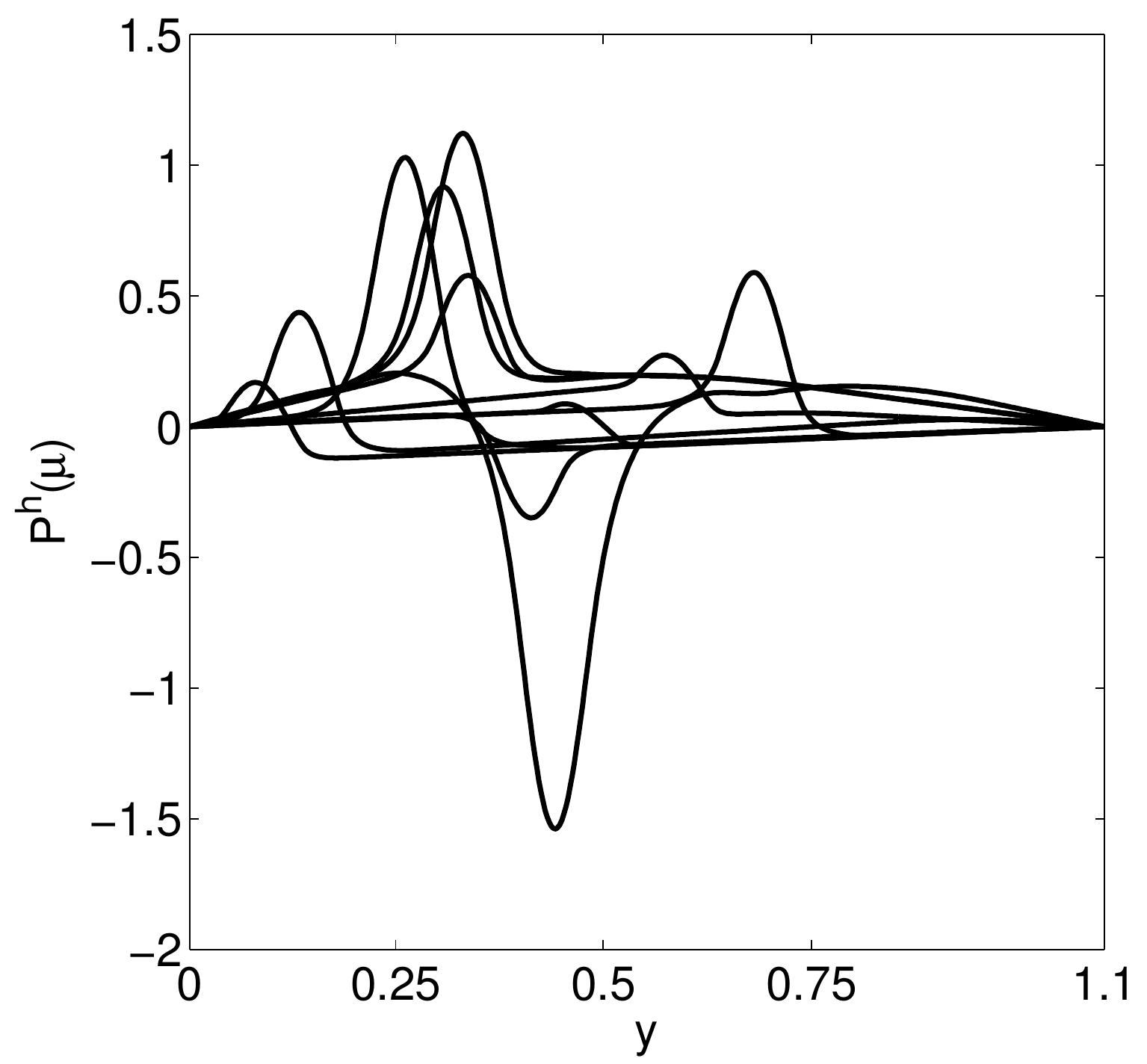}\label{figmani_sourceb}}
\subfloat[{\footnotesize \eqref{1D_prob_quad_para_gd}, $\bar{Q}=1$}]{
\includegraphics[scale = 0.23]{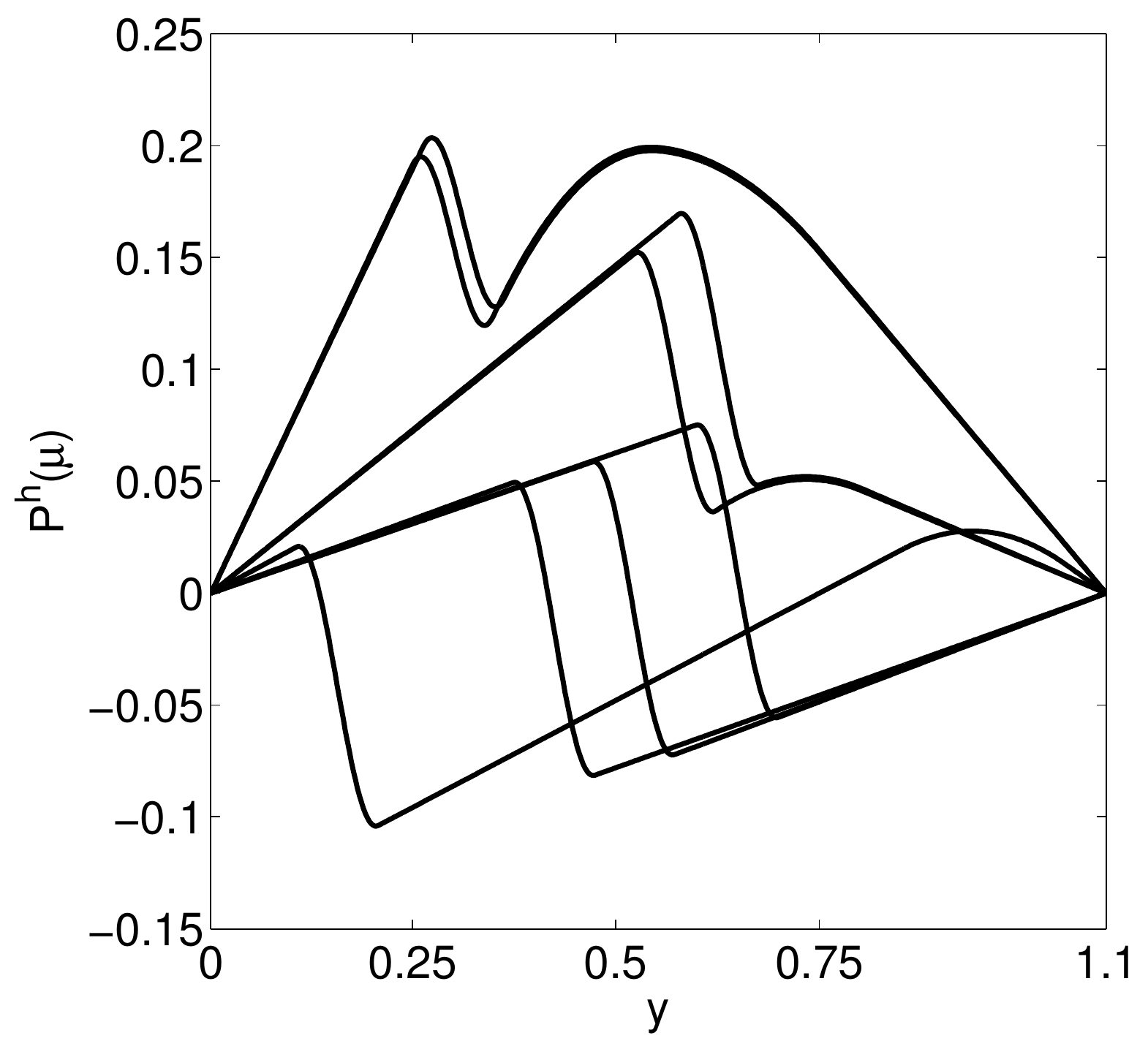}
 \label{figmani_sourced}}
\caption{Test case 2: Typical snapshots of the discrete solution manifold $\mathcal{M}_{\Xi}$ \eqref{disc_mani} if the lifting function $g_{D}$ \eqref{lifting_gd2} is used and therefore no information on the interface is included in the model reduction procedure (a), if $\Delta \mathfrak{h}$ is added to $F$ in \eqref{red_prob_gd}  (b), \eqref{1D_prob_quad_para_gd} is solved with $\bar{Q}=2$ (c) and $\bar{Q} = 1$ (d).\label{figmani_source}}
\end{figure}

Studying the convergence behavior of $\|e_{m}\|_{V}^{rel}$ we observe a bigger improvement of the convergence behavior for $m \leq 60$ for $H = 0.003125$ when using the lifting function $\mathfrak{h}$ instead of $g_{D}$, in comparison to the previous example (Fig.~\ref{fig46a}, Fig.~\ref{fig46b}, Fig.~\ref{fig44b}). For $m>60$ the convergence rate is comparable. Adding $\Delta \mathfrak{h}$ to the source term $F$ in \eqref{red_prob_gd} and omitting the term $-a(\mathfrak{h},v_{m})$ instead yields an exponential convergence rate of $\|e_{m}\|_{V}^{rel}$ (Fig.~\ref{fig46c}). The behavior of $\|e_{m}\|_{V}^{rel}$ in both cases --- either including $\mathfrak{h}$ in the weak (Fig.~\ref{fig46b}) or in the strong formulation (Fig.~\ref{fig46c}) ---
demonstrates the capacity of the ansatz proposed in \S \ref{ansatz_interface} to improve the convergence behavior for the RB-HMR approach. 
%As however the convergence rates of $e_{m}^{\text{{\tiny POD }}}$ and $\|e_{m}\|_{L^{2}(\Omega)}^{rel}$ and $\lambda_{m}$ and $\|\bar{p}_{m}\|_{L^{2}(\Omega_{1D})}$, respectively, for the solution of \eqref{prob_gd} are roughly the same (Fig.~\ref{fig46d}, \ref{fig46e}) and differ from the ones obtained if including $\Delta \mathfrak{h}$ in the strong formulation of \eqref{red_prob_gd}, we conclude that also for this test case the equivalence of including $\mathfrak{h}$ in the strong or weak formulation of \eqref{red_prob_gd} cannot be reproduced by the parametrized 1D problem, prohibiting an exponentially fast approximation of the reference solution $\tilde{p}^{H\times h}$ of \eqref{truth_prob_gd} by $\tilde{p}_{m}^{H}$ --- the solution of \eqref{prob_gd}. Finally, we observe again a significant deterioration of the convergence rate of $\|e_{m}\|_{V}^{rel}$ for the solution of \eqref{prob_gd} for one quadrature point in comparison to $\bar{Q}=2$ (Fig.~\ref{fig46f}). \\

%The fact that the convergence behavior of $e_{m}^{\text{{\tiny POD}}}$ for the present and the previous test case coincides (Fig.~\ref{fig44d},\ref{fig46d}), suggests that the interface is the decisive factor in the approximation behavior of the discrete reduced solution of \eqref{prob_gd}. 
In Fig.~\ref{figmani_source} some exemplary snapshots $\mathcal{P}^{h}(\mu)$ are depicted in order to help to explain the convergence behavior observed in Fig.~\ref{fig46}. In case we add $\Delta \mathfrak{h}$ to $F$, the snapshots look like the profile of $p^{H \times h}$ in the transverse direction (Fig.\ref{figmani_sourcea}), whereas the solutions of \eqref{1D_prob_quad_para_gd} for $\bar{Q}=2$ additionally exhibit a peak around the interface (Fig.\ref{figmani_sourceb}). If we use $g_{D}$ as the lifting function, the snapshots $\mathcal{P}^{h}(\mu)$ feature very strong oscillations but barely resemble the slope of $p^{H\times h}$ (Fig.\ref{figmani_sourcec}).  Also the snapshots for $\bar{Q}=1$ in \eqref{1D_prob_quad_para_gd} contain only few information on the profile of $p^{H\times h}$ in the transverse direction (Fig.\ref{figmani_sourced}), which yields a significant deterioration of the convergence rate of $\|e_{m}\|_{V}^{rel}$. This shows again that the artificial coupling introduced in \S \ref{1dproblem_bilFEM} is necessary to obtain a good approximation behavior of the RB-HMR approach. 
\begin{figure}[t]
\center
\includegraphics[scale=0.55]{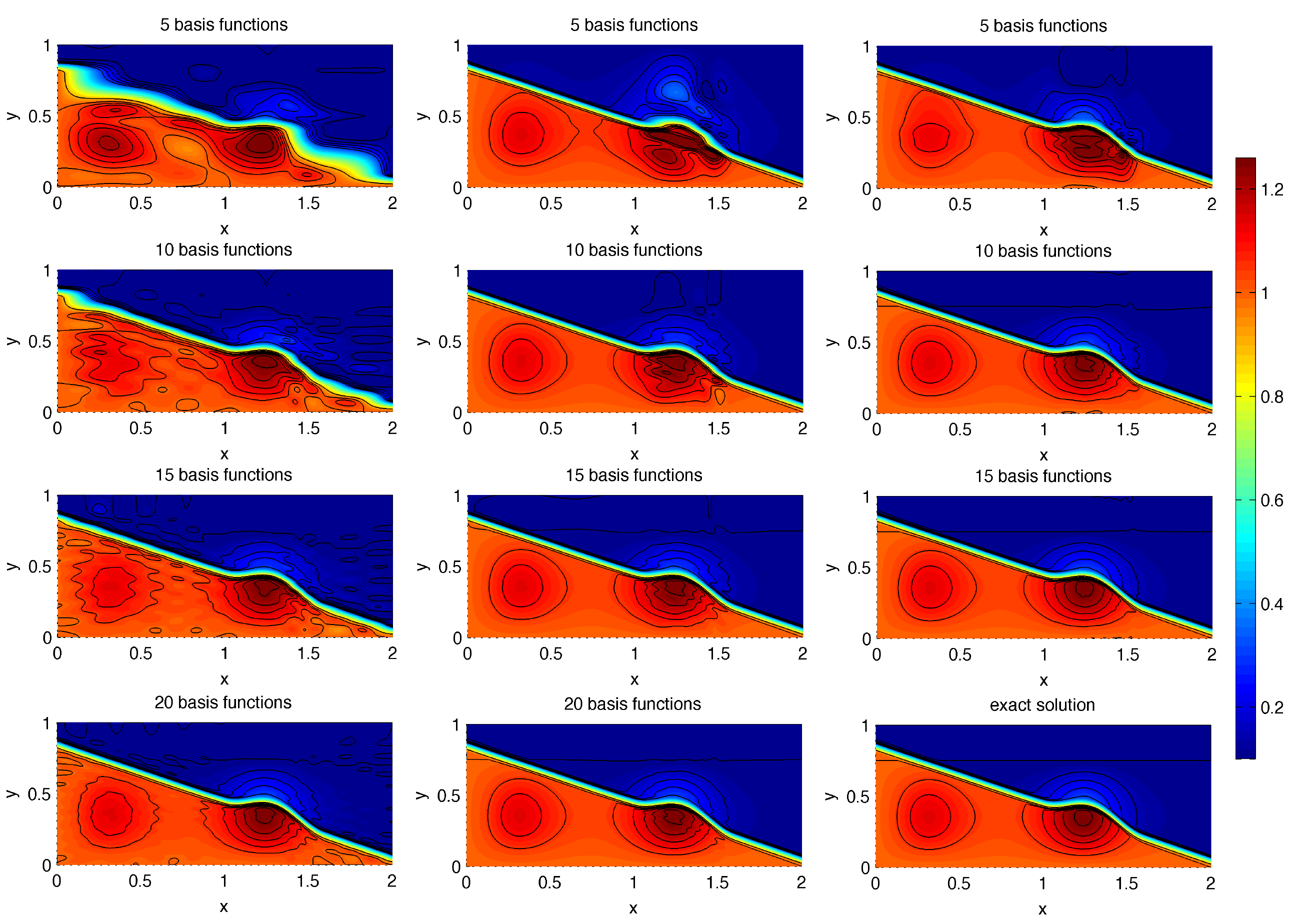} 
\caption{Test case 3: Comparison of the discrete reduced solution $\tilde{p}_{m}^{H}$ using the lifting function $g_{D}$  (left side) for $m =5,10,15,20$ (top-bottom), $\tilde{p}_{m}^{H}$ using $\mathfrak{h}$ \eqref{test3_interface_approx} as the lifting function in \eqref{prob_gd}  and $m =5,10,15,20$ (middle), $\tilde{p}_{m}^{H}$ if adding $\Delta \mathfrak{h}$ to $F$ in \eqref{prob_gd} and neglecting the term $-a(\mathfrak{h},\xi_{i}^{H}\phi_{l})$ instead for $m =5,10,15$ (right)  with the exact solution $\tilde{p}$ (right, bottom) for $N_{H}=800$, $n_{h}=400$, $N_{H'}=80$.\label{figwelle_gest_loesung}}
\end{figure}

\paragraph*{Test case 3}
\begin{figure}[t]
\centering
\subfloat[{\footnotesize lifting function $g_{D}$} ]
{\includegraphics[scale = 0.31]{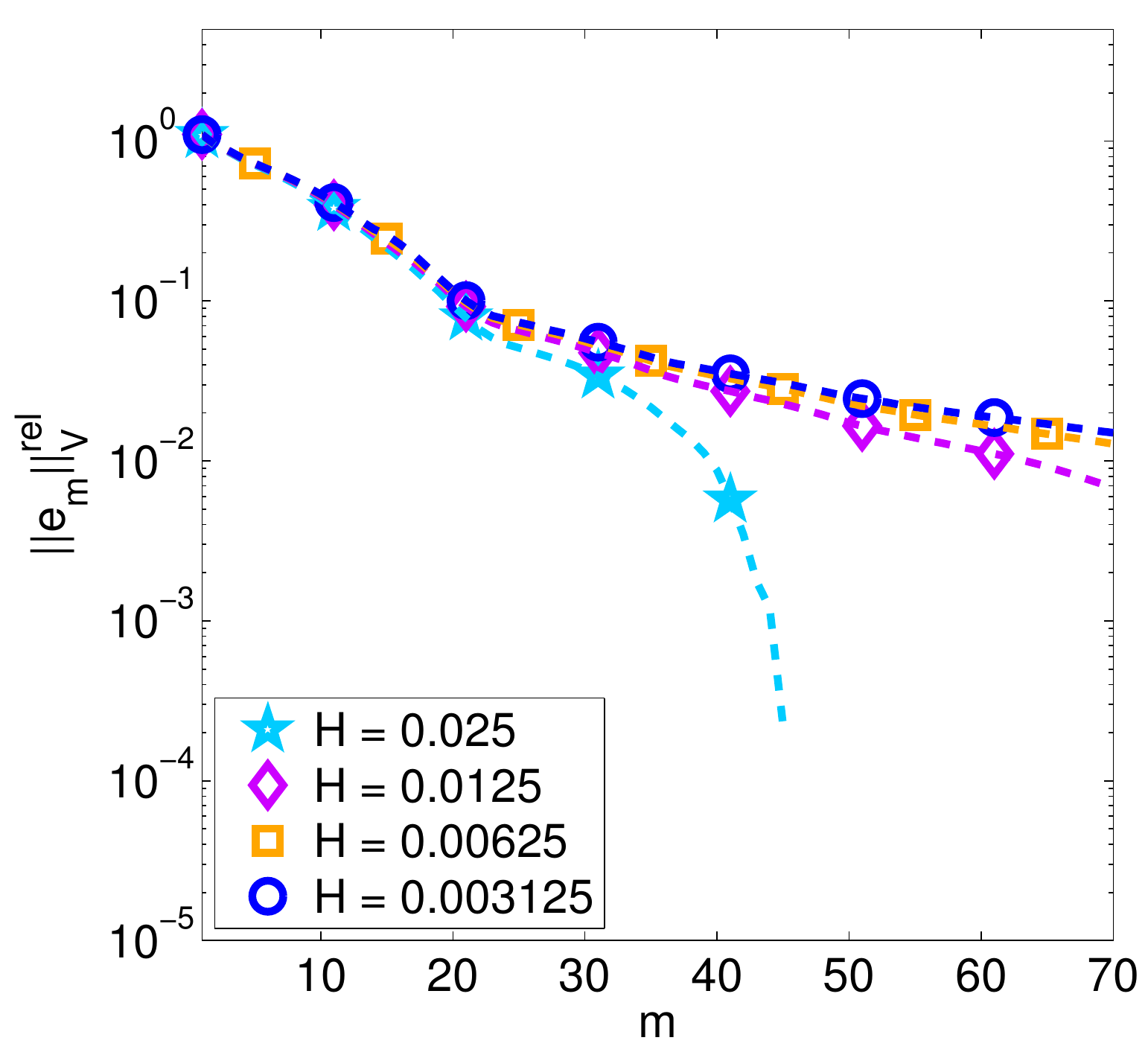}\label{fig48a}}
\subfloat[{\footnotesize solution of \eqref{prob_gd}, $\bar{Q}=2$}]{
\includegraphics[scale = 0.31]{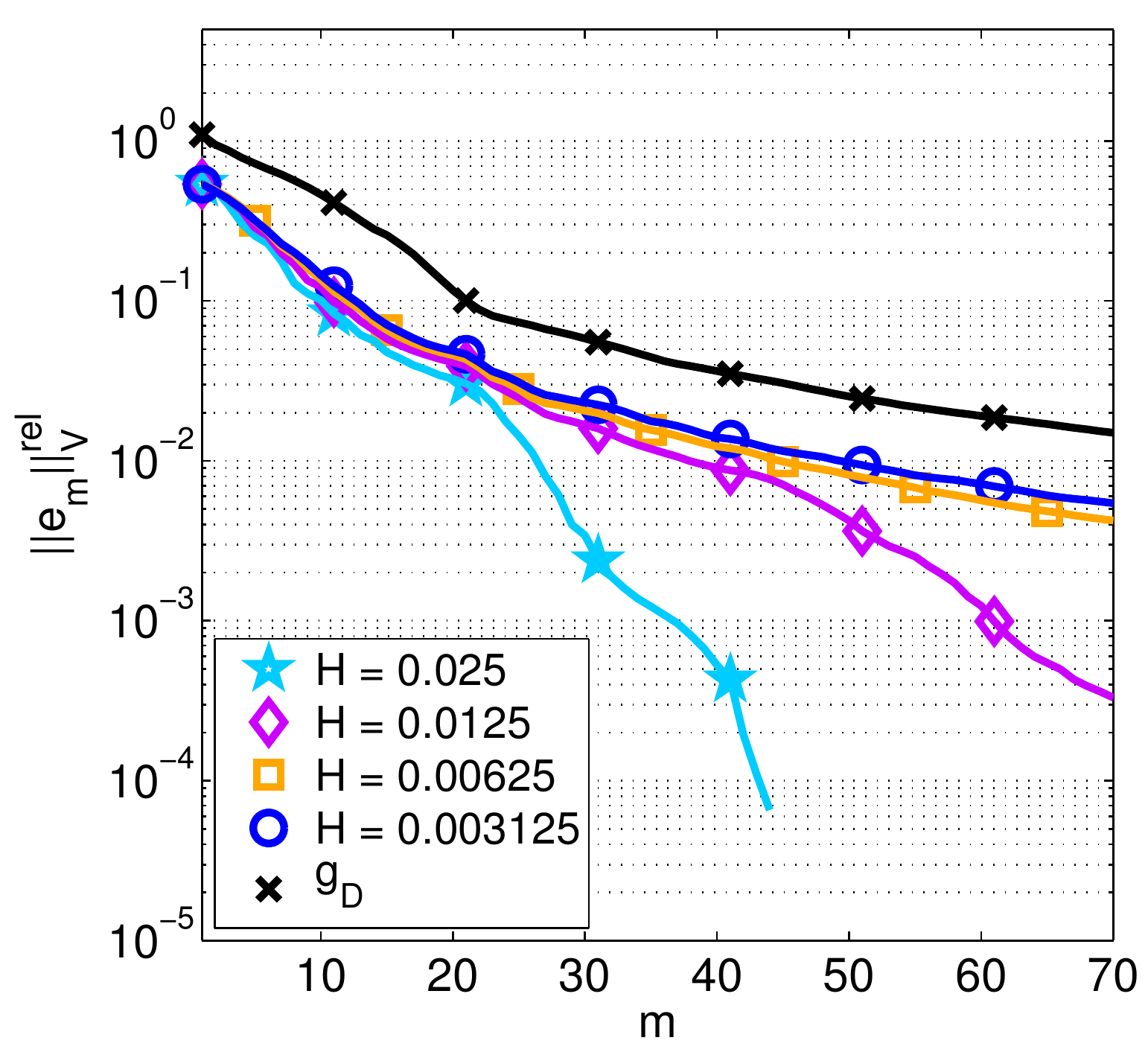}\label{fig48b}}
\subfloat[{\footnotesize  including $\Delta \mathfrak{h}$}]{
\includegraphics[scale = 0.31]{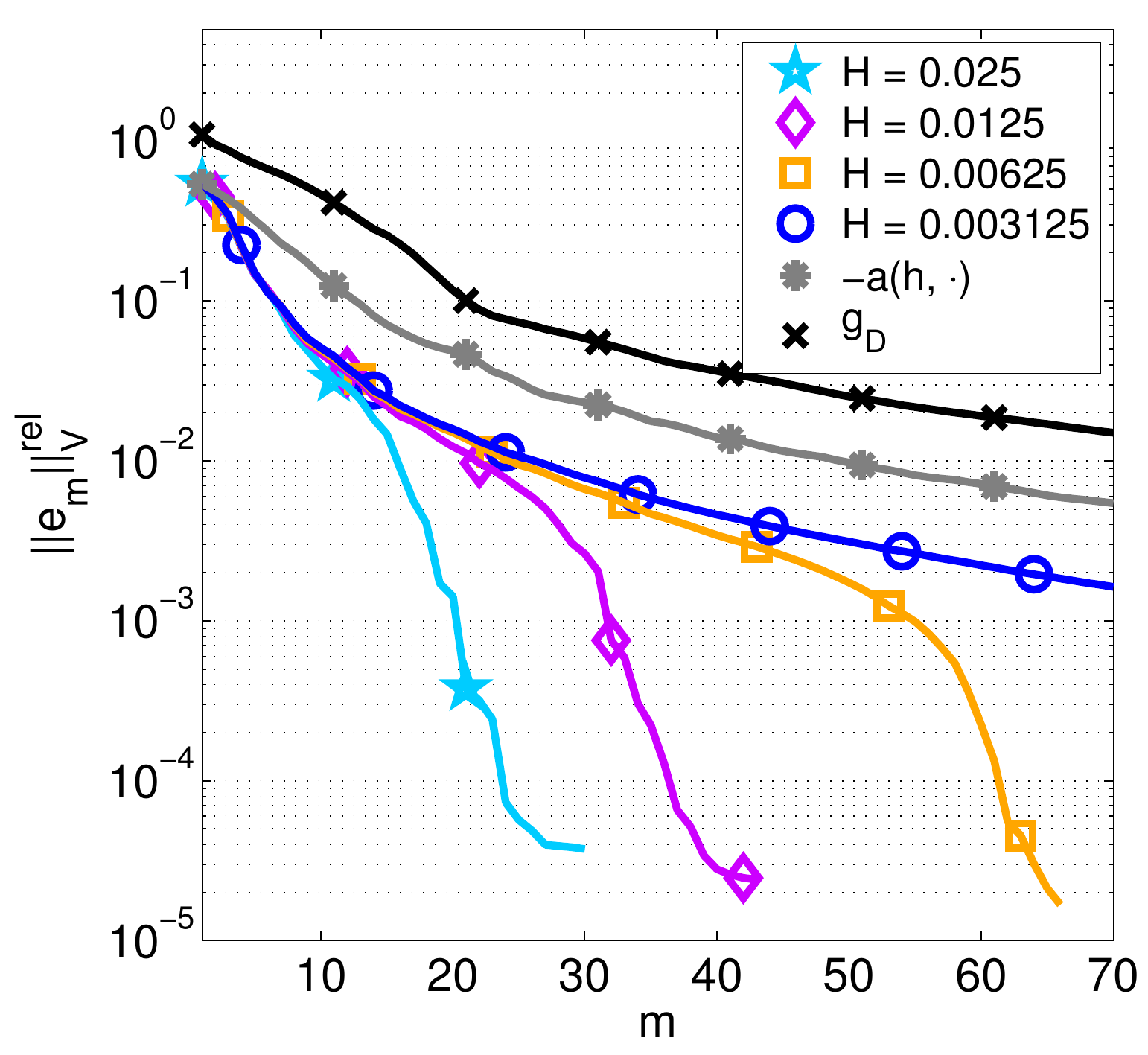}\label{fig48c}}
\caption{Test case 3:  Comparison of the model error convergence
of $\|e_{m}\|_{V}^{rel}$. (a): We use the lifting function $g_{D}$ defined in \eqref{lifting_gd}. (b): We consider \eqref{prob_gd} and choose $\bar{Q} = 2$ in \eqref{1D_prob_quad_para_gd}. (c): We consider the right hand side $F+\Delta \mathfrak{h}$ in \eqref{prob_gd} and neglect the term $-a(\mathfrak{h},\xi_{i}^{H}\phi_{l})$ instead. All plots: $N_{H'}=10$.\label{fig481}}
\vspace*{-10pt}
\end{figure}

\begin{figure}[t]
\centering
\subfloat[{\footnotesize $\|e_{m}\|_{L^{2}(\Omega)}^{rel}$ and $ e_{m}^{\mbox{{\tiny POD}}}$}]{
\includegraphics[scale = 0.31]{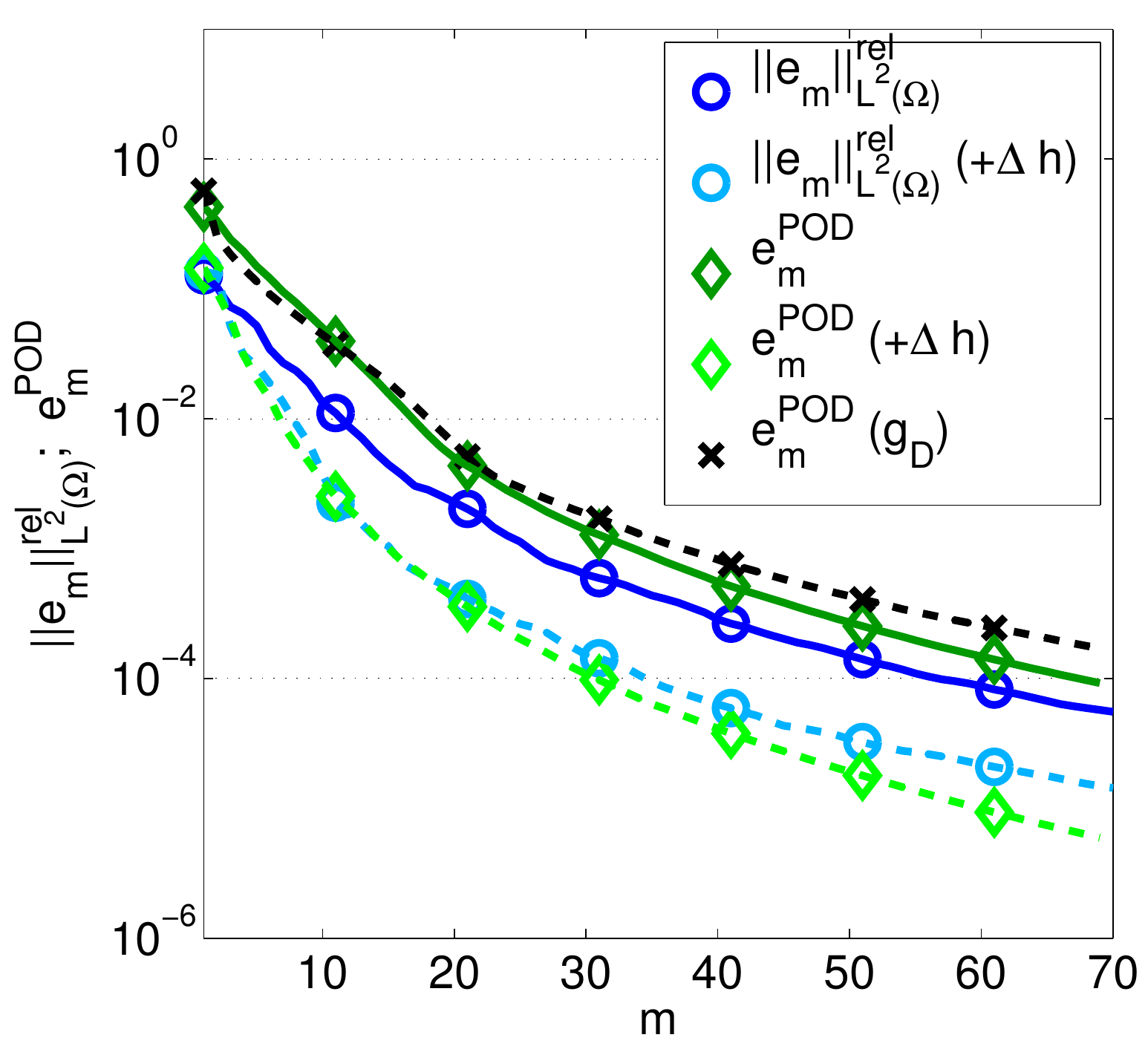} \label{fig48d}}
\subfloat[{\footnotesize $\lambda_{m}$ and $\|\bar{p}_{m}^{H}\|_{L^{2}(\Omega_{1D})}$}]{
\includegraphics[scale = 0.31]{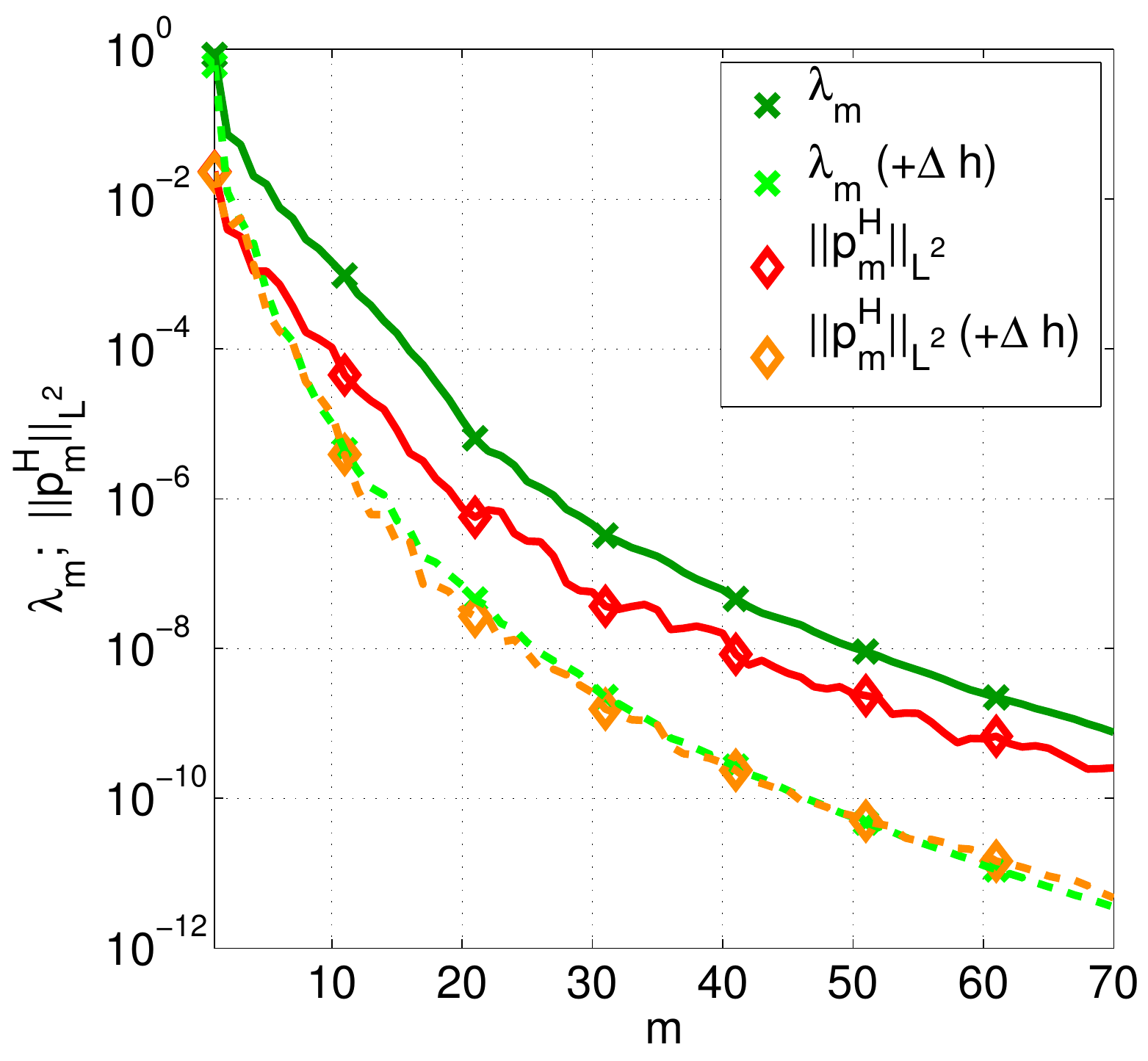}\label{fig48e}}
\subfloat[{\footnotesize $\|e_{m}\|_{V}^{rel}$, solution of \eqref{prob_gd}, $\bar{Q}=1$}]{
\includegraphics[scale = 0.31]{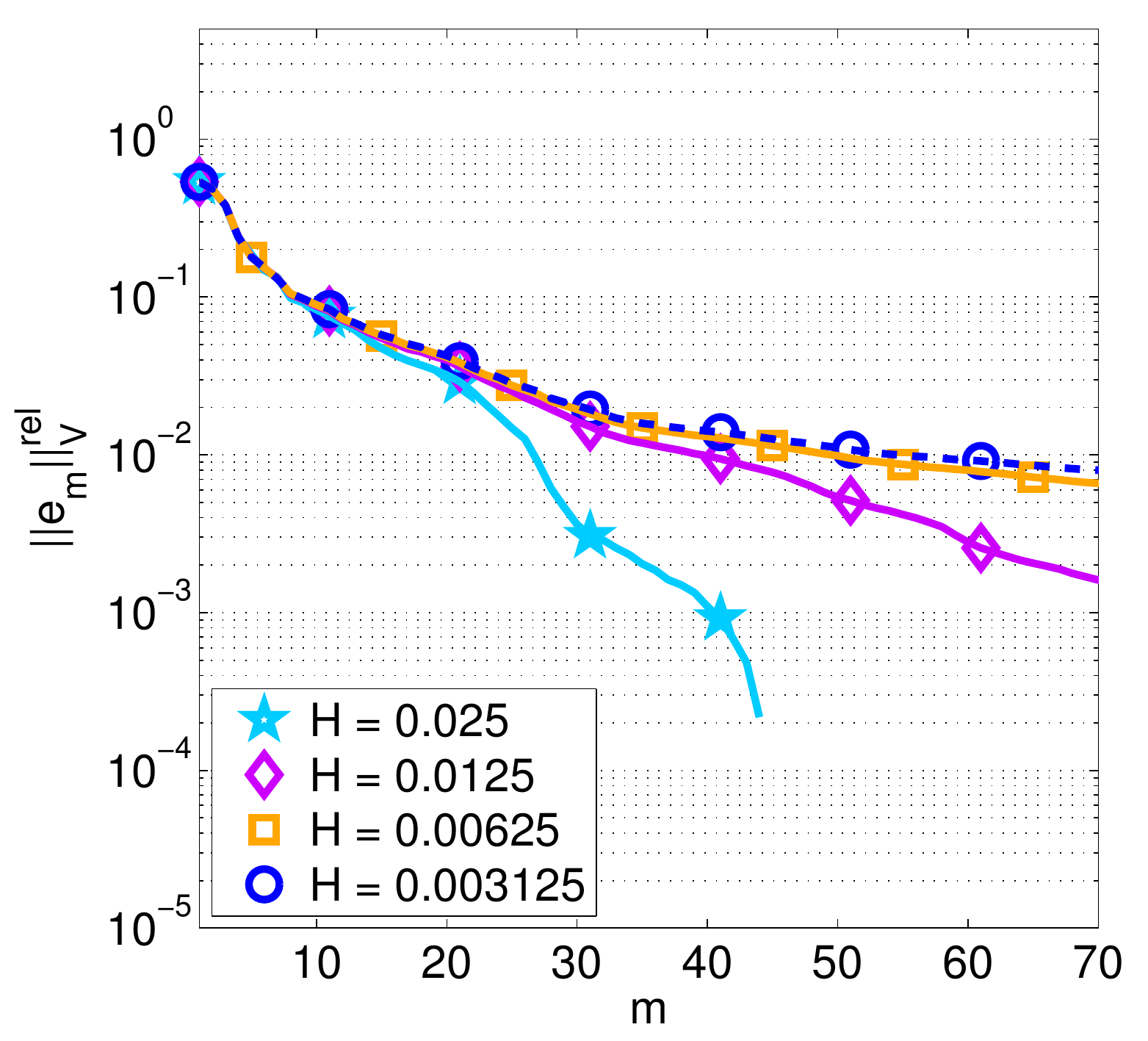} \label{fig48f}}
\caption{Test case 3: (a): Comparison of $\|e_{m}\|_{L^{2}(\Omega)}^{rel}$ and $e_{m}^{\mbox{{\tiny POD}}}$ for the following situations and legend entries: $e_{m}^{\mbox{{\tiny POD}}}$ or $\|e_{m}\|_{L^{2}(\Omega)}^{rel}$: We consider the discrete reduced problem \eqref{prob_gd}. $e_{m}^{\mbox{{\tiny POD}}} (+ \Delta \mathfrak{h})$ or $\|e_{m}\|_{L^{2}(\Omega)}^{rel} (+ \Delta \mathfrak{h})$: We consider the right hand side $F+\Delta \mathfrak{h}$ in \eqref{prob_gd} and neglect the term $-a(\mathfrak{h},\xi_{i}^{H}\phi_{l})$ instead. $e_{m}^{\mbox{{\tiny POD}}} (g_{D})$: We use the lifting function $g_{D}$ defined in \eqref{lifting_gd}. (b) Comparison of $\lambda_{m}$ and $\|\bar{p}_{m}^{H}\|_{L^{2}(\Omega_{1D})}^{2}$ for the solution of \eqref{prob_gd} or for the solution of \eqref{prob_gd} with a right hand side $F+\Delta \mathfrak{h}$ and no term $-a(\mathfrak{h},\xi_{i}^{H}\phi_{l})$ (legend entry contains additionally $(+ \Delta \mathfrak{h})$).  (c): Comparison of $\|e_{m}\|_{V}^{rel}$ for \eqref{prob_gd} and $\bar{Q} = 1$ in \eqref{1D_prob_quad_para_gd}. All plots: $N_{H'}=10$. \label{fig48}}
\vspace*{-10pt}
\end{figure}

Finally, we consider a test case in which we do not include the exact interface $\bar{\mathfrak{h}}$ but only an approximation $\mathfrak{h}$ as we usually cannot expect to be able to determine the exact interface but rather only an approximation. We solve a Poisson problem on $\Omega = (0,2) \times (0,1)$, where we choose the exact solution $\tilde{p} = p + \bar{\mathfrak{h}}$ where
\begin{align*}
\bar{\mathfrak{h}}(x,y) &= \begin{cases}
                        1 \quad &\text{if} \enspace y + 0.4x -\mathfrak{j}(x)< 0.8,\\
                        0.1 \quad & \text{if} \enspace y + 0.4x-\mathfrak{j}(x) > 0.9,\\
                        0.55 + 0.45\cos(10\pi (y+0.4x-0.8-\mathfrak{j}(x))) &\text{if} \enspace 0.8 \leq y+0.4x-\mathfrak{j}(x)\leq 0.9
                        \end{cases}\\
\nonumber \text{for} \enspace  &\mathfrak{j}(x) = \begin{cases}  
                        0.1\sin ^{2}(\frac{5\pi}{3} (x-1)) \quad &\text{if} \enspace 1\leq x \leq 1.6,\\
                        0  \quad &\text{else},
                         \end{cases}\\
\text{and} \enspace  &p(x,y) = 5y^{2}(1-y)^{2}(0.75-y)x(2-x)\exp(\sin(2\pi x)), \qquad \qquad
\end{align*}
which is the same choice of $p$ as in test case 1. The exact solution $\tilde{p}$ is depicted in the last picture of Fig.~\ref{figwelle_gest_loesung} and it can be seen that $\bar{\mathfrak{h}}$ is curved between $x =1$ and $x =1.6$. In Fig.~\ref{figwelle_gest_loesung} on the left we see the discrete reduced solution $\tilde{p}_{m}^{H}$, computed with the  same lifting function 
$
g_{D} = (-0.5x+1) \bar{\mathfrak{h}}(0,y) +0.5x \bar{\mathfrak{h}}(2,y)
$
as in test case 1 for $m = 5,10,15,20$, $N_{H}= 800$, $n_{h} =400$ and $N_{H'} = 80$  and observe a qualitative convergence behavior very similar to the one in test case 1 (Fig.~\ref{figwelle_loesung}, left). Next, we compare the exact solution $\tilde{p}$ with the discrete reduced solution $\tilde{p}_{m}^{H}$, which has been computed using the approximating interface
\begin{align}
\label{test3_interface_approx} \mathfrak{h}(x,y) &= \begin{cases}
                        1 \quad &\text{if} \enspace y + 0.4x < 0.8,\\
                        0.1 \quad & \text{if} \enspace y + 0.4x > 0.9,\\
                        0.55 + 0.45\cos(10\pi (y+0.4x-0.8)) &\text{if} \enspace 0.8 \leq y+0.4x\leq 0.9
                        \end{cases}
\end{align}
for $m = 5,10,15,20$, $N_{H}= 800$, $n_{h} =400$ and $N_{H'} = 80$ (Fig.~\ref{figwelle_gest_loesung}, middle). We observe that the bending of the actual interface $\bar{\mathfrak{h}}$, which is not included in the approximation $\mathfrak{h}$, causes oscillations in the interval $[0.9,1.6] \subset \Omega_{1D}$ yielding a slower approximation behavior of $\tilde{p}_{m}^{H}$ than in test case 1 (Fig.~\ref{figwelle_loesung}, right). Eventually, $20$ basis functions are required to obtain a good approximation of $\tilde{p}$. Overall we still observe a significant improvement of the qualitative convergence behavior if including an approximation of the interface, but, as expected, the gain is not as large as if including the exact interface. If we consider the right-hand side $F +\Delta \bar{\mathfrak{h}}$ in \eqref{prob_gd}, omit the term $-a(\mathfrak{h},\xi_{i}^{H}\phi_{l})$, and compute the associated discrete reduced solution for $m = 5,10,15$, $N_{H}= 800$, $n_{h} =400$ and $N_{H'} = 80$ (Fig.~\ref{figwelle_gest_loesung}, right), we detect only few oscillations and see already for $\tilde{p}_{15}^{H}$ a good visual agreement with the exact solution $\tilde{p}$. 
\begin{figure}[t]
\centering
\subfloat[{\footnotesize lift. func. $g_{D}$ \eqref{lifting_gd2}}]{
\includegraphics[scale = 0.23]{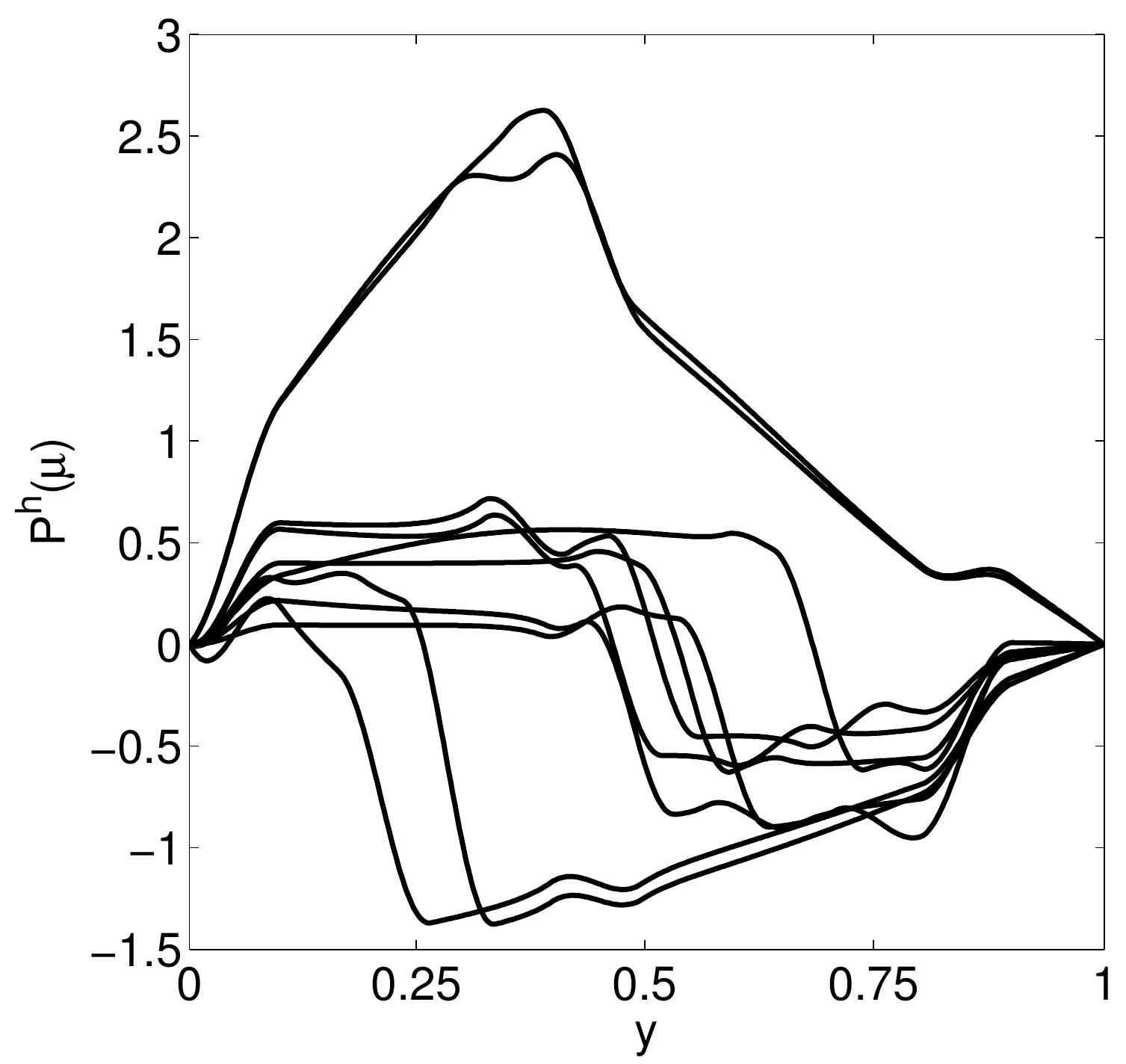}\label{figmani_welle_gestc}}
\subfloat[{\footnotesize incl. $\Delta \mathfrak{h}$}]{
\includegraphics[scale = 0.23]{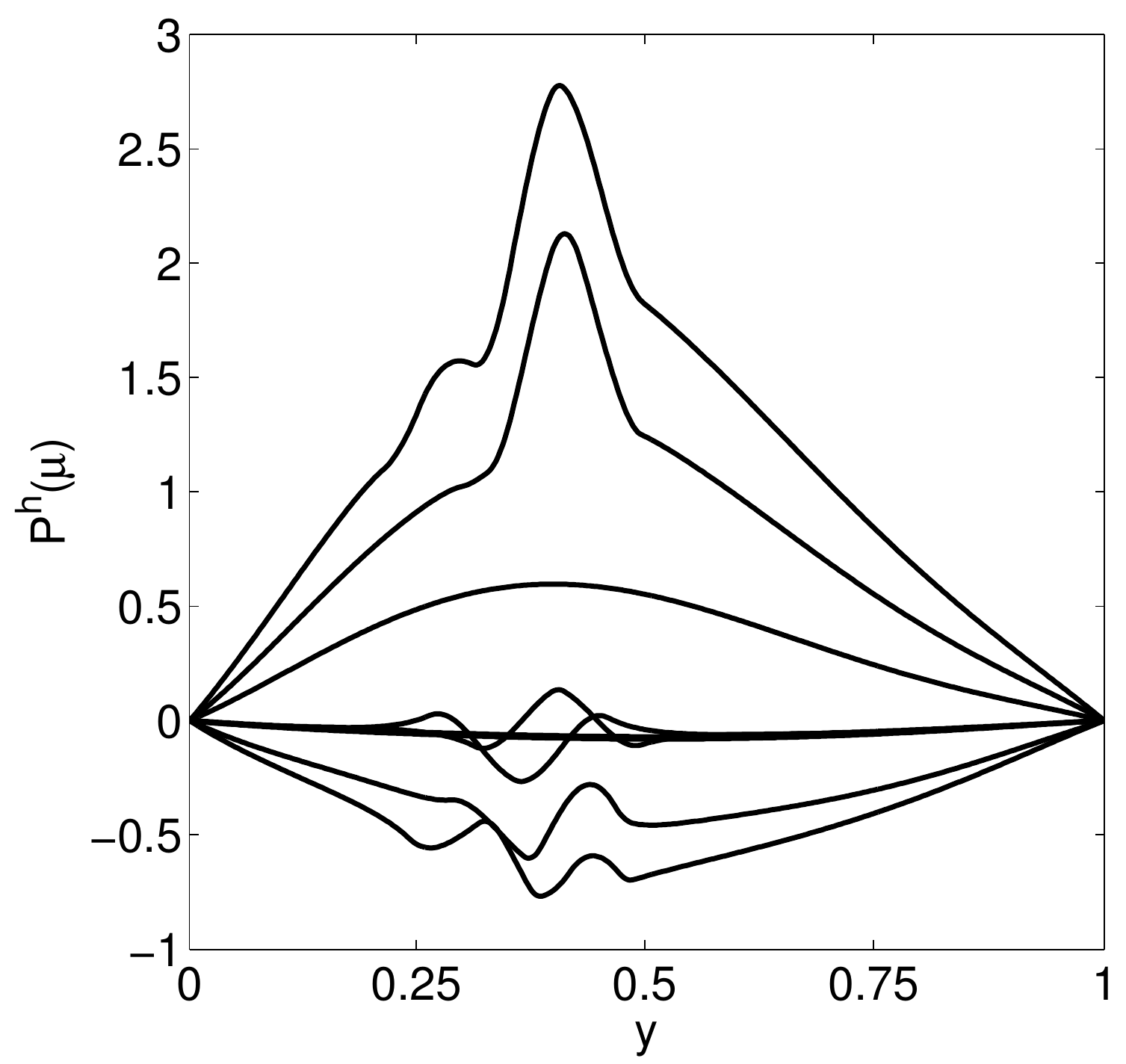}\label{figmani_welle_gesta}}
\subfloat[{\footnotesize \eqref{1D_prob_quad_para_gd}, $\bar{Q}=2$} ]
{\includegraphics[scale = 0.23]{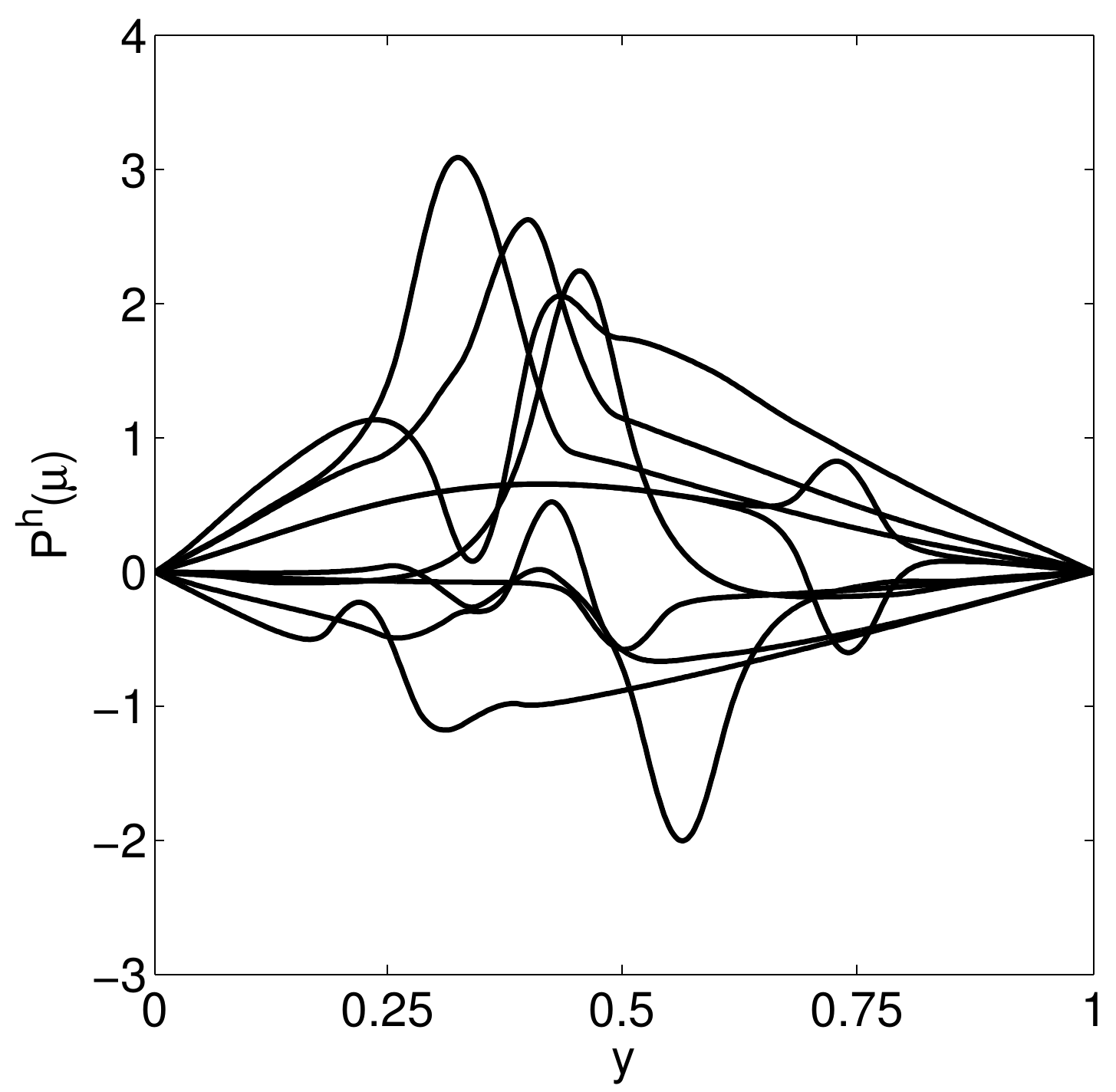}\label{figmani_welle_gestb}}
\subfloat[{\footnotesize \eqref{1D_prob_quad_para_gd}, $\bar{Q}=1$}]{
\includegraphics[scale = 0.23]{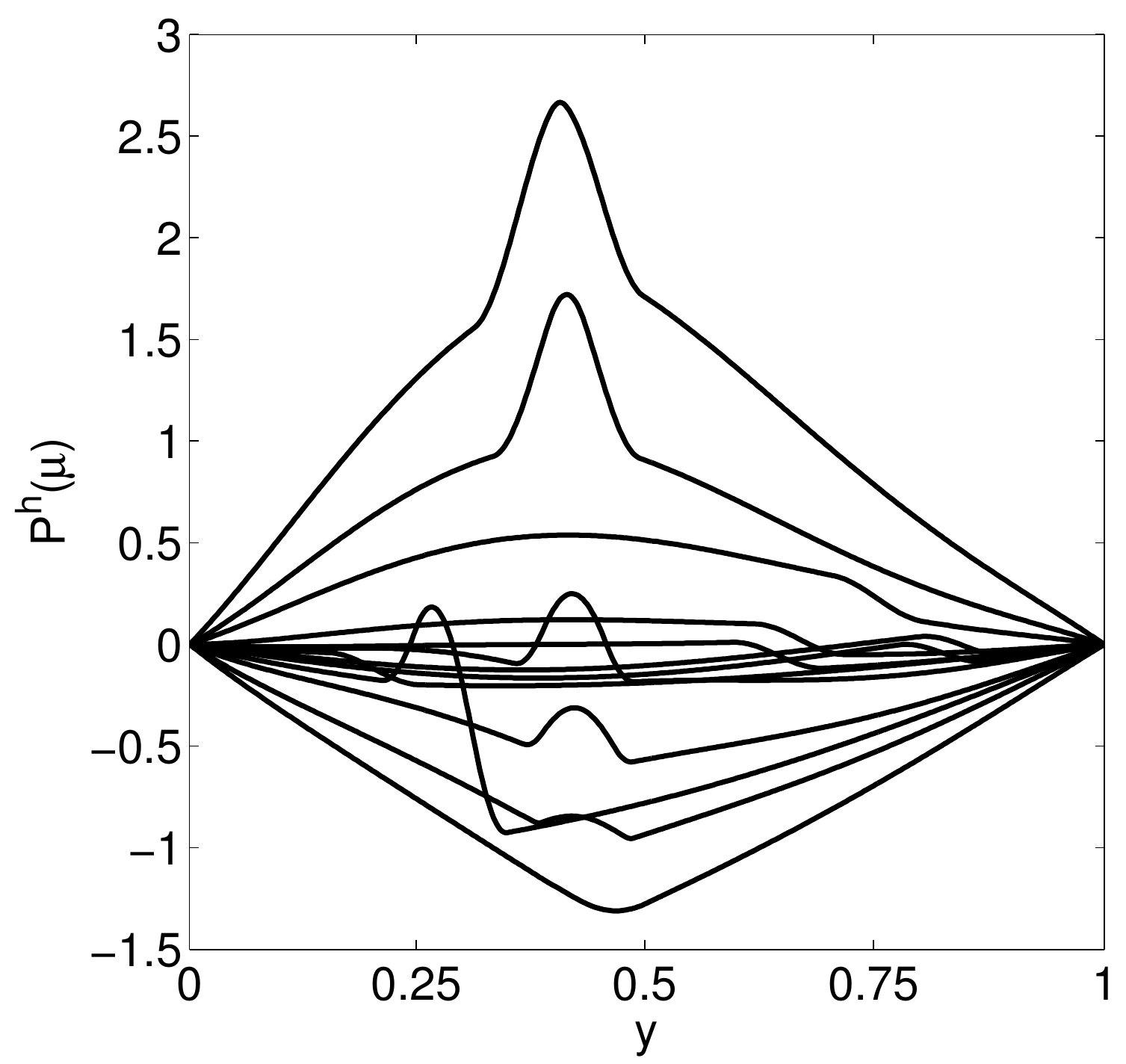} \label{figmani_welle_gestd}}
\caption{Test case 3: Typical snapshots of the discrete solution manifold $\mathcal{M}_{\Xi}$ \eqref{disc_mani} if the lifting function $g_{D}$ \eqref{lifting_gd2} is used and therefore no information on the interface is included in the model reduction procedure (a), if $\Delta \mathfrak{h}$ is added in the strong formulation of \eqref{red_prob_gd}  (b), \eqref{1D_prob_quad_para_gd} is solved with $\bar{Q}=2$ (c) and $\bar{Q} = 1$ (d).\label{figmani_welle_gest}}
\end{figure}

Analyzing the convergence behavior of the relative model error $\|e_{m}\|_{V}^{rel}$, we see in Fig.~\ref{fig48b} that on the one hand solving $\eqref{prob_gd}$ with the approximating interface $\mathfrak{h}$ enhances the convergence behavior but on the other hand the improvement is much smaller than in the two previous examples. Adding $\Delta \mathfrak{h}$ to $F$ in \eqref{prob_gd} and omitting the term $-a(\mathfrak{h},\xi_{i}^{H}\phi_{l})$ instead further improves both the convergence behavior and rate of $\|e_{m}\|_{V}^{rel}$ (Fig.~\ref{fig48c})  and $e_{m}^{\text{{\tiny POD}}}$ (Fig.~\ref{fig48d}) but not to the same extend as in the test cases 1 and 2. This supports the findings gained in the qualitative convergence analysis. As the convergence rate of $e_{m}^{\text{{\tiny POD}}}$ is better than the one of $\|e_{m}\|_{L^{2}(\Omega)}^{rel}$ if we include $\Delta \mathfrak{h}$, the convergence behavior of the model error may perhaps be further improved for instance by increasing the sample size. Finally, we observe in Fig.~\ref{fig48f} that the convergence behavior of $\|e_{m}\|_{V}^{rel}$ for choosing $\bar{Q}=1$ in \eqref{1D_prob_quad_para_gd} is only slightly worse than the one for $\bar{Q}=2$. This can be explained by the fact that the snapshots of \eqref{1D_prob_quad_para_gd} for $\bar{Q} =1$ (Fig.~\ref{figmani_welle_gestd}) resemble very much both the snapshots of \eqref{1D_prob_quad_para_gd} for $\bar{Q}=2$ (Fig.~\ref{figmani_welle_gestb}) and the snapshots $\mathcal{P}^{h}(\mu)$ if including $\Delta \mathfrak{h}$ (Fig.~\ref{figmani_welle_gesta}). Also for this test case the snapshots computed using the lifting function $g_{D}$ contain hardly any information on the profile of $p$ (Fig.~\ref{figmani_welle_gestc}), prohibiting a fast convergence to the exact solution. \\

We conclude the numerical experiments with some remarks on the reconstruction procedure. As computing $\mathcal{R}^{H\times h}$ with \eqref{Riesz_gd} is as expensive as solving the reference solution, this strategy is not feasible from a computationally viewpoint. We therefore propose to use an ansatz, which has some similarities with an oversampling strategy (see for instance \cite{HenPet13} and references therein), that is employed in the field of Multi-scale Methods. Precisely, when deriving the parametrized dimensionally reduced problem \eqref{1D_prob_quad_para_gd} for a parameter value $\mu \in \Xi_{train}$, we first solve 
\begin{equation*}
\int_{(\mu-R,\mu + R)\times \widehat{\omega}} \mathcal{R}^{H\times h}_{\mu} v^{H\times h} = a(\nabla \mathfrak{h}, v^{H\times h})_{(\mu -R,\mu + R) \times \widehat{\omega}}.
\end{equation*}
$\mathcal{R}_{\mu}^{H\times h}(\mu,\cdot)$ can then be added to the source term in the parametrized dimensionally reduced problem \eqref{1D_prob_quad_para_gd}. Studying such type of reconstructions is a task for future research.

\section{Conclusion and Perspectives}\label{conclusions}

We have suggested a new ansatz for the approximation of PDE solutions with interfaces which are skewed with respect to the coordinate axes with tensor-based model reduction approaches. 
Using the example of subsurface flow, we have demonstrated how to compute the height of the water table by solving a dimensionally reduced problem, which has been derived by assuming a hydrostatic pressure distribution. For other applications one can proceed in the same manner, if a dimensionally reduced problem which locates the interface can be derived. This can in general be achieved by using an asymptotic expansion. For advection-diffusion problems we have outlined how the location of the interface can be inferred from data functions. By choosing the obtained solution profile as the lifting function of the Dirichlet boundary conditions, we hope to remove the part of the full solution which can be badly approximated by a tensor-based model reduction approach, and thus prevents a fast convergence of the latter.  Exemplifying the proposed ansatz for the RB-HMR approach we have derived a coupled parametrized lower dimensional problem starting from the reference FE approximation, where the parameter vector coincides with the quadrature points. We solve for the unknown parts of the solution in the dominant direction via the coupling. 

For advection-diffusion problems we have demonstrated in the numerical experiments that the described procedure yields a very good approximation of the location of the interface even in the presence of a strong advective field. Moreover, the numerical experiments demonstrate that the convergence behavior improves considerably also if we can locate the interface only approximately. The validation of the proposed derivation of the parametrized coupled 1D problem shows that the 1D problems do not reproduce properly the equivalence of including the interface in the strong or the weak formulation of the PDE. Here, we have obtained a worse approximation behavior for the inclusion in the weak form. A possible ansatz to solve this problem is to use a reconstruction of the derivative of the interface in the dominant direction which mimics its behavior in two space dimensions. Using the Riesz representative of the lifting function is one example for such a reconstruction. The numerical experiments demonstrated that by employing this reconstruction the features of the full solution apart from the interface can be approximated exponentially fast if these features allow for such a rate. Finally, the numerical tests show that the features of the full solution apart from the interface are approximated very slowly if no information on the interface is included. 
We therefore expect the proposed ansatz to be beneficial for instance when applied within the context of subsurface flow. \\

\parindent 0pt {\bf Acknowledgements:} The authors would like to thank Dr. Stephan Rave for fruitful discussions.



%  \bibliographystyle{elsarticle-num} 
 % \bibliography{diss.bib}

%\section*{References}

\end{document}